\documentclass[11pt,reqno]{article}
\usepackage{amssymb,amsmath}
\usepackage{amsfonts}
\usepackage{amsbsy}
\usepackage{amsopn} 
\usepackage{amsthm}
\usepackage{amscd}
\usepackage{verbatim}
\usepackage{graphicx}

\usepackage{color}

\numberwithin{equation}{section}
\newtheorem{theorem}{Theorem}[section]
\newtheorem{remark}[theorem]{Remark}
\newtheorem{lemma}[theorem]{Lemma}
\newtheorem{cor}[theorem]{Corollary}
\newtheorem{prop}[theorem]{Proposition}
\newtheorem{definition}[theorem]{Definition}
\newtheorem{example}[theorem]{Example}

\newtheorem{assumption}[theorem]{Assumption}

\newcommand{\nat}{\mathbb{N}}
\newcommand{\rzecz}{\mathbb{R}}

\newcommand{\eps}{\varepsilon }

\newcommand{\rd}{{\rzecz }^{d}}

\newcommand{\ind}[1]{{1\! \!  1 }_{#1}  }

\newcommand{\diver}{\text{\rm div\,}}
\newcommand{\curl}{\text{\rm curl}\,}

\newcommand{\supp}{\text{\rm supp}\,}

\newcommand{\acal}{\mathcal{A}}

\newcommand{\ccal}{\mathcal{C}}

\newcommand{\fcal}{\mathcal{F}}

\newcommand{\hcal}{\mathcal{H}}

\newcommand{\kcal}{\mathcal{K}}
\newcommand{\lcal}{\mathcal{L}}
\newcommand{\mcal}{\mathcal{M}}

\newcommand{\ocal}{\mathcal{O}}

\newcommand{\scal}{\mathcal{S}}
\newcommand{\tcal}{\mathcal{T}}

\newcommand{\vcal}{\mathcal{V}}
\newcommand{\xcal}{\mathcal{X}}

\newcommand{\zcal}{\mathcal{Z}}

\newcommand{\Fmath}{\mathbb{F}}
\newcommand{\Hmath}{\mathbb{H}}
\newcommand{\Kmath}{\mathbb{K}}

\newcommand{\Umath}{\mathbb{U}}
\newcommand{\Vmath}{\mathbb{V}}

\newcommand{\Bbold}{\mathbf{B}}

\newcommand{\Cbold}{\mathbf{C}}

\newcommand{\fbold}{\mathbf{f}}

\newcommand{\Gbold}{\mathbf{G}}

\newcommand{\ubold}{\mathbf{u}}

\newcommand{\vbold}{\mathbf{v}}

\newcommand{\p}{\mathbb{P}}

\newcommand{\e}{\mathbb{E}}

\newcommand{\norm}[3]{{\|  #1 \| }_{#2}^{#3}}
\newcommand{\Norm}[3]{{\Bigl\|  #1 \Bigr\| }_{#2}^{#3}}
\newcommand{\ilsk}[3]{{( #1 | #2 )}_{#3}}
\newcommand{\Ilsk}[3]{{\Bigl( #1 \bigl| #2 \Bigr)}_{#3}}
\newcommand{\dual}[3]{{\langle #1 | #2 \rangle}_{#3}}
\newcommand{\Dual}[3]{{\Bigl< #1 \bigl| #2 \Bigr>}_{#3}}
\newcommand{\dirilsk}[3]{{( \! ( #1 | #2 ) \! ) }_{#3}}

\newcommand{\nnorm}[3]{{|  #1 |}_{#2}^{#3}}
\newcommand{\Nnorm}[3]{{\Bigl|  #1 {\Bigr| }_{#2}^{#3}}}

\newcommand{\lb}{\langle}
\newcommand{\rb}{\rangle}
\newcommand{\ddual}[4]{{}_{#1}\lb #2 |#3 {\rb }_{#4}}
\newcommand{\Ddual}[4]{{}_{#1}\Bigl< #2 |#3 {\Bigr> }_{#4}}

\newcommand{\qvar}[1]{\langle \! \langle  #1  \rangle \! \rangle }

\newcommand{\lhs}{{\mathcal{T}_2}}

\newcommand{\tOmega}{\tilde{\Omega}}

\newcommand{\tfcal}{\tilde{\fcal}}
\newcommand{\tp}{\tilde{\p}}

\newcommand{\te}{\tilde{\e}}

\newcommand{\ttfcal}{\tilde{\tilde{\fcal}}}

\newcommand{\ttOmega}{\tilde{\tilde{\Omega }}}
\newcommand{\ttp}{\tilde{\tilde{\p }}}

\newcommand{\ttW}{\tilde{\tilde{W }}}

\newcommand{\Hn}{{\Hmath }_{n}}

\newcommand{\Gn}{{G}_{n}}
\newcommand{\Mn}{{M}_{n}}

\newcommand{\tM}{\tilde{M}}
\newcommand{\tMn}{{\tilde{M}}_{n}}

\newcommand{\Pn}{{P}_{n}}
\newcommand{\un}{{u}_{n}}
\newcommand{\wn}{{w}_{n}}
\newcommand{\Jn}[1]{{J}_{n}^{#1}}

\newcommand{\ball}{\mathbb{B}}
\newcommand{\bball}{\bar{\ball}}

\newcommand{\taun}{{\tau}_{n}}

\newcommand{\mhd}{\widetilde{b}}
\newcommand{\MHD}{\widetilde{B}}
\newcommand{\Hall}{\hcal}
\newcommand{\tHall}{\widetilde{\Hall}}
\newcommand{\hall}{\mathfrak{h}}
\newcommand{\thall}{\widetilde{\mathfrak{h}}}

\newcommand{\X}{X}
\newcommand{\Xn}{X_n}
\newcommand{\Xnk}{{X}_{n_k}}
\newcommand{\tXnk}{{\tilde{X}}_{{n}_{k}}}
\newcommand{\tXn}{{\tilde{X}}_{n}}
\newcommand{\tX}{\tilde{X}}
\newcommand{\ttX}{\tilde{\tilde{X}}}

\newcommand{\taunR}{{\tau }_{R}^{n}}

\newcommand{\Law}{\mathrm{Law}}

\newcommand{\ARiesz}{{A}_{n}}
\newcommand{\BRiesz}{{\MHD }_{n}}
\newcommand{\RRiesz}{{\tHall }_{n}}
\newcommand{\fRiesz}{{f}_{n}}

\newcommand{\Vast}{{\Vmath }_{\ast}}

\newcommand{\Uprime}{{\Umath }^{\prime }}
\newcommand{\Vprime}{{\Vmath }^{\prime }}
\newcommand{\Vastprime}{{\Vmath}_{\ast }^{\prime }}

\newcommand{\Vtest}{{\Vmath }_{1,2}}

\newcommand{\Hsol}[1]{{V}_{#1}}
\newcommand{\Hsolprime}[1]{{V}_{#1}^{\prime}}

\newcommand{\Sn}{{S}_{n}}
\newcommand{\Ldn}{{L}^{2}_{n}}

\begin{document}
\title{\bf Martingale solutions of the stochastic Hall-magnetohydrodynamics  equations 
on ${\rzecz }^{3}$ \\
\
 \rm  }
\author{El\.zbieta Motyl
\footnote{Department of Mathematics and Computer Science, University of \L\'{o}d\'{z}, ul. Banacha 22,
91-238 \L \'{o}d\'{z}, Poland; email: elzbieta.motyl@wmii.uni.lodz.pl}}

\maketitle

\begin{abstract}
We prove the existence of a global martingale solution of a stochastic Hall-magne\-to\-hydro\-dynamics  equations  on ${\rzecz }^{3}$ with multiplicative noise. Using the Fourier analysis we construct a sequence of approximate solutions.
The existence of a solution is proved via the stochastic compactness method and the Jakubowski generalization of the Skorokhod theorem for nonmetric spaces, in particular, the spaces with weak topologies.
The main difficulty is caused by the Hall term which makes the equations strongly nonlinear.
\end{abstract}

\section{Introduction}

\medskip  \noindent
Magnetohydrodynamics describes the motion of electrically conductive fluid in the presence of a magnetic field with wide range of applications in geophysics and astrophysics. Mathematically rigorous analysis of the MHD equations started from Duvaut and Lions \cite{Duvaut+Lions'76}  and Sermange and Temam \cite{Sermange+Temam'83}, where deterministic MHD equations are considered. These equations are basically obtained by coupling the Navier-Stokes equations with the Maxwell equations. Stochastic MHD equations with the Gaussian noise were considered, e.g. in  \cite{Barbu_DaPrato'07}, \cite{Chueshov+Millet'10}, \cite{Idriss'18}, \cite{Id+Raza'19}, \cite{Sango'10}, \cite{Schenke'21a}, \cite{Sritharan+Sundar'99}, \cite{Yamazaki'19}.

\medskip  \noindent
\medskip  \noindent
The Hall-MHD model  is important in the physics of plasma. Mathematical derivation of a model taking into account the Hall effect was introduced in \cite{Acheritogaray+Degond'11}. Moreover, the authors in \cite{Acheritogaray+Degond'11} prove the existence of a global weak solutions for the incompressible viscous resistive Hall-MHD equations in ${[0,1]}^{3}$. 
The proof in \cite{Acheritogaray+Degond'11} is based on the Galerkin approximation and the compactness method. 
The uniqueness of global solution in general case  is an open problem. 
Deterministic Hall-MHD equations were also considered in, e.g.,  \cite{Chae+Degond+Liu'14}, \cite{Chae+Lee'14}, \cite{Chae+Schonbek'13}, \cite{Chae+Wan+Wu'15}.

\medskip  \noindent
The analysis of the Hall-magnetohydrodynamics equations was developed by Yamazaki \cite{Yamazaki'17}, where the stochastic Hall-MHD equations perturbed by a Gaussian random field on the domain $D={[0,1]}^{3}$ are considered. 
Using the Galerkin approximation and the tightness criteria introduced by Flandoli and G\c{a}tarek in \cite{Flandoli+Gatarek'1995}, the author proves the existence of a global martingale solution. The method depends strongly on the compactness of appropriate Sobolev embeddings in the case of the domain $D={[0,1]}^{3}$.
See also \cite{Yamazaki'19a} and \cite{Yamazaki'19}.

\medskip  \noindent
Inspired by \cite{Yamazaki'17}, we consider the stochastic Hall-MHD equations with a multiplicative Gaussian noise on ${\rzecz }^{3}$ and prove the existence of a global martingale solution. The main difficulty in comparison to \cite{Yamazaki'17} is the fact that in the case of an unbounded domain the standard Sobolev embeddings are not compact.
To overcome this problem we use the compactness and tightness results analogous to those proved in \cite{Brze+EM'13} in the context of Navier-Stokes equations and in \cite{EM'14} for other hydrodynamics equations. The details are presented in Appendices \ref{sec:aux_funct.anal} and \ref{sec:comp-tight}. 
Besides, the construction of a solution is based on the Fourier analysis and it is closely related to the Littlewood-Paley decomposition, see \cite{Bahouri+Chemin+Danchin'11}.


\medskip  \noindent
We consider the following Hall-MHD system on $[0,T] \times {\rzecz }^{3}$
\begin{align}
& d\ubold  + \Bigl[ (\ubold \cdot \nabla  ) \ubold + \nabla p 
- s \,(\Bbold \cdot \nabla ) \Bbold +s \, \nabla \Bigl( \frac{{|\Bbold |}^{2}}{2} \Bigr) 
- {\nu }_{1} \, \Delta \ubold \Bigr] \, dt \; = \; {\fbold }_1 (t) \nonumber \\
& \qquad \quad + {\Gbold }_{1}(t,\ubold ) \, dW_1(t),
\label{eq:Hall-MHD_u} \\
&  d \Bbold  + \Bigl[ (\ubold \cdot \nabla ) \Bbold - (\Bbold \cdot \nabla ) \ubold 
+ \eps \, \curl [(\curl \Bbold ) \times \Bbold ] 
- {\nu }_{2} \, \Delta \Bbold \Bigr] \, dt \; = \; {\fbold }_2(t) \nonumber \\ 
& \qquad \quad +{\Gbold }_{2}(t,\Bbold ) \, dW_2(t), 
\label{eq:Hall-MHD_B} \\
& \diver \ubold \; = \; 0 \quad \mbox{ and } \quad \diver \Bbold \; = \; 0.
\label{eq:Hall-MHD_incompressibility}
\end{align}
The equations are supplemented by the following initial conditions
\begin{equation}
\ubold (0) \; = \; {\ubold }_{0} \quad \mbox{ and } \quad \Bbold (0) \; = \; {\Bbold }_{0}.
\label{eq:Hall-MHD_ini-cond}
\end{equation}
In this problem $\ubold (t,x) = (u_1,u_2,u_3)(t,x)$, $\Bbold (t,x) = (B_1,B_2,B_3)(t,x)$ for $(t,x)\in [0,T] \times {\rzecz }^{3}$, are  three-dimensional vector fields representing velocity and magnetic fields, respectively, and the real valued function $p(t,x) $  denotes the pressure of the fluid.
The positive constants ${\nu }_{1},{\nu }_{2}, s$ represent kinematic viscosity, resistivity and the Hartmann number, respectively.
The $\curl $-operator is defined for a vector field $\phi : {\rzecz }^{3} \to {\rzecz }^{3}$ by
\[
\curl \phi \; := \; \nabla \times \phi .
\]
The expression
\[
\eps \, \curl [(\curl \Bbold ) \times \Bbold ] ,
\]
which makes the system \eqref{eq:Hall-MHD_u}-\eqref{eq:Hall-MHD_B} strongly nonlinear,
represents the Hall-term with the Hall parameter $\eps >0$. 
For simplicity we will assume that $s=1$ and $\eps =1$. 
Moreover, $\fbold =({\fbold }_{1},{\fbold }_{2})$ stands for the deterministic external forces
and ${\Gbold }_{1}(t,\ubold )d{W}_{1}(t)$, ${\Gbold }_{2}(t,\Bbold )d{W}_{2}(t)$, where ${W}_{1}(t), {W}_{2}(t)$ are cylindrical Wiener processes, stand for the random forces.

\medskip  \noindent
Problem \eqref{eq:Hall-MHD_u}-\eqref{eq:Hall-MHD_ini-cond} can be rewritten  as the following initial value problem for the stochastic equation in appropriate functional spaces 
\begin{equation}
\begin{split} 
&  d\X (t)  +  \bigl[  \acal \X (t)  + \MHD  (\X (t)) +\tHall (\X (t)) \bigr] \, dt  \\
& \qquad \; \;  \; = \;  {f}_{} (t) \, dt 
 +   G (t,\X (t) ) \, dW(t)  , \qquad t \in [0,T] .   \\
& \X (0) \; = \; {\X }_{0}.
\end{split}  \label{eq:Hall-MHD_functional_Intro} 
\end{equation} 
Here $\X = (\ubold ,\Bbold )$,  ${\X }_{0} := ({\ubold }_{0}, {\Bbold }_{0})$, and
  $\acal $, $\MHD $ and $\tHall $ are the maps corresponding to the Stokes-type operators, the MHD-term and the Hall-term, respectively, defined in Section \ref{sec:Hall-MHD_funct-setting}.

\medskip  \noindent
We prove the existence of a global martingale solution of problem \eqref{eq:Hall-MHD_functional_Intro}.
The main result is stated in Theorem  \ref{th:mart-sol_existence}.
Assumptions \ref{assumption-noise} allow to consider the noise term dependent on the unknown process $\X $ and its spatial derivatives.
The construction of a solution is based on the approximation motivated by \cite[Section 4]{Bahouri+Chemin+Danchin'11} and \cite{Feff+McCorm+Rob+Rod'2014}, \cite{Mohan+Sritharan'16}, \cite{Manna+Mohan+Srith'2017}, \cite{Brze+Dha'20}. We consider approximate stochastic equations, called also truncated equations, 
\begin{equation*}
\begin{cases}
 d \Xn (t) &+ \; \bigl[ \ARiesz( \Xn (t))   + \BRiesz  (\Xn (t)) + {\tHall }_{n}(\Xn (t)) \bigr] \, dt
\\
& = \;  \fRiesz (t)  \, dt + \, \Gn \bigl( t,\Xn (t)\bigr) \, dW(t),   \quad t \in [0,T] , \\
\Xn (0) & = \; \Pn {\X }_{0} .
\end{cases}
\end{equation*}
in the infinite dimensional Hilbert spaces ${\Hmath }_{n}$, $n\in \nat $, defined via the Fourier transform techniques, 
see Section \ref{sec:truncated_eq} and Appendix \ref{sec:Fourier_truncation}.  
The crucial point is to  prove suitable uniform a priori estimates for the approximate solutions ${(\Xn )}_{n\in \nat }$ stated in Lemma \ref{lem:Hall-MHD_truncated_estimates}. To deal with the Hall-term, which is strongly nonlinear, we introduce the tri-linear form $\hall $ and the bilinear map $\Hall $, see Section \ref{sec:form_h-map_Hall}. The results from Remark \ref{rem:Hall-term_properties} and Lemma \ref{lem:Hall-term_conv_general}, concerning $\hall $ and $\Hall $, are very important in our approach.
The main idea of the further steps is similar to \cite{Brze+EM'13}.
The processes $\Xn $ generate a tight sequence of probability measures $\{ \Law (\Xn ), n \in \nat  \} $ on appropriate functional space. Using Jakubowski's generalization of the Skorokhod theorem for non metrizable spaces and the martingale representation theorem we prove the existence of a martingale solution of problem \eqref{eq:Hall-MHD_functional_Intro}.   

\medskip  \noindent
Let us briefly recall some relevant applications of the Fourier analysis in partial differential equations.   
In \cite{Feff+McCorm+Rob+Rod'2014} Fefferman and co-authors study deterministic MHD equations on ${\mathbb{R}}^{d}$, $d=2,3$, and prove the existence of a unique local in time solution in the space $\mathcal{C}([0,{T}_{\ast }];{H}^{s})$ for $s>\frac{d}{2}$. By introducing the Fourier truncation the authors approximate the MHD equation by the truncated equations, see \cite[p. 1042]{Feff+McCorm+Rob+Rod'2014}.

\medskip  \noindent
The same idea, referred to as the Friedrichs method, is used by Bahouri, Chemin and Danchin \cite[Section 4]{Bahouri+Chemin+Danchin'11} to study other class of deterministic differential equations. Defining the cut-off operators, see \cite[p. 174]{Bahouri+Chemin+Danchin'11}, the authors consider appropriate approximate equations.

\medskip  \noindent
Analogously to \cite{Feff+McCorm+Rob+Rod'2014}, Mohan and Sritharan \cite{Mohan+Sritharan'16} apply the Fourier truncation method to prove the existence of unique local solution of the stochastic Euler equation on ${\mathbb{R}}^{d}$, $d=2,3$, in Sobolev spaces ${H}^{s}$ for $s > \frac{d}{2}+1$. The same method is also used by Manna, Mohan and Sritharan \cite{Manna+Mohan+Srith'2017} to study local solutions of the stochastic 
MHD equations with L\'{e}vy noise.
On the other hand, Brze\'{z}niak and Dhariwal \cite{Brze+Dha'20} apply the truncated approximation in the study of global solutions of the stochastic tamed Navier-Stokes equations.

\medskip  \noindent
The paper is organized  as follows. In Sections \ref{sec:basic_notation} and \ref{sec:form_b-map_B}  we recall some standard notations and results. In Section \ref{sec:form_h-map_Hall} we analyze the Hall-term. 
Section \ref{sec:Hall-MHD_funct-setting} is devoted to the functional setting of the Hall-MHD problem.
In Section \ref{sec:statement} we formulate the assumptions, definition of a martingale solution and state the main theorem. In Section \ref{sec:truncated_eq} we consider approximate equations and prove  a priori estimates.  The proof of the existence of a global martingale solution is established in Section \ref{sec:existence}. Some auxiliary results related to the Fourier analysis are contained in Appendix \ref{sec:Fourier_truncation}.  Results concerning compactness and tightness criteria are presented Appendices \ref{sec:aux_funct.anal} and \ref{sec:comp-tight}. Appendix \ref{app:Hall-term_conv_general_proof} 
contains the proof of some convergence result for the Hall-term,  used in the proof of the main theorem.

\medskip
\section{Functional setting} \label{sec:funct-setting}

\medskip  
\subsection{Basic spaces and notations} \label{sec:basic_notation}

\medskip  \noindent
Let ${\ccal }^{\infty }_{c} = {\ccal }^{\infty }_{c}({\rzecz }^{3},{\rzecz }^{3})$ denote the space of all ${\rzecz }^{3}$-valued functions of class ${\ccal }^{\infty }$ with compact supports in ${\rzecz }^{3}$, and let
\begin{itemize}
\item $\vcal := \{ u \in {\ccal }^{\infty }_{c} ({\rzecz }^{3},{\rzecz }^{3}): \; \diver u =0  \} $,
\item $H := $ the closure of $\vcal $ in $L^2({\rzecz }^{3},{\rzecz }^{3})$,
\item $V := $ the closure of $\vcal $ in $H^1({\rzecz }^{3},{\rzecz }^{3})$.
\end{itemize}
In the space $H$ we consider the inner product and the norm inherited from $L^2({\rzecz }^{3},{\rzecz }^{3})$
and denote them by $\ilsk{\cdot }{\cdot }{H}$ and $\nnorm{\cdot }{H}{}$, respectively, i.e.
\[
\ilsk{u}{v}{H} := \ilsk{u}{v}{L^2} , \qquad \nnorm{u}{H}{} := \nnorm{u}{L^2}{}, \qquad u,v \in H .
\]
In the space $V$ we consider the inner product inherited from ${H}^{1}({\rzecz }^{3},{\rzecz }^{3})$, i.e.
\begin{equation*}
\ilsk{u}{v}{V} := \ilsk{u}{v}{H}  + \ilsk{\nabla u}{\nabla v}{L^2} , \qquad u,v \in V ,
\end{equation*}
where 
\begin{equation*}
 \ilsk{\nabla u}{\nabla v}{L^2} 
= \sum_{i=1}^{3} \int_{{\rzecz }^{3}} \frac{\partial u}{\partial x_i} \cdot \frac{\partial v}{\partial x_i} \, dx ,  
\end{equation*}
and the norm induced by $\ilsk{\cdot }{\cdot }{V}$, i.e.
\begin{equation*}
\norm{u}{V}{} \; := \; \bigl( \nnorm{u}{H}{2} + \nnorm{\nabla u}{L^2}{2} {\bigr) }^{\frac{1}{2}} .
\end{equation*}
Let us also, for any $m\ge 0$ consider the following standard scale of Hilbert spaces
\begin{equation}
{V}_{m} \; := \;  \mbox{the closure of $\vcal $ in ${H}^{m}({\rzecz }^{3} , {\rzecz }^{3} )$} 
\label{eq:V_m}
\end{equation}
with the inner product inherited from the space ${H}^{m}({\rzecz }^{3} , {\rzecz }^{3})$.
Of course, ${V}_{0}=H$ and ${V}_{1}=V$.

\medskip  \noindent
By ${L}^{2}_{loc}({\rzecz }^{3},{\rzecz }^{3})$, we denote the space of all Lebesgue measurable functions $v:{\rzecz }^{3} \to {\rzecz }^{3}$ such that $\int_{K} {|v(x)|}^{2} \, dx < \infty $
for every compact subset $K \subset {\rzecz }^{3}$. In the space ${L}^{2}_{loc}({\rzecz }^{3},{\rzecz }^{3})$, we consider the Fr\'{e}chet topology generated by the family of seminorms
\begin{equation}
{p}_{R}(v) \; := \; \Bigl( \int_{{\ocal }_{R}} {|v(x)|}^{2} \, dx {\Bigr) }^{\frac{1}{2}}, \qquad R \in \nat ,
\label{eq:p_R-seminorms}
\end{equation} 
where ${({\ocal }_{R})}_{R \in \nat }$ is an increasing sequence of open bounded subsets of ${\rzecz }^{3}$ with smooth boundaries and such that $\bigcup_{R\in \nat } {\ocal }_{R} = {\rzecz }^{3}$.

\medskip \noindent
By ${H}_{loc}$, we denote the space $H$ endowed with the Fr\'{e}chet topology inherited from the space
${L}^{2}_{loc}({\rzecz }^{3},{\rzecz }^{3})$.

\medskip  \noindent
Let $T>0$. By ${L}^{2}(0,T;{H}_{loc})$ we denote the space of measurable functions 
 $ u :[0,T] \to H   $ such that for all $ R \in \nat $
\[
{p}_{T,R}(u) \; = \; \Bigl(  \int_{0}^{T} {p}_{R}^{2}(u (t,\cdot ))  \, dt {\Bigr) }^{\frac{1}{2}}  < \infty ,
\]
where ${p}_{R}$ is defined defined in \eqref{eq:p_R-seminorms}.
Note that ${p}_{T,R}$ are seminorms given explicitely by 
\begin{equation}
{p}_{T,R}(u) \; = \; \Bigl(  \int_{0}^{T} \int_{{\ocal }_{R}}  {|u  (t,x)|}^{2}  \, dxdt {\Bigr) }^{\frac{1}{2}} .
\label{eq:seminorms-L^2(0,T;H_loc)} 
\end{equation}
In ${L}^{2}(0,T;{H}_{loc})$ we consider the topology   generated by the seminorms 
$({p}_{T,R}{)}_{R\in \nat } .$


\medskip
\noindent
\bf Notations. \rm 
Let $(X,\norm{\cdot }{X}{}) $, $(Y,\norm{\cdot }{Y}{}) $ be two normed spaces. The symbol $\lcal (X,Y)$ stands for the space of all bounded linear operators from $X$ to $Y$. If $Y=\rzecz $, the $X':= \lcal (X, \rzecz )$ is called the dual space  of $X$. The standard duality pairing is denoted by $\ddual{X'}{\cdot }{\cdot }{X}$. If no confusion seems likely we omit the subscripts $X'$ and $X$ and write $\dual{\cdot }{\cdot }{}$. If $X$ and $Y$ are separable Hilbert spaces, then by $\lhs (X,Y)$ we will denote the space of all Hilbert-Schmidt operators from $X$ to $Y$ endowed with the standard norm.

\medskip
\subsection{The form $b$ and the map $B$}  \label{sec:form_b-map_B}

\medskip  \noindent
Let us consider the following tri-linear form
\begin{equation}  
b(u,w,v ) = \int_{{\rzecz }^{3}}\bigl( u \cdot \nabla w \bigr) v \, dx .
\label{eq:b-form}
\end{equation}
We will recall basic properties of the form $b$, see also Temam \cite{Temam'79}. 
By the Sobolev embedding theorem, see \cite{Adams}, and the H\"{o}lder inequality, we obtain the following estimate
\begin{align}
|b(u,w,v )|
\; &\le \;   c \norm{u }{V}{} \norm{w }{V}{} \norm{v }{V}{} , \qquad u,w,v \in V   
\label{eq:b-form_est-V}
\end{align}
for some positive constant $c$. Thus the form $b$ is continuous on $V$.
Moreover, if we define a bilinear map $B$ by $B(u,w):=b(u,w, \cdot )$, then by inequality 
\eqref{eq:b-form_est-V} we infer that $B(u,w) \in {V}_{}^{\prime }$ for all $u,w\in V$ and that the following inequality holds
\begin{equation}  
|B(u,w) {|}_{V^{\prime }} 
\le c  \norm{u }{V}{}\norm{w }{V}{},\qquad u,w \in V .
\label{eq:B-map_est-V}  
\end{equation}
Moreover, the mapping $B: V \times V \to V^{\prime } $ is bilinear and continuous.

\medskip  \noindent
Let us also recall the following properties of the form $b$, see \cite[Lemma II.1.3]{Temam'79}, 
\begin{equation}  
b(u,w, v ) =  - b(u,v ,w), \qquad u,w,v \in V .
\label{eq:b-form_antisym}
\end{equation}
In particular,
\begin{equation}  
b(u,v,v) =0   \qquad u,v \in V.
\label{eq:b-form_wirowosc}
\end{equation}

\medskip  \noindent
If $m > \frac{5}{2} $ then by the Sobolev embedding theorem,
\[
{H}^{m-1}({\rzecz }^{3}  , {\rzecz }^{3} ) \hookrightarrow  {\ccal }_{b}({\rzecz }^{3} , {\rzecz }^{3} )
   \hookrightarrow {L}^{\infty } ({\rzecz }^{3} , {\rzecz }^{3} ).
\]
Here ${\ccal }_{b}({\rzecz }^{3} , {\rzecz }^{3} )$ denotes the space of continuous and bounded ${\rzecz }^{3} $-valued functions defined on ${\rzecz }^{3} $.
If $u,w \in V$ and $v \in {V}_{m}$ with $m > \frac{5}{2}$ then
\[
|b(u,w,v)| \; = \; |b(u,v,w)|
\; \le \;  \nnorm{u}{{L}^{2}}{} \nnorm{w}{{L}^{2}}{} \nnorm{\nabla v}{{L}^{\infty }}{}
\; \le \;  {c}_{} \nnorm{u}{{L}^{2}}{} \nnorm{w}{{L}^{2}}{} \norm{v}{{V}_{m}}{}
\]
for some constant ${c}_{} >0 $, where ${V}_{m}$ is the space defined by \eqref{eq:V_m}.
Thus, $b$ can be uniquely extended to the tri-linear form (denoted by the same letter)
\[
   b : H \times H \times {V}_{m} \to \rzecz
\]
and $|b(u,w,v)| \le  {c}_{} \nnorm{u}{{L}^{2}}{} \nnorm{w}{{L}^{2}}{} \norm{v}{{V}_{m}}{}$
for $u,w \in H$ and $v \in {V}_{m}$. At the same time the operator $B$ can be uniquely extended
to a bounded bilinear operator
\[
    B : H \times H \to {V}_{m}^{\prime } .
\]
In particular, it satisfies the following estimate
\begin{equation}  
|B(u,w) {|}_{{V}_{m}^{\prime }} \le c \nnorm{u}{H}{}  \nnorm{w}{H}{} ,\qquad u,w \in H.
\label{eq:B-map_est-H} 
\end{equation}
We will also use the following notation, $B(u):=B(u,u)$.

\medskip
\begin{lemma} \ \label{lem:B-map_loc-Lipsch}
The map $B:V \to V^{\prime }$ is locally Lipschitz continuous, i.e. for every $r>0$ there exists a constant ${L}_{r}$ such that
\begin{equation*}
 \bigl| B(u) - B(\tilde{u}) {\bigr| }_{V^{\prime }} \le {L}_{r} \norm{u - \tilde{u}}{V}{} ,
 \qquad u , \tilde{u } \in V , \quad \norm{u}{V}{}, \norm{\tilde{u}}{V}{} \le r  .
\end{equation*}
\end{lemma}

\medskip
\begin{proof}
The assertion is classical and  follows from the following estimates
\begin{equation*}
\begin{split}
&\bigl| B(u,u) -  B(\tilde{u},\tilde{u}) {\bigr| }_{V^{\prime }}
 \; \le \; \bigl| B(u,u-\tilde{u})  {\bigr| }_{V^{\prime }}
   + \bigl| B(u-\tilde{u},\tilde{u}) {\bigr| }_{V^{\prime }} \\
\; &\le \; \norm{B}{}{}  (\norm{u}{V}{} +\norm{\tilde{u}}{V}{} ) \norm{u-\tilde{u}}{V}{}
  \le 2r  \norm{B}{}{}  \cdot \norm{u-\tilde{u}}{V}{} .
\end{split}
\end{equation*}
Thus the Lipschitz condition holds with ${L}_{r}= 2r \norm{B}{}{} $, where $\norm{B}{}{}$ stands for the norm of the bilinear map $B:V\times V \to V^{\prime }$. The proof is thus complete.
\end{proof}

\medskip  \noindent
We will use the following version of the convergence result for the map $B$ proved in \cite{EM'14}.

\medskip
\begin{lemma} \label{lem:B-map_conv-aux}
(See \cite[Lemma 6.1]{EM'14}.) \it 
 Let $u,w \in {L}^{2}(0,T;H)$ and let ${(u_n)}_{n} , {(w_n)}_{n} \subset {L}^{2}(0,T;H)$ be two sequence bounded in ${L}^{2}(0,T;H)$ and  convergent to $u$,  $w$, respectively, in the the space ${L}^{2}(0,T;{H}_{loc}) $. 
If $m > \frac{5}{2}$, then for all $t \in [0,T]$ and all $\varphi  
\in \Hsol{m} ({\rzecz }^{3};{\rzecz }^{3})$:
\[
\lim_{n \to \infty } \int_{0}^{t} \dual{B({u}_{n}(s),{w}_{n}(s))}{\varphi }{} \, ds 
\; = \; \int_{0}^{t} \dual{B(u(s),w (s))}{\varphi }{} \, ds . 
\] 
\end{lemma}

\medskip  \noindent
Recall that the space ${L}^{2}(0,T;{H}_{loc})$ is defined in Section \ref{sec:basic_notation}.

\medskip
\subsection{The form $\hall $ and the map $\Hall $} \label{sec:form_h-map_Hall}

\medskip  \noindent
Let us introduce the following tri-linear form associated with the Hall term and defined by 
\[
\hall (u,w,v) \; := \; - \int_{{\rzecz }^{3}} \curl [u \times \curl w ] \cdot v \, dx 
\]
for $u,w,v \in \vcal $.
Using the integration by parts formula for the  $\curl $-operator, we obtain
\begin{equation}
\hall (u,w,v ) \; := \; - \int_{{\rzecz }^{3}} [u \times \curl w ] \cdot \curl v \, dx .
\label{eq:hall-form}
\end{equation}
Since $(a\times b) \cdot c = - (a\times c) \cdot b $  for $a,b,c \in {\rzecz }^{3}$, we infer that  
\begin{equation}
\hall (u,v,w) \; = \; - \hall (u,w,v). 
\label{eq:hall-form_antisym}
\end{equation} 
In particular,
\begin{equation}
\hall (u,v,v) \; = \; 0 .
\label{eq:hall-form_perp}
\end{equation}
Note that,
if $u=w$ then using the formula: $u \times (\curl u) = \nabla (\frac{{|u|}^{2}}{2}) - (u \cdot \nabla )u $, we obtain
\begin{equation}
\hall (u,u,v) \; = \; - b (u,u,\curl v).
\label{eq:hall-form_b-form}
\end{equation}


\medskip
\begin{remark} \label{rem:Hall-term_properties}
 \bf (Basic properties of the  form $\hall $.) \it
\begin{description}
\item[(i) ] There exists a constant $c>0$ such that
\begin{equation}
\begin{split}
|\hall (u,w,v)|
\; &\le \; c \, \norm{u}{H^1}{} \cdot \norm{w}{H^1}{} \cdot \norm{ v}{H^2}{}, \qquad u,w \in V, \quad v \in \Hsol{2} .
\end{split}
\label{eq:hall-form_est-V-V}
\end{equation}
Thus the form $\hall $ can be extended to the continuous tri-linear form (denoted again by $\hall $)
\begin{equation*}
\hall : V \times V \times \Hsol{2}  \; \to \; \rzecz .
\end{equation*}
\item[(ii) ] Moreover, there exists a continuous bilinear map  $\Hall : V \times V \to \Hsolprime{2} $ such that  
\begin{equation}
\dual{\Hall (u,w)}{v}{} \; = \; \hall (u,w,v) , \qquad u,w \in V , \quad v \in \Hsol{2} ,
\label{eq:Hall-map_V-V}
\end{equation} 
and 
\begin{equation}
\nnorm{\Hall (u,w)}{\Hsolprime{2}  }{} 
\; \le \; c \, \norm{u}{H^1}{} \cdot \norm{w}{H^1}{} , \qquad u,w \in V.
\label{eq:Hall-map_est-V-V}
\end{equation}
By \eqref{eq:hall-form_antisym}
\begin{equation*}
\dual{\Hall (u,w)}{v}{} \; = \; -\dual{\Hall (u,v)}{w}{}  , \qquad u\in V , 
\quad w,v \in \Hsol{2} ,
\end{equation*}
and, in particular,
\begin{equation*}
\dual{\Hall (u,v)}{v}{} \; = \; 0 , \qquad u\in V , \quad v \in \Hsol{2} .
\end{equation*}
Let us also note that for $u=w \in V $, by  \eqref{eq:hall-form_b-form}, we have
\begin{equation*}
\dual{\Hall (u,u)}{v}{}  \; = \; \dual{B(u,u)}{\curl v}{} , \qquad u \in V , 
\quad v \in \Hsol{2} .
\end{equation*} 
\item[(iii) ] If $m  > \frac{5}{2}$,  then there exists a constant $c>0$ such that
\begin{equation*}
\begin{split}
|\hall (u,w,v)| 
\; &\le \; c \, \nnorm{u}{L^2}{} \cdot \norm{w}{H^1}{} \cdot \norm{v}{H^{m }}{} ,
\qquad u \in H, \quad w \in V , \quad v \in \Hsol{m}.
\end{split}
\end{equation*}
Thus the form $\hall $  extends to the continuous tri-linear form (denoted still by $\hall $)
\begin{equation*}
\hall : H \times V \times \Hsol{m}  \; \to \; \rzecz ,
\end{equation*}
and the map $\Hall $  extends to the continuous bilinear map
  $\Hall : H \times V \to \Hsolprime{m} $   
such that 
\begin{equation}
\nnorm{\Hall (u,w)}{\Hsolprime{m}}{} 
\; \le \; c \, \nnorm{u}{L^2}{} \cdot \norm{w}{H^1}{} , \qquad u \in H , \quad w \in V.
\label{eq:Hall-map_est-H-V}
\end{equation}
\end{description}
\end{remark}

\medskip
\begin{proof} 
By the H\"{o}lder inequality and the Sobolev embedding theorem we obtain
\[
\begin{split}
& |\hall (u,w,v)| \; = \; \Bigl| \int_{{\rzecz }^{3}} [u \times (\curl w)] \cdot \curl v \, dx   \Bigr| 
\; \le \; \nnorm{u}{L^4}{} \cdot \nnorm{\curl w}{L^2}{} \cdot \nnorm{\curl v}{L^4}{}  \\
\; &\le \; \tilde{c} \, \norm{u}{H^1}{} \cdot \norm{w}{H^1}{} \cdot \norm{\curl v}{H^1}{} 
\; \le \; c \, \norm{u}{H^1}{} \cdot \norm{w}{H^1}{} \cdot \norm{ v}{H^2}{} ,
\end{split}
\]
for some constants $\tilde{c},c >0$.
This completes the proof of assertion (i). Assertion (ii) follows from (i) and \eqref{eq:hall-form_antisym} and \eqref{eq:hall-form_perp}.

\medskip  \noindent
To prove (iii), let us fix $m  > \frac{5}{2}$. Since ${H}^{m -1} ({\rzecz }^{3}) \hookrightarrow {L}^{\infty } ({\rzecz }^{3})$,
by the H\"{o}lder inequality and the Sobolev embedding theorem we obtain
\[
\begin{split}
& |\hall (u,w,v)| \; = \; \Bigl| \int_{{\rzecz }^{3}} [u \times (\curl w)] \cdot \curl v \, dx   \Bigr|
\; \le \; \nnorm{u}{L^2}{} \cdot \nnorm{\curl w}{L^2}{} \cdot \nnorm{\curl v}{L^{\infty }}{} \\
\; &\le \; \tilde{c} \, \nnorm{u}{L^2}{} \cdot \norm{w}{H^1}{} \cdot \norm{\curl v}{H^{m -1}}{} 
\; \le \; c \, \nnorm{u}{L^2}{} \cdot \norm{w}{H^1}{} \cdot \norm{v}{H^{m}}{} .
\end{split}
\]
The proof of the remark is thus complete.
\end{proof}

\medskip  \noindent
We will use the following notation
\begin{equation*}
\Hall (u) \; := \; \Hall (u,u). 
\end{equation*}

\medskip
\begin{lemma} \label{lem:Hall-map_loc-Lipsch}
The map $\Hall : V \to \Hsolprime{2} $ is locally Lipschitz continuous, i.e.
for every $ r>0 $ there exists   ${L}_{\Hall }(r) > 0 $ such that 
\[
\nnorm{\Hall (u)-\Hall (\tilde{u})}{\Hsolprime{2} }{} \; \le \; L_{\Hall }(r) \, \norm{u-\tilde{u}}{V}{}, \qquad u,\tilde{u} \in V, \quad \norm{u}{V}{}, \norm{\tilde{u}}{V}{} \le r.
\]
\end{lemma}

\medskip
\begin{proof}
By \eqref{eq:Hall-map_est-V-V}, we infer that for all $u,\tilde{u} \in V$ such that  $ \norm{u}{V}{}, \norm{\tilde{u}}{V}{} \le r $,
\[
\begin{split}
\nnorm{\Hall (u) - \Hall (\tilde{u})}{\Hsolprime{2}  }{}
\; &= \; \nnorm{\Hall (u,u) - \Hall (\tilde{u},\tilde{u})}{\Hsolprime{2} }{}
\; \le \; \nnorm{\Hall (u,u-\tilde{u})}{\Hsolprime{2} }{} + \nnorm{\Hall (u-\tilde{u}, \tilde{u})}{\Hsolprime{2}}{} \\
\; & \le \; \norm{\Hall }{}{} \, \norm{u}{V}{} \, \norm{u-\tilde{u}}{V}{} 
+ \norm{\Hall }{}{} \, \norm{u-\tilde{u}}{V}{} \, \norm{\tilde{u}}{V}{} \\
\; &= \;  \norm{\Hall }{}{} \, (\norm{u}{V}{} + \norm{\tilde{u}}{V}{} ) \, \norm{u-\tilde{u}}{V}{} 
\; \le \; 2r \, \norm{\Hall }{}{} \, \norm{u-\tilde{u}}{V}{}.
\end{split}
\]
This means that the local Lipschitz condition holds with the constant 
$L_{\Hall }(r):= 2r \norm{\Hall }{}{}$, where $\norm{\Hall }{}{}$ denotes the norm of the bilinear map
$\Hall :V\times V \to \Hsolprime{2}  $. The proof of the lemma is thus complete.
\end{proof}

\medskip
\noindent
In the following lemma we will prove some result concerning the convergence of the Hall term $\Hall $. 
This result is analogous to Lemma \ref{lem:B-map_conv-aux}.

\medskip
\begin{lemma} \label{lem:Hall-term_conv_general}  \it 
Let $u \in {L}^{2}(0,T;H)$ and $w \in {L}^{2}(0,T;V)$  and let
$(\un ) \subset {L}^{2}(0,T;H)$
and $ (\wn ) \subset {L}^{2}(0,T;V)$ be two sequences  such that
\begin{itemize}
\item $(\un )$ is bounded in ${L}^{2}(0,T;H)$ and $\wn \to w $ weakly in ${L}^{2}(0,T;V)$,
\item  $\un \to u$ and $\wn \to w $  in ${L}^{2}(0,T; {H}_{loc})$.
\end{itemize}
If $m > \frac{5}{2}$, then for all $t \in [0,T] $ and  all 
 $ \psi \in \Hsol{m}  $:
\begin{equation*}
\lim_{n\to \infty } \int_{0}^{t}\dual{\Hall (\un (s),\wn (s))}{ \psi }{} \, ds  
\; = \; \int_{0}^{t}\dual{\Hall (u(s),w(s))}{\psi }{} \, ds .
\end{equation*}
\end{lemma}

\medskip  \noindent
Recall that the space  ${L}^{2}(0,T;{H}_{loc})$ is  defined in Section \ref{sec:basic_notation}. 

\medskip  
\begin{proof}
The proof of Lemma \ref{lem:Hall-term_conv_general} is postponed to Appendix \ref{app:Hall-term_conv_general_proof}.
\end{proof}

\medskip
\subsection{Functional setting of the Hall-MHD system} \label{sec:Hall-MHD_funct-setting}

\medskip  \noindent
Using the spaces $H$ and $V$  defined in Section \ref{sec:basic_notation},
let us consider the spaces
\begin{equation}
\Hmath \; := \; H \times H , \qquad \Vmath := V \times V 
, \qquad   {\Vmath }^{\prime } \; := \;  \text{the dual space of $\Vmath $}
\label{eq:Hall-MHD_H-V}
\end{equation}
with the following inner products
\begin{align*}
\ilsk{\phi }{\psi }{\Hmath }  
\; &:= \;  \ilsk{\ubold }{\vbold }{{L}^{2}} + \ilsk{\Bbold }{\Cbold }{{L}^{2}} 
\end{align*}
for all $ \phi =(\ubold ,\Bbold ), \, \, \;  \psi =(\vbold ,\Cbold ) \in \Hmath $ ,
and
\begin{equation*}
\begin{split}
\ilsk{\phi }{\psi }{\Vmath }
\; &:= \;   \ilsk{\phi }{\psi }{\Hmath }  
 + \dirilsk{\phi }{\psi }{}   
\end{split}
\end{equation*} 
for all $ \phi  =(\ubold ,\Bbold ) , \;  \psi =(\vbold ,\Cbold )  \in \Vmath $,
where
\begin{equation}
\dirilsk{\phi }{\psi }{} \; := \; {\nu }_{1} \, \ilsk{\nabla \ubold }{\nabla \vbold }{L^2} 
+{\nu }_{2} \, \ilsk{\nabla \Bbold }{\nabla \Cbold }{L^2} .
\label{eq:il-sk_Dirichlet+curl}  
\end{equation}
In the spaces $\Hmath $ and $\Vmath $ we consider the norms induced by the inner products $\ilsk{\cdot }{\cdot }{\Hmath }$ and  $\ilsk{\cdot }{\cdot }{\Vmath }$, respectively, i.e.
$\nnorm{\phi }{\Hmath }{2}:= \ilsk{\phi }{\phi }{\Hmath }$ for $\phi \in \Hmath $, and
\begin{equation}
\norm{\phi }{\Vmath }{2} \; = \; \nnorm{\phi }{\Hmath }{2} + \norm{\phi }{}{2},
\label{eq:Hall-MHD_V-norm} 
\end{equation}
where
\begin{equation}
\norm{\phi }{}{2} \; := \; \dirilsk{\phi }{\phi }{}, \qquad \phi \in \Vmath . 
\label{eq:norm_Dirichlet+curl}  
\end{equation}

\medskip  
\noindent
\bf The operator $\acal $.  \rm 
We define the operator $\acal $ by the following formula
\begin{equation}
\dual{\acal \phi }{\psi }{} \; = \;  \dirilsk{\phi }{\psi }{}, \qquad \phi ,\psi \in \Vmath , 
\label{eq:A_acal_rel} 
\end{equation}
where $\dirilsk{\cdot  }{\cdot }{}$ is given by \eqref{eq:il-sk_Dirichlet+curl} .

\medskip  
\begin{remark}  \label{rem:Acal-term_properties} 
It is clear that  $\acal \in \lcal (\Vmath ,\Vmath ')$
and
\begin{equation}
\nnorm{\acal \phi }{\Vmath '}{} \; \le \; \norm{\phi }{}{}, \qquad \phi \in \Vmath ,
\label{eq:Acal_norm}
\end{equation}
where $\norm{\cdot }{}{}$ is given by \eqref{eq:norm_Dirichlet+curl}.
\end{remark}

\medskip  \noindent
Indeed, inequality \eqref{eq:Acal_norm} follows from \eqref{eq:Hall-MHD_V-norm}  and the following inequalities
\[
|\dual{\acal \phi }{\psi }{} | \; = \; |\dirilsk{\phi }{\psi }{}|
\; \le \; \norm{\phi }{}{} \, \norm{\psi }{}{}
\; \le \; \norm{\phi }{}{} \, (\nnorm{\psi }{}{2} +\norm{\psi }{}{2} {)}^{\frac{1}{2}}
\; = \; \norm{\phi }{}{} \, \norm{\psi }{\Vmath }{}.
\]

\medskip  \noindent
For $m_1, m_2 \ge 0 $ let us define
\begin{equation}
{\Vmath }_{m_1,m_2} \; := \; {V}_{m_1} \times {V}_{m_2},
\label{eq:Vmath_m1,m2}
\end{equation}
where ${V}_{m_1}, {V}_{m_2}$ are the spaces defined by \eqref{eq:V_m}. In ${\Vmath }_{m_1,m_2}$ we consider the product norm
\begin{equation}
\norm{\phi }{{m_1,m_2} }{2}  \; := \; 
\norm{\ubold }{V_{m_1}}{2} + \norm{\Bbold }{V_{m_2}}{2}
\label{eq:Vmath_m1,m2-norm} 
\end{equation}
for all $\phi =(\ubold , \Bbold ) \in {\Vmath }_{m_1,m_2}$.
In the case when $m_1=m_2=:m$ we denote
\begin{equation}
{\Vmath }_{m} \; := \; V_m \times V_m   
\quad \mbox{ and } \quad \norm{\cdot }{m}{} \; := \; \norm{\cdot }{m,m}{} .
\label{eq:Vmath_m}
\end{equation}
It is clear that if $m=1$, then ${\Vmath }_{1}=\Vmath $ and $\norm{\cdot }{1}{} = \norm{\cdot }{\Vmath }{}$. 

\medskip  \noindent
Let $T>0$. By ${L}^{2}(0,T;{\Hmath }_{loc})$ we denote the space of measurable functions 
 $ \phi :[0,T] \to \Hmath    $ such that for all $ R \in \nat $
 \begin{equation*}                 
{p}_{T,R}(\phi):= \Bigl( \int_{0}^{T} [{p}_{R}^{2}(\ubold (t,\cdot )) + {p}_{R}^{2}(\Bbold (t,\cdot )) ] \,dt  {\Bigr) }^{\frac{1}{2}} <\infty , 
\end{equation*}
where $\phi =(\ubold ,\Bbold )$, and ${p}_{R}$ are defined in \eqref{eq:p_R-seminorms}.
Explicitly, 
\begin{equation*}
{p}_{T,R}(\phi) \; = \; \Bigl(  \int_{0}^{T} \int_{{\ocal }_{R}} 
\bigl[ {|\ubold (t,x)|}^{2}  + {|\Bbold (t,x)|}^{2} \bigr] \, dxdt {\Bigr) }^{\frac{1}{2}}  .
\end{equation*}
In the space ${L}^{2}(0,T;{\Hmath }_{loc})$ we consider the topology   generated by the seminorms  $({p}_{T,R}{)}_{R\in \nat } .$

\medskip  \noindent
\bf The form $\mhd$ and the operator $\MHD$\rm .
Using the form $b$  defined by \eqref{eq:b-form} we will consider the tri-linear  form  $\mhd$ on $\Vmath \times \Vmath \times \Vmath $, where $\Vmath $ is defined by \eqref{eq:Hall-MHD_H-V}, see 
Sermange and Temam \cite{Sermange+Temam'83} and Sango \cite{Sango'10}. Namely,
\begin{align*}
\mhd ({\phi }^{(1)},{\phi }^{(2)},{\phi }^{(3)})
\; := \;  & b({\ubold }^{(1)}, {\ubold }^{(2)}, {\ubold }^{(3)})
     -   b({\Bbold }^{(1)}, {\Bbold }^{(2)}, {\ubold }^{(3)})  \\
     & + b({\ubold }^{(1)}, {\Bbold }^{(2)}, {\Bbold }^{(3)})
     -  b({\Bbold }^{(1)}, {\ubold }^{(2)}, {\Bbold }^{(3)}) ,
\end{align*}
where ${\phi }^{(i)} = ({\ubold }^{(i)},{\Bbold }^{(i)}) \in \Vmath $, $i=1,2,3$. 
By \eqref{eq:b-form_est-V} we see that the form $\mhd$ is continuous.
Moreover, by \eqref{eq:b-form_antisym} and \eqref{eq:b-form_wirowosc}
the form $\mhd$ has the following properties
\begin{align*} 
&  \mhd ({\phi }^{(1)},{\phi }^{(2)},{\phi }^{(3)})
\; = \; -  \mhd ({\phi }^{(1)},{\phi }^{(3)},{\phi }^{(2)}) , \qquad 
{\phi }^{(i)} \in \Vmath , \quad i=1,2,3  
\end{align*}
and in particular
\begin{align*}  
& \mhd ({\phi }^{(1)},{\phi }^{(2)},{\phi }^{(2)}) \; = \; 0 , \qquad {\phi }^{(1)}, {\phi }^{(2)} \in \Vmath .
\end{align*}        
Now, let us define a bilinear map $\MHD$ by 
\begin{equation}   
\MHD (\phi ,\psi ) \; := \;  \mhd (\phi ,\psi ,\cdot )  ,  \qquad \qquad \phi ,\psi  \in \Vmath . 
\label{eq:MHD-map}
\end{equation}

\medskip  \noindent
We will also use the notation $ \MHD (\phi ) :=  \MHD (\phi ,\phi ) $.

\medskip
\noindent
Let us recall   properties of the map $\MHD$ stated in \cite{EM'14}.

\medskip
\begin{lemma} \label{lem:MHD-term_properties}
\rm (See \cite[Lemma 6.4]{EM'14}) \it 
\begin{description}
\item[(i) ] There exists a constant ${c}_{\MHD } >0 $ such that 
\[
\nnorm{ \MHD(\phi , \psi )}{\Vmath '}{} 
\;  \le \;  {c}_{\MHD }\norm{\phi }{\Vmath }{} \norm{\psi }{\Vmath }{}  , \qquad \phi ,\psi  \in \Vmath . 
\]
In particular, the map $\MHD: \Vmath \times \Vmath \to \Vmath '$ is bilinear and continuous.
Moreover, 
\[
\dual{\MHD(\phi ,\psi )}{\theta  }{} \; = \; - \dual{\MHD(\phi ,\theta  )}{\psi  }, 
\qquad \phi ,\psi , \theta  \in \Vmath , 
\]
and, in particular,
\begin{equation}
\dual{\MHD (\phi )}{\phi }{} \; = \; 0 , 
\qquad \phi  \in \Vmath  . 
\label{eq:MHD-map_perp}
\end{equation}
\item[(ii) ] The mapping $\MHD$ is locally Lipschitz continuous on the space $\Vmath $, i.e.
for every $r>0 $ there exists a constant ${L}_{r}>0$ such that 
\[
\nnorm{ \MHD(\phi )- \MHD (\tilde{\phi } )}{\Vmath '}{} 
\; \le \; {L}_{r}  \norm{\phi -\tilde{\phi } ) }{\Vmath }{} , \qquad \phi ,\tilde{\phi } \in \Vmath ,
  \quad \norm{\phi }{\Vmath }{} , \norm{\tilde{\phi }}{\Vmath }{} \le r .
\]   
\item[(iii) ] If $m> \frac{5}{2}$, then $\MHD$ can be extended to the bilinear mapping from $\Hmath \times \Hmath $ to ${\Vmath }_{m}'$ (denoted still by $\MHD$) such that
\begin{equation}
\nnorm{ \MHD(\phi , \psi )}{{\Vmath }_{m}'}{} 
\; \le \; {c}_{\MHD }(m) \nnorm{\phi }{\Hmath }{} \nnorm{\psi }{\Hmath }{} , \qquad \phi ,\psi  \in \Hmath  ,
\label{eq:MHD-map_est-H-H} 
\end{equation}
where ${c}_{\MHD }(m)$ is a positive constant.
\end{description}
\end{lemma}

\medskip  \noindent
From \eqref{eq:MHD-map} and Lemma \ref{lem:B-map_conv-aux}  we obtain immediately the following corollary.

\medskip
\begin{cor} \label{cor:MHD-map_conv-aux}
 Let $\phi ,\psi  \in {L}^{2}(0,T;\Hmath )$ and let $({\phi }_n) , ({\psi }_n) \subset {L}^{2}(0,T;\Hmath )$ be two sequence bounded in ${L}^{2}(0,T;\Hmath )$ and  such that 
\[
{\phi }_{n} \; \to \;  \phi   \quad \mbox{ and } \quad {\psi }_{n} \to \psi   \quad   \mbox{ in } \quad {L}^{2}(0,T;{\Hmath }_{loc}).
\] 
If $m > \frac{5}{2}$, then for all $t \in [0,T]$ and all 
$\varphi  \in {\Vmath }_{m} $:
\[
\lim_{n \to \infty } \int_{0}^{t} \dual{\MHD ({\phi }_{n}(s),{\psi }_{n}(s))}{\varphi }{} \, ds 
\; = \; \int_{0}^{t} \dual{\MHD (\phi (s),\psi (s))}{\varphi }{} \, ds ,
\] 
where ${\Vmath }_{m}$ is the space defined by \eqref{eq:Vmath_m}.
\end{cor}

\medskip
\noindent
\bf The form $\thall $ and the map $\tHall $. \rm 
Using the form $\hall $  defined by \eqref{eq:hall-form} we will consider the tri-linear  form  $\thall $ on $\Vmath \times \Vmath \times \Vtest   $ defined by
\[
\thall ({\phi }^{(1)},{\phi }^{(2)},{\phi }^{(3)}) 
\; := \; \hall ({\Bbold }^{(1)},{\Bbold }^{(2)},{\Bbold }^{(3)}),
\]
where ${\phi }^{(i)} = ({\ubold }^{(i)},{\Bbold }^{(i)}) \in \Vmath $ for $i=1,2$, and ${\phi }^{(3)} = ({\ubold }^{(3)},{\Bbold }^{(3)}) \in \Vtest   $. Due to \eqref{eq:Vmath_m1,m2}, ${\Vmath }_{1,2} := \Hsol{1} \times \Hsol{2}  $.  
By \eqref{eq:hall-form_est-V-V} we see that the form $\thall $ is continuous.
Moreover, by \eqref{eq:hall-form_antisym} and \eqref{eq:hall-form_perp}
the form $\thall $ has the following properties
\begin{align*} 
&  \thall ({\phi }^{(1)},{\phi }^{(2)},{\phi }^{(3)})
\; = \; -  \thall  ({\phi }^{(1)},{\phi }^{(3)},{\phi }^{(2)}) , \qquad 
{\phi }^{(1)} \in \Vmath , \quad {\phi }^{(i)} \in \Vtest  , \quad i=2,3 ,  
\end{align*}
and in particular
\begin{align*}  
& \thall ({\phi }^{(1)},{\phi }^{(2)},{\phi }^{(2)}) \; = \; 0 , 
\qquad {\phi }^{(1)} \in \Vmath , \quad {\phi }^{(2)} \in \Vtest .
\end{align*}        
Now, let us define a bilinear map $\tHall $ by 
\begin{equation} 
\tHall (\phi ,\psi ) \; := \;  \thall (\phi ,\psi ,\cdot )  ,  \qquad \qquad \phi ,\psi  \in \Vmath . 
\label{eq:tHall_map} 
\end{equation}
We will also use the notation $ \tHall (\phi ) :=  \tHall  (\phi ,\phi ) $.

\medskip  \noindent
Using Remark \ref{rem:Hall-term_properties} and Lemma \ref{lem:Hall-map_loc-Lipsch} we  obtain the following result.

\medskip
\begin{lemma} \label{lem:tHall-term_properties} 
\bf (Properties of the map $\tHall $.) \it 
\begin{description}
\item[(i) ] There exists a constant ${c}_{\tHall } >0 $ such that 
\[
\nnorm{ \tHall (\phi , \psi )}{{\Vmath }_{1,2}'}{} 
\;  \le \;  {c}_{\tHall }\norm{\phi }{\Vmath }{} \norm{\psi }{\Vmath }{}  , \qquad \phi ,\psi  \in \Vmath . 
\]
In particular, the map $\tHall : \Vmath \times \Vmath \to {\Vmath }_{1,2}' $ is well-defined bilinear and continuous.
Moreover, 
\[
\dual{\tHall (\phi ,\psi )}{\Theta  }{} \; = \; - \dual{\tHall (\phi ,\theta  )}{\psi  }, 
\qquad \phi \in \Vmath , \quad \psi , \theta \in {\Vmath }_{1,2} ,
\]
and, in particular,
\begin{equation}
\dual{\tHall (\phi )}{\phi }{} \; = \; 0 ,  \qquad \phi  \in  {\Vmath }_{1,2}  . 
\label{eq:tHall-map_perp}
\end{equation}
\item[(ii) ] The map $\tHall $ is locally Lipschitz continuous on the space $\Vmath $, i.e.
for every $r>0 $ there exists a constant ${L}_{r}>0$ such that 
\[
\nnorm{ \tHall (\phi )- \tHall  (\tilde{\phi } )}{ {\Vmath }_{1,2}' }{} 
\; \le \; {L}_{r}  \norm{\phi -\tilde{\phi } ) }{\Vmath }{} , \qquad \phi ,\tilde{\phi } \in \Vmath ,
  \quad \norm{\phi }{\Vmath }{} , \norm{\tilde{\phi }}{\Vmath }{} \le r .
\]   
\item[(iii) ] If $s\ge 0$ and $m> \frac{5}{2}$, then $\tHall $ can be extended to the bilinear mapping from $\Hmath \times \Vmath $ to $ {\Vmath }_{s,m}' $ (denoted still by $\tHall $) such that
\[
\nnorm{ \tHall (\phi , \psi )}{{\Vmath }_{s,m}'}{} 
\; \le \; {c}_{\tHall }(s,m) \nnorm{\phi }{\Hmath }{} \nnorm{\psi }{\Vmath }{} , \qquad \phi   \in \Hmath  , \quad \psi \in \Vmath ,
\]
where ${c}_{\tHall }(s,m)$ is a positive constant.

\medskip \noindent
In particular, if $m> \frac{5}{2}$, then  $\tHall $ can be extended to the bilinear mapping from $\Hmath \times \Vmath $ to ${\Vmath }_{m}'$ (denoted still by $\tHall $) such that
\begin{equation}
\nnorm{ \tHall (\phi , \psi )}{{\Vmath }_{m}'}{} 
\; \le \; {c}_{\tHall }(m) \nnorm{\phi }{\Hmath }{} \norm{\psi }{\Vmath }{} , \qquad \phi   \in \Hmath  ,
\quad  \psi \in \Vmath ,
\label{eq:tHall-map_est-H-V}
\end{equation}
where ${c}_{\tHall }(m):={c}_{\tHall }(m,m)$ and ${\Vmath }_{m}$ is the space defined by \eqref{eq:Vmath_m}.
\end{description}
\end{lemma}

\medskip
\noindent
From lemma \ref{lem:Hall-term_conv_general} and the definition of $\tHall $, we obtain immediately the following result

\medskip
\begin{cor} \label{cor:tHall-term_conv_general} 
Let $\phi \in {L}^{2}(0,T;\Hmath )$ and $\psi  \in {L}^{2}(0,T;\Vmath )$  and let
$({\phi }_{n} ) \subset {L}^{2}(0,T;\Hmath )$
and $ ({\psi }_{n}) \subset {L}^{2}(0,T;\Vmath )$ be two sequences  such that
\begin{itemize}
\item $({\phi }_{n})$ is bounded in ${L}^{2}(0,T;\Hmath )$ and ${\psi }_{n} \to \psi $ weakly in ${L}^{2}(0,T;\Vmath )$,
\item  ${\phi }_{n} \to \phi $ and ${\psi }_{n} \to \psi  $  in ${L}^{2}(0,T; {\Hmath }_{loc}) $.
\end{itemize}
If $s\ge 0 $ and $m > \frac{5}{2}$, then for all $t \in [0,T] $ and all
 $ \varphi  \in {\Vmath }_{s,m}  $:
\begin{equation*}
\lim_{n\to \infty } \int_{0}^{t}\dual{\tHall ({\phi }_{n} (s),{\psi }_{n} (s))}{ \varphi  }{} \, ds  
\; = \; \int_{0}^{t}\dual{\tHall (\phi (s),\psi (s))}{\varphi  }{} \, ds ,
\end{equation*}
where ${\Vmath }_{s,m}  $ is the the space defined by \eqref{eq:Vmath_m1,m2}-\eqref{eq:Vmath_m1,m2-norm}.
\end{cor}

\section{Martingale solutions of the Hall-magnetohydrodynamics equations} \label{sec:statement}

\medskip
We will formulate assumptions imposed on the noise terms, the deterministic external forces and the initial conditions in problem \eqref{eq:Hall-MHD_u}-\eqref{eq:Hall-MHD_ini-cond}.

\medskip
\begin{assumption} \label{assumption-noise}  \rm 
We assume that
\begin{description}
\item[(G.1)] ${\Kmath }_{1}, {\Kmath }_{2} $  are separable Hilbert spaces, and
\[
G_i: [0,T] \times V  \to \lhs ({\Kmath }_{i}, H), \qquad i=1,2,
\]
are two measurable map which are Lipschitz continuous, i.e.
there exist constants ${L}_{i}$, $i=1,2$,
such that 
\begin{equation}
\norm{G_i(t,{\phi }_{1}) - G_i(t,{\phi }_{2})}{\lhs ({\Kmath }_{i}, H)}{2} 
\; \le \; {L}_{i}\, \norm{{\phi }_1-{\phi }_{2}}{V}{2}  ,
   \qquad {\phi }_{1}, {\phi }_{2} \in V  , \, \,  t \in [0,T] .
\label{eq:G_i_Lipschitz}   
\end{equation} 
In addition, there exist ${\lambda }_{i}$, ${\varrho }_{i} \in \rzecz $ and 
$\eta \in (0,2]$  such that
\begin{equation} \label{eq:G_i}
\norm{G_i(t,\phi  )}{\lhs ({\Kmath }_{i},H)}{2}
\; \le \; {\nu }_{i}(2- \eta ) \nnorm{\nabla \phi }{L^2}{2} +{\lambda }_{i} \nnorm{\phi }{H}{2} +{\varrho }_{i}  , \quad (t,\phi ) \in [0,T] \times V  .
\end{equation}
\item[(G.2)] The maps $G_i$, $i=1,2$, can be extended to  measurable maps
\[ 
g_i :[0,T] \times H \to \lcal  ( {\Kmath }_{i},  {V}^{\prime }) 
\]
such that for some $C_i>0$
\begin{equation} \label{eq:G_i*}
\sup_{\psi \in V ,\norm{\psi }{V}{} \le 1 } \sup_{y\in {\Kmath }_{i},\norm{y}{{\Kmath }_{i}}{} \le 1}
{|\ddual{V'}{g(t,\phi )(y)}{\psi }{V}|}^{2} 
 \le C_i (1 + \nnorm{\phi }{H}{2}) , \qquad (t,\phi ) \in [0,T] \times H  . 
\end{equation}
\item[(G.3)]
Moreover, for every $\psi  \in V$ the maps ${\tilde{g_i}}_{\psi }$ defined for $\phi \in {L}^{2}(0,T;H) $  by
\begin{equation} \label{eq:G_i**}
 \bigl( { \tilde{g_i} }_{\psi }(\phi )\bigr)  (t) := 
\{ {\Kmath }_{i} \ni y \mapsto \ddual{V'}{g_i(t,\phi (t))(y)}{\psi }{V} \in \rzecz \} \in  \lhs ({\Kmath }_{i},\rzecz ) ,
\quad t \in [0,T],
\end{equation}
are continuous maps from ${L}^{2}(0,T;{H}_{loc}) $ into $ {L}^{2}(0,T;\lhs ({\Kmath }_{i},\rzecz ) ) $.
\end{description}
\end{assumption}

\medskip  \noindent
The spaces $H$, $V$ and ${L}^{2}(0,T;{H}_{loc})$ are defined in Section \ref{sec:basic_notation}.
Recall also that for any Hilbert spaces $\Kmath $ and $Y$ by $\lhs (\Kmath;Y)$ we denote the space of Hilbert-Schmidt operators from $\Kmath$ into  $Y$.

\medskip
\begin{assumption} \label{assumption-data}  \rm 
We assume also that the following objects are given.
\begin{description}
\item[(H.1)] A real number $p$ such that 
\begin{equation} \label{eqn-p_cond}
p \in [2,2+\gamma )  ,
\end{equation}
where
\begin{equation}
\gamma := \begin{cases}
\frac{\eta }{2-\eta } , \quad &\mbox{ if }  \eta \in [0,2) , \\
\infty , \quad &\mbox{ if }  \eta =2,
\end{cases}
\label{eq:gamma}
\end{equation}
and $\eta $ is the parameter from inequality \eqref{eq:G_i}. 
\item[(H.2)] ${\X }_{0} := ({\ubold }_{0}, {\Bbold }_{0}) \in H \times H $ and $f:= (f_1,f_2)$, where $f_i \in {L}^{p}(0,T;{V}^{\prime })$ for $i=1,2$.
\item[(H.3)] 
$\mathfrak{A}:=(\Omega , \fcal , \Fmath , \p )$ is a filtered probability space with a filtration
 $\Fmath ={({\fcal }_{t})}_{t \ge 0}$ satisfying usual hypotheses and 
 $W_i(t)$  are two cylindrical  Wiener processes in a separable Hilbert space ${\Kmath }_{i}$, $i=1,2$, defined on the stochastic basis $\mathfrak{A}$. 
\end{description}
\end{assumption}

\medskip \noindent
Let $W(t):= (W_1(t),W_2(t))$. Then $W(t)$ is a cylindrical Wiener process on $\Kmath :={\Kmath }_{1} \times {\Kmath }_{2}$,  on the stochastic basis $\mathfrak{A}$.

\medskip
\begin{example} \rm 
Let ${\Kmath }_{1} = {\Kmath }_{2} := {\ell }^{2}(\nat )$, where ${\ell }^{2}(\nat )$ denotes the space of all sequences $({h}_{j}{)}_{j \in \nat } \subset \rzecz $ such that $\sum_{j=1}^{\infty }{h}_{j}^{2} < \infty $. It is a Hilbert space with the scalar product given by
$
\ilsk{h}{k}{{\ell }^{2}} \; := \;  \sum_{j=1}^{\infty } {h}_{j} {k}_{j},
$
where $h=({h}_{j})$ and $k=({k}_{j})$ belong to ${\ell }^{2}(\nat )$.
\bigskip  \noindent
Let us put
\begin{equation} \label{eq:G_def}
G_i(\phi )h \; := \;  \sum_{j=1}^{\infty } \bigl[
\bigl( {b}_{i}^{(j)} \cdot \nabla  \bigr) \phi   + {c}_{i}^{(j)} \phi \bigr] {h}_{j} ,
\qquad \phi \in V, \quad h=({h}_{j}) \in {\ell }^{2}(\nat ), \quad i=1,2,
\end{equation}
where 
\begin{equation*}
\begin{split}
&   {b}_{i}^{(j)} : {\rzecz }^{3} \to {\rzecz }^{3}  \quad \mbox{ and } \quad 
  {c}_{i}^{(j)}  : {\rzecz }^{3} \to \rzecz    \quad \mbox{ for } \quad j \in \nat \quad \mbox{ and } \quad 
  i=1,2
\end{split}
\end{equation*}
are given functions of class ${\ccal }^{\infty }$ and such that for each $i=1,2$
\begin{equation}  \label{eq:C_1}
{C}_{i} \; := \; \sum_{j=1}^{\infty }\bigl( \nnorm{{b}_{i}^{(j)}}{{L}^{\infty }}{2}+\nnorm{\diver {b}_{i}^{(j)} }{{L}^{\infty }}{2} + \nnorm{{c}_{i}^{(j)} }{{L}^{\infty }}{2}\bigr) \; < \; \infty
\end{equation}
and
\begin{equation} \label{eq:est_coercive}
\sum_{k,l=1}^{3} \bigl( 2 {\delta }_{kl}
- \sum_{j=1}^{\infty } {b}_{ik}^{(j)}(x){b}^{(j)}_{il}(x) \bigr) {\zeta }_{k} {\zeta }_{l}
\; \ge \;  a_i {|\zeta |}^{2} , \qquad \zeta \in \rd
\end{equation}
for some $a_i \in (0,{\nu }_{i}]$, where ${\nu }_{i} > 0$ for $i=1,2$ are coefficients given in equations \eqref{eq:Hall-MHD_u} and \eqref{eq:Hall-MHD_B}. Proceeding similarly as in \cite[Section 6]{Brze+EM'13} we can prove that the maps $G_i$, $i=1,2$, given by \eqref{eq:G_def} satisfy Assumption \ref{assumption-noise}. 

\medskip  \noindent
The maps $G_i$, $i=1,2$, define the following noise terms
\begin{equation}
G_i(\phi )(t,x) d W_i(t) \; := \; \sum_{j=1}^{\infty } \bigl[
\bigl( {b}_{i}^{(j)}(x) \cdot \nabla  \bigr) \phi  (t,x) + {c}_{i}^{(j)} (x) \phi (t,x) \bigr] \,
d {\beta }_{i}^{(j)}(t) ,
\end{equation}
where $ {\beta }_{i}^{(j)} $ for $j\in \nat $ and $i=1,2$ are independent standard Brownian motions. \qed
\end{example}

\medskip
\noindent
The functional setting for the Hall-MHD equations \eqref{eq:Hall-MHD_u}-\eqref{eq:Hall-MHD_incompressibility} involves spaces $\Hmath $ and $\Vmath $ being appropriate products of spaces $H$ and $V$ (see Section \ref{sec:Hall-MHD_funct-setting}). Given the maps $G_i$ and $g_i$, $i=1,2$, from Assumption \ref{assumption-noise} we  introduce   maps $G$ and $g$ defined on appropriate product spaces.

\medskip  
\begin{remark} \label{rem:G_properties}
\bf (The maps $G$ and $g$ and their properties.)  \rm 
Let ${G}_{1}$ and ${G}_{2}$ be the maps given in Assumption \ref{assumption-noise}. 
Let us define the following map
\begin{equation}
G(t,\Phi )(y) \; := \; (G_1(t,\ubold )(y_1),G_2(t, \Bbold )(y_2)), 
\label{eq:G_map}
\end{equation}
where $ t \in [0,T], \;  \Phi := (\ubold ,\Bbold ) \in \Vmath ,  \; y=(y_1,y_2) \in \Kmath  \; := \; {\Kmath }_{1} \times {\Kmath }_{2}. $ 
\begin{description} 
\item[(i) ] Then
\[
G : [0,T] \times \Vmath \; \to \; \lhs (\Kmath ,\Hmath ). 
\]
The map $G$ satisfies the Lipschitz condition, i.e there exists a constant $L>0$ such that
\begin{equation}
\norm{G(t,{\Phi }_{1}) - G(t,{\Phi }_{2})}{\lhs (\Kmath,\Hmath )}{2} 
\; \le \; L \norm{{\Phi }_{1}- {\Phi }_2}{\Vmath }{2} ,
   \qquad {\Phi }_{1}, {\Phi }_{2} \in \Vmath  , \, \,  t \in [0,T] .
\label{eq:G_Lipschitz}   
\end{equation} 
\item[(ii) ] The map $G$ satifies the following inequality
\begin{equation}
\begin{split}
\norm{G(s,\Phi )}{\lhs ({\Kmath },{\Hmath })}{2}
\; & \le \; (2-\eta ) \norm{\Phi }{}{2} + \lambda \nnorm{\Phi }{\Hmath }{2} + \varrho  ,  \quad (s,\Phi  ) \in [0,T] \times \Vmath ,
\end{split}
\label{eq:G}
\end{equation}
where $\lambda := {\lambda }_{1}+{\lambda }_{2} $ and $\varrho :={\varrho }_{1} + {\varrho }_{2}$.
\item[(iii) ] Let ${g}_{1}$ and ${g}_{2}$ are the maps from Assumption \ref{assumption-noise} and let us define
\begin{equation}
g(t,\Phi )(y) \; := \; (g_1(t,\ubold )(y_1),g_2(t, \Bbold )(y_2)), 
\label{eq:g_map}
\end{equation}
where $ t \in [0,T], \;  \Phi := (\ubold ,\Bbold ) \in \Vmath ,  \; y=(y_1,y_2) \in \Kmath  $. 
Then the map $g$ is an extension of the map $G$ to a measurable map
\[ 
g :[0,T] \times \Hmath  \to \lcal  ( \Kmath,  {\Vmath }^{\prime }) 
\]
and by \eqref{eq:G_i*} we obtain
\begin{equation} \label{eq:G*}
\sup_{\Psi \in \Vmath  ,\norm{\Psi }{\Vmath }{} \le 1 } \sup_{y\in \Kmath ,\norm{y}{\Kmath }{} \le 1}
{|\ddual{\Vmath '}{g(t,\Phi )(y)}{\Psi }{\Vmath }|}^{2} 
 \le C (1 + \nnorm{\Phi }{\Hmath }{2}) , \qquad (t,\Phi ) \in [0,T] \times \Hmath   . 
\end{equation}
Moreover, for every $\Psi  \in \Vmath $ the map ${\tilde{g}}_{\Psi }$ defined for $\Phi \in {L}^{2}(0,T;\Hmath ) $  by
\begin{equation} \label{eq:G**}
 \bigl( { \tilde{g} }_{\Psi }(\Phi )\bigr)  (t) \; := \; 
\{ \Kmath \ni y \mapsto \ddual{\Vmath '}{g(t,\Phi (t))(y)}{\Psi }{\Vmath } \in \rzecz \} \in  \lhs (\Kmath ,\rzecz ) ,
\quad t \in [0,T],
\end{equation}
are continuous maps from ${L}^{2}(0,T;{\Hmath }_{loc}) $ into $ {L}^{2}(0,T;\lhs (\Kmath,\rzecz ) ) $.
\end{description}
\end{remark}

\medskip  \noindent
Let us recall that the spaces $\Hmath $, $\Vmath $ and ${L}^{2}(0,T;{\Hmath }_{loc})$ are defined in Section \ref{sec:Hall-MHD_funct-setting}.

\medskip  
\noindent
Using the maps introduced in Section \ref{sec:Hall-MHD_funct-setting}, we can rewrite problem \eqref{eq:Hall-MHD_u}-\eqref{eq:Hall-MHD_ini-cond}  as the following stochastic equation
\begin{equation}
\begin{split} 
&  d\X (t)  +  \bigl[  \acal \X (t)  + \MHD  (\X (t)) +\tHall (\X (t)) \bigr] \, dt  \\
& \qquad  \; = \;   {f}_{} (t) \, dt 
 + G (t,\X (t) ) \, dW(t)  , \qquad t \in [0,T] ,   \\
& \X (0) \; = \; {\X }_{0}.
\end{split}  \label{eq:Hall-MHD_functional}
\end{equation} 
Here ${\X }_{0} := ({\ubold }_{0}, {\Bbold }_{0})$ and
$\acal $, $\MHD $ and $\tHall $ are the maps defined by \eqref{eq:A_acal_rel}, 
\eqref{eq:MHD-map} and \eqref{eq:tHall_map}, respectively.

\medskip
\begin{definition}  \rm  \label{def:mart-sol}
Let Assumptions  \ref{assumption-noise} and \ref{assumption-data}  be satisfied. 
We say that there exists a \bf martingale solution \rm of  problem \eqref{eq:Hall-MHD_functional} 
iff there exist
\begin{itemize}
\item[$\bullet $] a stochastic basis $\bar{\mathfrak{A}}:= \bigl( \bar{\Omega }, \bar{\fcal },  \bar{\Fmath } ,\bar{\p }  \bigr) $ with a  filtration $\bar{\Fmath } = \{ {\bar{\fcal }_{t}}{\} }_{t \in [0,T]} $ satisfying the usual conditions, 
\item[$\bullet $] a $\Kmath $-cylindrical Wiener process $\bar{W}$  over
$\bar{\mathfrak{A}} $,
\item[$\bullet $] and an $\bar{\Fmath }$- progressively measurable process $\X : [0,T] \times \Omega \to \Hmath $ with  
$\bar{\p } $-a.e. paths satisfying
\begin{equation*}
\X (\cdot , \omega ) \in \ccal \bigl( [0,T]; {\Hmath }_{w} \bigr) \cap {L}^{2}(0,T;\Vmath ) ,
\end{equation*}
and such that 
for all $ t \in [0,T] $ and  $\varphi  \in \Vtest  $ the following identity holds $\bar{\p }$ - a.s.
\begin{equation} 
\begin{split}
&\ilsk{\X (t)}{\varphi }{\Hmath }  + \int_{0}^{t} \dual{\acal \X (s)}{\varphi }{}  ds
+ \int_{0}^{t} \dual{\MHD (\X (s))}{\varphi  }{}  ds 
+ \int_{0}^{t} \dual{\tHall (\X (s))}{\varphi  }{}  ds
    \\
\; & = \; \ilsk{{\X }_{0}}{\varphi }{\Hmath } 
 + \int_{0}^{t} \dual{f(s)}{\varphi  }{} ds
  + \Dual{\int_{0}^{t}G(s,\X (s))\,  d\bar{W}(s)}{\varphi  }{} .
\end{split}  \label{eq:mart-sol_int-identity}
\end{equation}
and
\begin{equation} 
\bar{\e } \Bigl[ \sup_{t \in [0,T]} \nnorm{\X (t)}{\Hmath }{2} 
 + \int_{0}^{T} \norm{\X (t)}{\Vmath }{2} \, dt\Bigr] \; < \; \infty .
\label{eq:mart-sol_energy-ineq}
\end{equation}   
\end{itemize}
If all the above conditions are satisfied, then the system
\[
\bigl( \bar{\mathfrak{A}}, \bar{W} ,\X \bigr) 
\]
is called a \bf martingale solution \rm of problem \eqref{eq:Hall-MHD_functional}.
\end{definition}

\noindent
${\Hmath }_{w}$ denotes the Hilbert space $\Hmath $ endowed with the weak topology
and $ \ccal ( [0,T]; {\Hmath }_{w} )$ is the space of all weakly continuous functions $\psi :[0,T] \to \Hmath  $, i.e., such that for all $h \in \Hmath $ the real valued function
\[
[0,T] \ni t \to \ilsk{\psi (t)}{h}{\Hmath } \in \rzecz 
\] 
is continuous.

\medskip  
\begin{remark}
Explicitly, if an $\Hmath $-valued process $\X = (\ubold , \Bbold )$ is a solution of problem  \eqref{eq:Hall-MHD_functional}, then the processes $\ubold $ and $\Bbold $ satisfy $\bar{\p }$-a.s. the following identities: for every $t \in [0,T]$ and $\varphi =({\varphi }_{1}, {\varphi }_{2}) \in {V}_{1} \times {V}_{2}$
\[
\begin{split}
& \ilsk{\ubold (t)}{{\varphi }_{1}}{L^2} + \int_{0}^{t} \bigl\{ 
{\nu }_{1} \ilsk{\nabla \ubold (s)}{\nabla {\varphi }_{1}}{L^2} 
+ \ilsk{ [\ubold (s) \cdot \nabla ] \ubold (s) 
- [\Bbold (s) \cdot \nabla ] \Bbold (s) }{{\varphi }_{1}}{L^2}  \bigr\} \, ds 
\\
& \; = \; \ilsk{\ubold (0)}{{\varphi }_{1}}{L^2} +  \int_{0}^{t} \dual{f_1(s)}{{\varphi}_{1}}{} ds
+ \Dual{\int_{0}^{t}{G}_{1}(s, \ubold (s))\,  d{\bar{W}}_{1}(s)}{{\varphi}_{1}}{} ,
\\
& \ilsk{\Bbold (t)}{{\varphi }_{2}}{L^2} + \int_{0}^{t} \bigl\{ 
{\nu }_{2} \ilsk{\nabla \Bbold (s)}{\nabla {\varphi }_{2}}{L^2} 
+ \ilsk{ [\ubold (s) \cdot \nabla ] \Bbold (s) 
- [\Bbold (s) \cdot \nabla ] \ubold (s)  }{{\varphi }_{2}}{L^2} \\
& \qquad \qquad \qquad \qquad - \ilsk{ \Bbold (s) \times \curl \Bbold (s)}{\curl {\varphi }_{2}}{L^2} \bigr\} \, ds 
\\
& \; = \; \ilsk{\Bbold (0)}{{\varphi }_{2}}{L^2} +  \int_{0}^{t} \dual{f_2(s)}{{\varphi}_{2}}{} ds
+ \Dual{\int_{0}^{t}{G}_{2}(s, \Bbold (s))\,  d{\bar{W}}_{2}(s)}{{\varphi}_{2}}{} ,
\end{split}
\] 
and the energy inequality
\[
\bar{\e } \Bigl[ \sup_{t \in [0,T]} (\nnorm{\ubold (t)}{L^2}{2} +\nnorm{\Bbold  (t)}{L^2}{2} )
 + \int_{0}^{T} ( \nnorm{\nabla \ubold (t)}{L^2 }{2} + \nnorm{\nabla \Bbold (t)}{L^2 }{2} ) \, dt\Bigr] \; < \; \infty .
\]
\end{remark}

\medskip  \noindent
Now we formulate the main result concerning the existence of a martingale solution.

\medskip
\begin{theorem} \label{th:mart-sol_existence}
Let Assumptions  \ref{assumption-noise} and \ref{assumption-data}  be satisfied.
In particular, we assume that
 $p$ satisfies  \eqref{eqn-p_cond}, i.e.
\[
p \in [2,2+\gamma ),
\] 
where $\gamma $ is given by \eqref{eq:gamma}.
Then there exists a  martingale solution  of  problem \eqref{eq:Hall-MHD_functional} such that
\begin{description}
\item[(i)] for every $q\in [1,p]$ there exists a constant $C_1(p,q)$ such that
\begin{equation}
\bar{\e } \Bigl[ \sup_{t \in [0,T]} \nnorm{\X (t)}{\Hmath }{q} \Bigr] \; \le \;  C_1(p,q),
\label{eq:H_estimate_Hall-MHD_q}
\end{equation}
\item[(ii)] there exists a constant $C_2(p)$ such that
\begin{equation}
\bar{E} \biggl[ \int_{0}^{T} \norm{\X (t)}{\Vmath }{2} \, dt \biggr] \; \le \; C_2(p) .
\label{eq:V_estimate_Hall-MHD}
\end{equation}
\end{description}
\end{theorem}

\medskip  \noindent
The rest of the paper is devoted to the proof of Theorem \ref{th:mart-sol_existence}.
In the next section we consider equations approximating  equation \eqref{eq:Hall-MHD_functional}.

\medskip
\section{Approximate SPDEs } \label{sec:truncated_eq}

\medskip  \noindent
We will use the Friedrichs method based on the Fourier transform techniques, see \cite[Section 4, p.174]{Bahouri+Chemin+Danchin'11}.
This method has been also used, e.g., in  \cite{Feff+McCorm+Rob+Rod'2014}, \cite{Mohan+Sritharan'16}, \cite{Manna+Mohan+Srith'2017} and \cite{Brze+Dha'20}.

\medskip
\subsection{The subspaces $\Hn $ and the  operators $\Pn $}

\medskip  \noindent
Let 
\begin{equation*}
{\bar{B}}_{n} \; := \; \{ \xi \in {\rzecz }^{3} : \; \nnorm{\xi }{}{} \le n \}  , \qquad n \in \nat 
\end{equation*}
and let
\begin{equation*}
{H}_{n} \; := \; \{ v \in H: \; \; \supp v \in {\bar{B}}_{n}  \} .
\end{equation*}
In the subspace ${H}_{n}$ we consider the norm inherited from the  the space $H$ defined in Section \ref{sec:basic_notation}.
For each $n\in \nat $ let us define a map ${\pi }_{n}$ by
\begin{equation*}
{\pi }_{n} v \; := \; {\fcal }^{-1} (\ind{{\bar{B}}_{n}} \widehat{v}) , \qquad v \in H ,
\end{equation*}
where ${\fcal }^{-1}$ denotes denotes the inverse of the Fourier transform, see Appendix \ref{sec:Fourier_truncation}. 
Using Remark \ref{rem:S_n-projection}, we infer that the map ${\pi }_{n}:H \to H_n$ is the orthogonal projection onto $H_n$.

\medskip  \noindent
Let
\begin{equation*}
{\bar{\ball }}_{n} \; := \; {\bar{B}}_{n} \times {\bar{B}}_{n}
\end{equation*}
and
\begin{equation}
 \Hn \; := H_n \times H_n. 
\label{eq:H_n} 
\end{equation} 
In the subspace $\Hn $ we consider the norm inherited from the space $\Hmath =H\times H $ defined by \eqref{eq:Hall-MHD_H-V}.
Let us define the  operator
\begin{equation}
\Pn \; := \;  {\pi }_{n} \times {\pi }_{n} : \Hmath \; \to \;  \Hn   .
\label{eq:P_n}
\end{equation}
Explicitly, for $\Phi = (\ubold ,\Bbold ) \in \Hmath $
\begin{equation*}
\Pn (\ubold ,\Bbold ) \; = \; ({\pi }_{n}\ubold , {\pi }_{n}\Bbold )
\; = \; ({\fcal }^{-1} (\ind{{\bar{B}}_{n}} \widehat{\ubold }), {\fcal }^{-1} (\ind{{\bar{B}}_{n}} \widehat{\Bbold })) .
\end{equation*}
Since the map ${\pi }_{n}:H \to H_n$ is the orthogonal projection onto $H_n$, we infer that 
\begin{equation*}
\Pn : \Hmath \to \Hn 
\end{equation*}
is the orthogonal projection onto $\Hn $. 

\medskip  \noindent
Using Lemma \ref{lem:Ld_n-H^s-relation} and Corollary \ref{cor:Ld_n-H^s-norm_equiv},  we infer that the subspaces $\Hn $ are embedded in the spaces 
${\Vmath }_{m_1,m_2}$ for $m_1,m_2\ge 0 $, defined by \eqref{eq:Vmath_m1,m2} with the equivalence of norms.
We have the following results. 

\medskip
\begin{lemma} \label{lem:H_n-V_m1,m2-relation}
Let $n \in \nat $ and  ${m}_{1},m_2 \ge 0 $.
Then
\begin{equation*}
\Hn \; \hookrightarrow  \; {\Vmath }_{m_1,m_2},   
\end{equation*}
and   for all  $  u \in \Hn :$
\begin{equation*}
\norm{u}{{m_1,m_2}}{2} \; \le \;  {(1+{n}^{2})}^{m} \, \nnorm{u}{\Hn }{2} ,  
\end{equation*}
where $m=\max \{ m_1,m_2\} $.
\end{lemma}

\noindent
(Note that  the norm of the embedding $\Hn \hookrightarrow  {\Vmath }_{m_1,m_2}$ depends on $n$ and $m_1 ,m_2$.)

\medskip
\begin{cor} \label{cor:H_n-V_m1,m2-norm_equiv}
On the subspace $\Hn $ the norm $\nnorm{\cdot}{\Hn }{}$ and the norms  $\norm{\cdot }{{m_1,m_2}}{}$, for $m_1,m_2\ge 0 $, inherited from the spaces ${\Vmath }_{m_1,m_2}$ are equivalent (with appropriate constants depending on $m_1,m_2$ and $n$).   
\end{cor}

\medskip 
\noindent
Now, we will concentrate on some properties of the operators $\Pn $ in the spaces ${\Vmath }_{m_1,m_2}$ defined by \eqref{eq:Vmath_m1,m2}.
Directly from Lemma \ref{lem:S_n-pointwise_conv} we obtain the following lemma.

\medskip
\begin{lemma} \label{lem:P_n-pointwise_conv}
Let us fix $m_1 ,m_2 \ge 0 $.   
Then for all $n \in \nat $:
\[
\Pn  : {\Vmath }_{m_1,m_2} \; \to \; {\Vmath }_{m_1,m_2}
\]
is well defined linear and bounded. Moreover,  for every $ u \in {\Vmath }_{m_1,m_2}:$
\begin{align*}
\lim_{n \to \infty } \norm{\Pn u-u}{{m_1,m_2}}{}  \; = \; 0 . 
\end{align*} 
\end{lemma}

\medskip  \noindent
From Lemma \ref{lem:P_n-pointwise_conv} we obtain the following corollary which will be frequently used in the proofs.

\medskip
\begin{cor} \label{cor:P_n-pointwise_conv} 
In particular,
\begin{description}
\item[(i) ]  $\Pn \in \lcal (\Vmath , \Vmath )$, and for all $u \in \Vmath $
\[
\lim_{n\to \infty } \norm{ \Pn u -u }{\Vmath }{} \; = \; 0,
\]
\item[(ii) ] For every $m \ge 0 $,  $\Pn \in \lcal ({\Vmath }_{m} , {\Vmath }_{m})$   and for all $u \in {\Vmath }_{m} $
\[
\lim_{n\to \infty } \norm{ \Pn u -u }{{\Vmath }_{m} }{} \; = \; 0,
\]
\item[(iii) ] For every $m \ge 1$, $\Pn \in \lcal ( {\Vmath }_{m} , \Vmath )$  and for all $u \in {\Vmath }_{m} $
\[
\lim_{n\to \infty } \norm{ \Pn u -u }{\Vmath }{} \; = \; 0.
\]
\end{description}
The spaces $\Vmath $ and  ${\Vmath }_{m}$ is defined by \eqref{eq:Hall-MHD_H-V} and \eqref{eq:Vmath_m}, respectively.
\end{cor}

\medskip
\noindent
From Lemma \ref{lem:S_n-norm_conv} we obtain the following 

\medskip
\begin{lemma} \label{lem:P_n-norm_conv}
If $k_1,k_2>0 $, then 
\begin{equation*}
\Pn : {\Vmath }_{m_1+k_1,m_2+k_2} \to {\Vmath }_{m_1,m_2}
\end{equation*} 
is well defined  and bounded. 
Moreover, for every $ u \in {\Vmath }_{m_1+k_1,m_2+k_2}$: 
\begin{equation*}
\norm{\Pn u-u}{{m_1,m_2}}{2} \; \le \; \frac{c}{{(1+{n}^{2})}^{k}} \, \norm{u}{{m_1+k_1,m_2+k_2}}{2},
\end{equation*}
where $k=\max \{ k_1 ,k_2\}$ and $c$ is some constant,
and  hence
\begin{equation*}
\lim_{n\to \infty } \nnorm{\Pn -I}{\lcal ({\Vmath }_{m_1+k_1,m_2+k_2},{\Vmath }_{m_1,m_2})}{} \; = \; 0, 
\end{equation*}
i.e. the sequence $(\Pn )$ is convergent to the identity operator in the sense of operator-norm.
\end{lemma}

\medskip
\subsection{Approximating SPDEs}

\medskip  \noindent
Construction of a solution of problem \eqref{eq:Hall-MHD_functional} is based on appropriate  approximation in the space $\Hn $ defined by \eqref{eq:H_n}. 

\medskip
\begin{definition} \label{def:trunctated_weak} \rm
Let ${\X }_{0} \in \Hn $.
By an \bf approximation \rm of  equation  \eqref{eq:Hall-MHD_functional}  we mean an $\Hn $-valued  continuous, $\Fmath $-adapted  process  ${\{ \Xn (t) \} }_{t \in [0,T]}$  such that for all $t\in [0,T]$ and $\phi \in \Hn $ the following identity holds $\p $ - a.s. 
\begin{equation}
\begin{split}
&\ilsk{\Xn (t)}{\phi}{\Hmath }
+ \int_{0}^{t} \dual{ \acal  \Xn (s)}{ \phi }{} \, ds
+ \int_{0}^{t} \dual{ \MHD (\Xn (s))}{\phi}{}  ds
+ \int_{0}^{t} \dual{ \tHall (\Xn (s))}{ \phi }{} \, ds
   \\
 &= \,\, \ilsk{{\X }_{0}}{\phi}{\Hmath }
+ \int_{0}^{t} \dual{f(s)}{\phi}{}\, ds
+ \int_{0}^{t} \dual{ G(s,\Xn (s))\,  dW(s)}{\phi}{}.
\end{split}   \label{eq:Hall-MHD_truncated_weak}
\end{equation}
\end{definition}

\medskip
\noindent
Note that the test functions $\phi $ in \eqref{eq:Hall-MHD_truncated_weak} belong the subspace $\Hn $.
Using  the Riesz representation theorem for continuous linear functionals on $\Hn$ and the fact that in the space $\Hn $ all norms inherited from the space ${H}^{m_1} \times {H}^{m_2} := {H}^{m_1}({\rzecz }^{3},{\rzecz }^{3}) \times {H}^{m_2}({\rzecz }^{3},{\rzecz }^{3})$, where $m_1,m_2 \ge 0$, are equivalent (see Corollary \ref{cor:H_n-V_m1,m2-norm_equiv}), identity \eqref{eq:Hall-MHD_truncated_weak} can be written as a stochastic equation in $\Hn$. 
Since $\Pn : \Hmath \to \Hn $ is the $\ilsk{\cdot }{\cdot }{\Hmath }$-orthogonal projection, in particular we have
\begin{equation*}
\ilsk{v}{\Pn \varphi }{\Hmath } \; = \; \ilsk{\Pn v}{\varphi }{\Hmath } \quad \mbox{ for all } \quad  \varphi \in \Hn . 
\label{eq:Pn_Hn_Riesz}
\end{equation*}
Thus for a fixed $v \in \Hmath $ the Riesz representation of the functional
\[
\Hn \ni \varphi  \,\, \mapsto \,\, \ilsk{v}{\varphi }{\Hmath } \in \rzecz 
\]
is equal $\Pn v$.

\medskip  
\begin{remark} \label{rem:truncated_funct_Riesz} 
Let $n\in \nat $ be fixed.
\begin{description}
\item[(i)] For every $v \in \Vmath $ there exist $\ARiesz (v), \BRiesz (v), \RRiesz (v)  \in \Hn $ such that
for every $\varphi \in \Hn $
\begin{align}
\ddual{\Vprime }{ \acal v}{ \varphi }{\Vmath }  
\; &= \; \ilsk{\ARiesz(v)}{\varphi }{\Hmath } , 
\label{eq:Acal_Hn_Riesz}
\\
\ddual{\Vprime }{ \MHD (v)}{ \varphi }{\Vmath }  
\; &= \; \ilsk{\BRiesz(v)}{\varphi }{\Hmath } , 
\label{eq:Bn_Hn_Riesz}
\\
\ddual{{\Vmath }_{1,2}' }{ \tHall (v)}{ \varphi }{{\Vmath }_{1,2}}  
\; &= \; \ilsk{\RRiesz(v)}{\varphi }{\Hmath }  .
\label{eq:R_Hn_Riesz}
\end{align}
Moreover, the map $\Vmath \ni v \mapsto \ARiesz(v) \in \Hn  $ is linear.
\item[(ii)] For every $f \in \Vmath '$ there exists $\fRiesz (v) \in \Hn $ such that 
\begin{equation}
\ddual{\Vprime }{f}{ \varphi }{\Vmath }  
\; = \; \ilsk{\fRiesz }{\varphi }{\Hmath } \quad \mbox{ for all } \quad  \varphi \in \Hn .
\label{eq:f_Hn_Riesz}
\end{equation}
\end{description}
\end{remark}

\medskip
\begin{proof}
\bf Ad (i). \rm Let us fix $v \in \Vmath $. Consider the following functional
\begin{equation}
\Hn \ni \varphi \,\, \mapsto \,\, \ddual{\Vprime }{ \acal v}{ \varphi }{\Vmath } \in \rzecz .  
\label{eq:Acal_Hn}
\end{equation}
Since, by Remark \ref{rem:Acal-term_properties} and Corollary \ref{cor:H_n-V_m1,m2-norm_equiv}, for all $\varphi \in \Hn$
\begin{equation*}
|\ddual{\Vprime }{ \acal v}{ \varphi }{\Vmath } |
\; \le \; \nnorm{\acal  v}{\Vprime }{} \, \norm{\varphi }{\Vmath }{}
\; \le \;  c_n \norm{v}{\Vmath }{} \,  \nnorm{\varphi }{\Hmath }{} 
\end{equation*}
for some constant ${c}_{n}$ independent of $\varphi $,
the map defined by \eqref{eq:Acal_Hn} is a continuous linear functional on $\Hn$.
By the Riesz representation theorem there exists $\ARiesz (v) \in \Hn$ such that
\begin{equation*}
\ddual{\Vprime }{ \acal v}{ \varphi }{\Vmath }  
\; = \; \ilsk{\ARiesz(v)}{\varphi }{\Hmath } , \qquad \varphi \in \Hn ,
\end{equation*}
i.e. \eqref{eq:Acal_Hn_Riesz} holds.
Since $\acal $ is linear, the map $\Vmath \ni v \mapsto \ARiesz(v) \in \Hn  $ is linear as well.

\medskip  \noindent
To prove \eqref{eq:Bn_Hn_Riesz},
let us consider the following functional
\begin{equation}
\Hn \ni \varphi \; \mapsto \; \ddual{\Vprime }{ \MHD (v)}{ \varphi }{\Vmath } \in \rzecz .  
\label{eq:Bn_Hn}
\end{equation}
Since, by Lemma \ref{lem:MHD-term_properties}(i) and Corollary \ref{cor:H_n-V_m1,m2-norm_equiv}, for all $\varphi \in \Hn$
\begin{equation*}
|\ddual{\Vprime }{ \MHD (v)}{ \varphi }{\Vmath } |
\; \le \; \nnorm{\MHD (v)}{\Vprime }{} \, \norm{\varphi }{\Vmath }{}
\; \le \;  {c}_{n} \, \norm{v}{\Vmath }{2} \,  \nnorm{\varphi }{\Hmath }{} 
\end{equation*}
for some constant ${c}_{n}$ independent of $\varphi $,
the map defined by \eqref{eq:Bn_Hn} is a continuous linear functional on $\Hn$.
By the Riesz representation theorem there exists $\BRiesz (v) \in \Hn$ such that
\begin{equation*}
\ddual{\Vprime }{ \MHD (v)}{ \varphi }{\Vmath }  
\; = \; \ilsk{\BRiesz(v)}{\varphi }{\Hmath } , \qquad \varphi \in \Hn ,
\end{equation*}
i.e., \eqref{eq:Bn_Hn_Riesz} holds.

\medskip  \noindent
Similarly, let us  consider the following functional
\begin{equation}
\Hn \ni \varphi \; \mapsto \; \ddual{\Vprime }{ \tHall (v)}{ \varphi }{\Vmath } \in \rzecz .  
\label{eq:R_Hn}
\end{equation}
Since, by Lemma \ref{lem:tHall-term_properties}(1)  and Corollary \ref{cor:H_n-V_m1,m2-norm_equiv}, for all $\varphi \in \Hn$
\begin{equation*}
|\ddual{ \Vtest '   }{ \tHall (v)}{ \varphi }{\Vtest  } |
\; \le \; \nnorm{\tHall (v)}{ \Vtest ' }{} \, \norm{\varphi }{ \Vtest }{}
\; \le \;  {c}_{n} \, \norm{v}{\Vmath  }{2} \,  \nnorm{\varphi }{\Hmath }{} 
\end{equation*}
for some constant ${c}_{n}$ independent of $\varphi $,
the map defined by \eqref{eq:R_Hn} is a continuous linear functional on $\Hn$.
By the Riesz representation theorem there exists $\RRiesz (v) \in \Hn$ such that
\begin{equation*}
\ddual{ \Vtest ' }{ \tHall (v)}{ \varphi }{ \Vtest }  
\; = \; \ilsk{\RRiesz(v)}{\varphi }{\Hmath } , \qquad \varphi \in \Hn ,
\end{equation*}
i.e., \eqref{eq:R_Hn_Riesz} holds.

\medskip \noindent
\bf Ad. (ii). \rm 
For a fixed $f\in \Vprime $ let us consider the functional
\begin{equation}
\Hn \ni \varphi \,\, \mapsto \,\, \ddual{\Vprime }{ f}{ \varphi }{\Vmath } \in \rzecz .  
\label{eq:f_Hn}
\end{equation}
Note that that the map defined by  \eqref{eq:f_Hn} is a restriction of functional $f$ to the subspace $\Hn$. 
Since for all $\varphi \in \Hn$
\begin{equation*}
|\ddual{\Vprime }{ f}{ \varphi }{\Vmath } |
\,\, \le \,\, \nnorm{f}{\Vprime }{} \, \norm{\varphi }{\Vmath }{}
\,\, \le \,\,  {c}_{n}  
 \nnorm{f}{\Vprime }{} \,  \nnorm{\varphi }{\Hmath }{} 
\end{equation*}
for some constant ${c}_{n}$ independent of $\varphi $,
the map defined by \eqref{eq:f_Hn} is a continuous linear functional on $\Hn$.
Let $\fRiesz (v) \in \Hn$ denote its Riesz representation in $\Hn$. Then we have
\begin{equation*}
  \ddual{\Vprime }{f}{ \varphi }{\Vmath }  
  \,\, = \,\, \ilsk{\fRiesz }{\varphi }{\Hmath } , \qquad \varphi \in \Hn ,
\end{equation*}
i.e., \eqref{eq:f_Hn_Riesz} holds.
\end{proof}

\medskip  \noindent
\begin{remark} \label{rem:truncated_Riesz}   
By Remark \ref{rem:truncated_funct_Riesz}, integral identity \eqref{eq:Hall-MHD_truncated_weak}
 can be written equivalently as the following stochastic equation in $\Hn$
\begin{equation*}
\begin{split}
&  \Xn (t)  + \int_{0}^{t}\bigl[ \ARiesz( \Xn (s)) + \BRiesz  (\Xn (s)) +\RRiesz ( \Xn (s)) \bigr] \, ds \\
\,\, &= \,\,  \Pn {\X }_{0}  + \int_{0}^{t}  \fRiesz (s)  \, ds  
  + \int_{0}^{t} \Gn (s,\Xn (s)) \, dW(s) ,
\quad  t \in [0,T]  ,   
\end{split} 
\end{equation*}
where $\Gn $ is a map defined by
\begin{equation}
\Gn : [0,T] \times \Hmath \ni (s,X) \; \mapsto \; [\Kmath \ni y \mapsto \Pn (G(s,X)(y)) ] \; \in \; 
\lhs (\Kmath ,\rzecz )  .
\label{eq:G_n}
\end{equation}
\end{remark}

\medskip  \noindent
Let us note that, by Lemmas \ref{lem:MHD-term_properties} and \ref{lem:tHall-term_properties}, for every $n\in \nat $ the map
\[
\Hn \ni v \,\, \mapsto \,\,   \BRiesz  (v) + \RRiesz (v) \in \Hn
\]
is locally Lipschitz continuous.  (The Lipschitz constants depend also on $n$.)

\medskip  \noindent
Let us consider the $n$-th approximating stochastic partial differential equation in the space $\Hn $ defined by \eqref{eq:H_n}, i.e.,
\begin{equation}
\begin{cases}
 d \Xn (t) &+ \; \bigl[ \ARiesz( \Xn (t))   + \BRiesz  (\Xn (t)) + {\tHall }_{n}(\Xn (t)) \bigr] \, dt
\\
& = \;  \fRiesz (t)  \, dt + \, \Gn \bigl( t,\Xn (t)\bigr) \, dW(t),   \quad t \in [0,T] , \\
\Xn (0) & = \; \Pn {\X }_{0} .
\end{cases}
\label{eq:Hall-MHD_truncated}
\end{equation}

\medskip
\begin{prop} \label{prop:Hall-MHD_truncated_existence}
For each $n \in \nat $, there exists a unique global solution ${(\Xn (t))}_{t\in [0,T]}$ of problem \eqref{eq:Hall-MHD_truncated}.
\end{prop}

\medskip
\begin{proof} The proof is a direct application of Theorem 3.1 from \cite{Albeverio+Brzezniak+Wu'2010}.
In fact, by Lemmas \ref{lem:MHD-term_properties} and \ref{lem:tHall-term_properties}, for every $n\in \nat $ the nonlinear terms  $\BRiesz $   and ${\tHall }_{n} $ are locally Lipschitz. Thus the exists a local solution $\Xn $ of problem  \eqref{eq:Hall-MHD_truncated} defined on some random interval $[0, {\tau }_{n})$. 
Since moreover, by \eqref{eq:MHD-map_perp} and \eqref{eq:tHall-map_perp} for all $\Xn \in \Hn  $
\[
 \dual{\BRiesz (\Xn ) + {\tHall }_{n} (\Xn )}{\Xn }{} \; = \; 0
\]
using the It\^{o} formula for the function $F(x) = \nnorm{x}{{\Hmath }}{2}$, $x \in \Hmath $, we can prove  that the processes $\Xn $, $n \in \nat $, satisfy  on the intervals $[0, {\tau }_{n})$, the same   uniform  estimates  as in the subsequent Lemma \ref{lem:Hall-MHD_truncated_estimates} for $q=2$. Local existence together with uniform estimates guaranty that the solutions $\Xn $ are global, i.e. defined on the interval $[0,T]$. 
\end{proof}

\medskip
\subsection{A priori estimates} \label{sec:a_priori_est-truncated}

\medskip  \noindent
In the following lemma we will prove some a priori estimates of the solutions of the approximating equation \eqref{eq:Hall-MHD_truncated}.

\medskip
\begin{lemma} \label{lem:Hall-MHD_truncated_estimates}
Let Assumptions \ref{assumption-noise} and \ref{assumption-data}  be satisfied. In particular, we assume that
 $p$ satisfies  \eqref{eqn-p_cond}, i.e.
\[
p \in [2,2+\gamma ),
\] 
where $\gamma $ is given by \eqref{eq:gamma}.
Then the solutions ${(\Xn )}_{n\in \nat }$ of equations \eqref{eq:Hall-MHD_truncated} satisfy the following uniform estimates:
\begin{description}
\item[(i) ] For every $q\in [2,p]$ there exist positive constants ${C}_{1}(p,q)$ and ${C}_{2}(p,q)$
such that
\begin{equation} \label{eq:H_estimate_truncated_p}
\sup_{n \in \nat }\mathbb{E} \Bigl[ \sup_{s\in [0,T] } \nnorm{\Xn (s)}{{\Hmath }}{q} \Bigr] 
\ \le \ {C}_{1}(p,q)
\end{equation}
and
\begin{equation} \label{eq:HV_estimate_truncated}
\sup_{n \in \nat } \mathbb{E} \Bigl[ \int_{0}^{T} \nnorm{\Xn (s)}{{\Hmath }}{q-2} \norm{  \Xn (s)}{}{2} \, ds \Bigr] \; \le \;  {C}_{2}(p,q)  .
\end{equation}
\item[(ii) ] In particular, there exists a positive constant ${C}_{2}(p)$ such that
\begin{equation} \label{eq:V_estimate_truncated}
\sup_{n \in \nat } \mathbb{E} \Bigl[ \int_{0}^{T} \norm{\Xn (s)}{\Vmath }{2} \, ds \Bigr]
\; \le \;  {C}_{2}(p).
\end{equation}
\end{description}
\end{lemma}

\medskip
\begin{proof}[Proof of Lemma \ref{lem:Hall-MHD_truncated_estimates}]
For any $R>0 $ us define the stopping time
\begin{equation*}
\taunR \; := \;  \inf \{ t \in [0,T]: \nnorm{\Xn (t)}{\Hmath }{} \ge R \} .
\end{equation*}
Let us fix $q \in [2,p]$, where $p$ satisfies  condition \eqref{eqn-p_cond}. We apply the It\^{o} formula 
to the function $F$ defined by
\[ F: {\Hmath } \ni x \mapsto  \nnorm{x}{{\Hmath }}{q}\in \mathbb{R}.\]
In the sequel we will often omit the subscript ${\Hmath }$ and write $|\cdot |:= \nnorm{\cdot }{{\Hmath }}{}$. Note that
\begin{align*}
F^\prime(x) (h) \; & = \; d_xF (h) \; = \; q \cdot \nnorm{{x}}{}{q-2} \cdot \dual{x}{h}{\Hmath } , \quad h \in \Hmath ,  \\
 \norm{F^{\prime\prime}(x)}{}{} \; &= \; \norm{d^2_xF }{}{} \; \le \;  q (q-1) \cdot \nnorm{x}{}{q-2} ,
\quad x \in {\Hmath }.
\end{align*}
By the It\^{o} formula
\begin{equation*}
\begin{split}
&\nnorm{{\Xn } (t\wedge \taunR )}{}{q} -  \nnorm{\Pn {\X }_{0}}{}{q}   \\
\;  &=  \; \int_{0}^{t\wedge \taunR} \Bigl\{ q \, \nnorm{\Xn (s)}{}{q-2} \dual{\Xn  (s)}{-\ARiesz \Xn  (s)
- {\MHD }_{n} (\Xn (s)) -{\tHall }_{n}( \Xn (s)) + \fRiesz(s)}{}  \\
& \qquad + \, \,   \frac{q(q-1)}{2}  \nnorm{\Xn (s)}{}{q-2} \norm{ \Gn (s,\Xn  (s)) }{\lhs (\Kmath ,\Hmath )}{2}  \Bigr\} \, ds \\
&  \qquad + q \, \, \int_{0}^{t\wedge \taunR} \nnorm{\Xn (s)}{}{q-2} \dual{\Xn (s)}{\Gn (s,\Xn (s))\, dW(s)}{} .
\end{split}
\end{equation*}
Since $\Xn $ is an $\Hn $-valued process,
\begin{itemize}
\item by \eqref{eq:Acal_Hn_Riesz} and \eqref{eq:A_acal_rel}  we have  $\ilsk{\ARiesz \Xn}{\Xn }{} = \dual{\acal \Xn }{\Xn }{} =  \norm{\Xn }{}{2}$, 
\item by  \eqref{eq:Bn_Hn_Riesz} and \eqref{eq:MHD-map_perp}, $ \ilsk{{\MHD }_{n}\Xn }{\Xn }{} = \dual{\MHD \Xn }{\Xn }{} =0$,
\item  by \eqref{eq:R_Hn_Riesz} and  \eqref{eq:tHall-map_perp}, $\ilsk{{\tHall }_{n}\Xn }{\Xn }{} = \dual{\tHall \Xn }{\Xn }{} =0$,
\item and by \eqref{eq:f_Hn_Riesz}, $\ilsk{\fRiesz (s)}{\Xn }{} = \dual{f(s)}{\Xn }{}$,
\end{itemize} 
we infer that
\begin{equation}
\begin{split}
& \nnorm{{\Xn } (t\wedge \taunR )}{}{q } +q\, \int_{0}^{t\wedge \taunR} \nnorm{\Xn (s)}{}{q-2} \norm{\Xn (s)}{}{2} \, ds \\
&= \; \nnorm{\Pn {\X }_{0}}{}{q } 
+ \int_{0}^{t\wedge \taunR}  \Bigl\{ q \, \nnorm{\Xn (s)}{}{q-2} \dual{\Xn (s)}{ f(s)}{}
\\
& \qquad  + \, \,  \frac{q(q-1)}{2}  \nnorm{\Xn (s)}{}{q-2} \norm{ \Gn (s,\Xn  (s)) }{\lhs (\Kmath ,\Hmath )}{2}  \Bigr\} \, ds    \\
& \qquad  + \, \,  q \, \int_{0}^{t\wedge \taunR } \nnorm{\Xn (s)}{}{q-2} \dual{\Xn (s)}{\Gn (s,\Xn (s)) \, d W(s)}{}  .
\end{split} \label{ineq-01_truncated}
\end{equation} 
Since $\Pn : \Hmath \to {\Hmath }_{n}$ is the $\ilsk{\cdot }{\cdot }{\Hmath }$-orthogonal projection, 
using \eqref{eq:G}, we obtain for all $(s,\X)  \in [0,T] \times {{\Vmath } }_{} $
\begin{equation*}
\begin{split}
\norm{\Gn (s,\X )}{\lhs ({\Kmath },{\Hmath })}{2}
\; \le \; \norm{G(s,\X )}{\lhs ({\Kmath },{\Hmath })}{}
\; & \le \; (2-\eta ) \norm{\X }{}{2} + \lambda \nnorm{\X }{\Hmath }{2} + \varrho  .
\end{split}
\end{equation*}
Using moreover the Young inequality (for numbers), 
$\nnorm{x}{}{q-2}  \le  \bigl( 1-\frac{2}{q} \bigr)  \nnorm{x}{}{q}  +\frac{2}{q}  $   for all $x\ge 0 $, we obtain
\begin{equation}
\begin{split}
& q(q-1) \,  \nnorm{\X }{}{q-2} \norm{ \Gn (s,\X ) }{\lhs (\Kmath ,\Hmath )}{2}
\\
\; & \le \; q(q-1) (2-\eta ) \nnorm{\X }{}{q-2} \norm{\X }{}{2} 
 + {\tilde{K}}_{q}(\lambda ,\varrho ) \, \nnorm{\X }{}{q} 
  + 2(q-1)\varrho    ,
\end{split} \label{eq:tr_est_truncated}
\end{equation}
where ${\tilde{K}}_{q}(\lambda ,\varrho ):=(q-1) \bigl[ q\lambda + (q-2)\varrho  \bigr] $.
Moreover, by condition \bf (H.2) \rm in Assumption \ref{assumption-data}, \eqref{eq:Hall-MHD_V-norm}  and the Young inequality (for numbers) with exponents $2,\frac{2q}{q-2}$ and $q$,
we obtain for every 
${\eps }_{0}>0$ and  for all $s \in [0,t \wedge \taunR ] $
\begin{equation}
\begin{split}
&\nnorm{\Xn (s)}{\Hmath }{q-2} \, \dual{f(s)}{\Xn (s)}{}
\; \le  \;  \nnorm{\Xn (s)}{\Hmath }{q-2} \, \nnorm{f(s)}{{\Vmath }^{\prime }}{} \cdot \norm{\Xn (s)}{\Vmath }{} \\
\; &= \;  \nnorm{\Xn (s)}{\Hmath }{q-2} \, 
{(\nnorm{\Xn (s)}{\Hmath }{2} + \norm{\Xn (s)}{}{2}) }^{\frac{1}{2}} \, \nnorm{f(s)}{{\Vmath }^{\prime }}{}  \\
\;  &\le  \; {\eps }_{0} \, \nnorm{\Xn (s)}{\Hmath }{q-2} \, \norm{\Xn (s)}{}{2}
+ \Bigl( {\eps }_{0} +\frac{1}{2} - \frac{1}{q}  \Bigr) \, \nnorm{\Xn (s)}{\Hmath }{q}
 + \frac{1}{q} \, {(2{\eps }_{0})}^{-\frac{q}{2}} \, \nnorm{f(s)}{{\Vmath }^{\prime }}{q} . 
\end{split}
\label{eq:f_est_truncated}
\end{equation}
Using  estimates \eqref{eq:tr_est_truncated} and \eqref{eq:f_est_truncated} in \eqref{ineq-01_truncated}, we obtain
\begin{equation}
\begin{split}
& \nnorm{\Xn (t\wedge \taunR )}{}{q}   
+ q \, \Bigl[ 1 -{\eps }_{0}  - \frac{1}{2} (q-1) (2- \eta ) \Bigr] \,\int_{0}^{t\wedge \taunR } \nnorm{\Xn (s)}{}{q-2} \norm{\Xn (s)}{}{2} \, ds  \\
& \qquad  \; \le \; \nnorm{\Pn {\X }_{0}}{}{q} 
+ \int_{0}^{t\wedge \taunR }  \Bigl\{  q\Bigl( {\eps }_{0} +\frac{1}{2} - \frac{1}{q}  \Bigr) 
+ \frac{1}{2} {\tilde{K}}_{q}(\lambda ,\varrho ) 
 \Bigr\} \, \nnorm{\Xn (s)}{}{q} \, ds \\
& \qquad \qquad + \, \, (q-1)\varrho  \, t  
 + {(2{\eps }_{0})}^{-\frac{q}{2}} \, \int_{0}^{t \wedge \taunR} \nnorm{f(s)}{{\Vmath }^{\prime } }{q}   \, ds
\\
&\qquad \qquad   + \, \,  q \, \int_{0}^{t \wedge \taunR } \nnorm{\Xn (s)}{}{q-2} \dual{\Xn (s)}{\Gn (s,\Xn (s)) \, d W(s) }{} .
\end{split}  \label{ineq-01'_truncated}
\end{equation}
Let us choose ${\eps }_{0} \in (0,1) $ such that  
$$ \delta=\delta(q,\eta):= 
q \, \bigl[ 1 -  {\eps }_{0}  - \frac{1}{2} (q-1) (2- \eta ) \bigr] >0 ,$$
or equivalently,
$
{\eps }_{0} \; <  \;  1 \wedge \bigl[ 1 -  \frac{1}{2}(q-1) (2-\eta )\bigr] .
$
Notice that under  condition \eqref{eqn-p_cond} such ${\eps }_{0}$ exists. Denote also
\[
K_q(\lambda , \varrho ) 
\; := \;  q\Bigl( {\eps }_{0} +\frac{1}{2} - \frac{1}{q}  \Bigr) 
+ \frac{1}{2} \, {\tilde{K}}_{q}(\lambda ,\varrho ) .
\]
Then, using the fact that $\int_{0}^{t} \nnorm{f(s)}{{\Vmath }^{\prime } }{q} \, ds
\le {t}^{1-\frac{q}{p}} \cdot \bigl( \int_{0}^{t} \nnorm{f(s)}{\Vmath '}{p} \, ds {\bigr) }^{\frac{q}{p}} $, we obtain
\begin{equation}
\begin{split}
&\nnorm{\Xn (t \wedge \taunR )}{}{q} \;  + \; \delta
  \, \int_{0}^{t \wedge \taunR } \nnorm{\Xn (s)}{}{q-2} \norm{ \Xn (s)}{}{2} \, ds \\
\; &\le \; \nnorm{{\X }_{0}}{}{q} 
+ {K}_{q} (\lambda ,\varrho )\int_{0}^{t\wedge \taunR } \, \nnorm{\Xn (s)}{}{q} \, ds
+ \varrho  (q-1)  t  +  {(2{\eps }_{0})}^{-\frac{q}{2}} \, 
 {t}^{1-\frac{q}{p}} \cdot \Bigl( \int_{0}^{t \wedge \taunR } \nnorm{f(s)}{\Vmath '}{p} \, ds {\Bigr) }^{\frac{q}{p}} 
 \\
& \qquad + \;   q \,  \int_{0}^{t \wedge \taunR } \nnorm{\Xn (s)}{}{q-2} \dual{\Xn (s)}{\Gn (s,\Xn (s)) \, d W(s)}{} ,
\quad t \in [0,T].
\end{split}     \label{eq:apriori}
\end{equation}
Since $\Xn $ is the solutions of the approximating equation \eqref{eq:Hall-MHD_truncated},  we infer that the process
\[
{\mcal }_{n}(t) \; := \; \int_{0}^{t\wedge \taunR } \nnorm{\Xn (s)}{}{q-2} \dual{\Xn (s)}{\Gn  (s,\Xn (s)) \, d W(s) }{} ,
 \qquad t \in [0,T],
\]
is a square integrable martingale. Indeed,  by 
\eqref{eq:G} and the fact that $\Pn $ is the orthogonal projection in ${\Hmath }$ we infer that for every $t\in [0,T]$,
\begin{equation*}
\begin{split}
&\int_{0}^{t \wedge \taunR } \norm{ \nnorm{\Xn (s)}{}{q-2} \dual{\Xn (s)}{\Gn  (s,\Xn (s)) }{\Hmath }}{\lhs ({\Kmath },\rzecz )}{2}\, ds \\
& \qquad \le \; \int_{0}^{t} \ind{[0,\taunR )}(s) 
\nnorm{\Xn (s)}{}{q} \norm{G(s,\Xn (s))}{\lhs ({\Kmath },\Hmath )}{2}\, ds
\\
& \qquad  \le \; \int_{0}^{t} \ind{[0,\taunR )}(s) 
\nnorm{\Xn (s)}{}{q} \bigl[ (2- \eta ) \, \norm{\Xn (t)}{}{2}
  + {\lambda }_{0} \nnorm{\Xn (t)}{}{2} + \varrho \bigr] \, ds.
\end{split}
\end{equation*}
Since $\Xn $ is a solution of the approximate equation \eqref{eq:Hall-MHD_truncated} and, by Corollary \ref{cor:H_n-V_m1,m2-norm_equiv}, the norms $\nnorm{\cdot }{\Hmath }{}$ and $\norm{\cdot }{\Vmath }{}$  are equivalent on the subspace $\Hn $,  we infer that
\[ 
\mathbb{E}  \Bigl[ \int_{0}^{t \wedge \taunR } \norm{ \nnorm{\Xn (s)}{}{q-2} \dual{\Xn (s)}{\Gn  (s,\Xn (s)) }{\Hmath }}{\lhs ({\Kmath },\rzecz  )}{2}\, ds \Bigr]  \; < \; \infty , \quad  t\in [0,T].
\]
and thus we infer, as claimed,  that the process $ {\mcal }_{n}$ is a square integrable  martingale. Hence,
$
\mathbb{E}[{\mcal  }_{n} (t) ] \; = \; 0 .
$ 

\medskip  \noindent
By taking expectation in inequality \eqref{eq:apriori} we infer that for all $ t \in [0,T] $:
\begin{equation}
\begin{split}
&\mathbb{E}\bigl[ \, \nnorm{\Xn (t\wedge \taunR )}{}{q} \,\bigr]
+\delta \, \mathbb{E}\, \Bigl[ \int_{0}^{t\wedge \taunR } \nnorm{\Xn (s)}{}{q-2} \norm{\Xn (s)}{}{2} \, ds  \Bigr] 
\; \le   \;   \nnorm{{\X}_0}{}{q} \\
\; &  +
K_q(\lambda ,\varrho ) \int_{0}^{t\wedge \taunR } 
\, \mathbb{E}\,\bigl[ \nnorm{\Xn (s)}{}{q} \bigr] \, ds 
+ \;  \varrho  (q-1) t  + {(2{\eps }_{0})}^{-\frac{q}{2}} \, 
 {t}^{1-\frac{q}{p}} \cdot \Bigl( \int_{0}^{t\wedge \taunR } \nnorm{f(s)}{\Vmath '}{p} \, ds {\Bigr) }^{\frac{q}{p}}  .
  \end{split}   \label{eq:apriori'}
\end{equation}
In particular,
\begin{equation*}
\begin{split}
&\mathbb{E}\bigl[ \, \nnorm{\Xn (t\wedge \taunR )}{}{q} \,\bigr]
\; \le \;  \nnorm{{\X}_0}{}{q}  +
K_q(\lambda ,\varrho ) \int_{0}^{t\wedge \taunR } 
\, \mathbb{E}\,\bigl[ \nnorm{\Xn (s)}{}{q} \bigr] \, ds 
 \\
&+ \;  \varrho  (q-1) T  + {(2{\eps }_{0})}^{-\frac{q}{2}} \,
 {T}^{1-\frac{q}{p}} \cdot \norm{f}{{L}^{p}(0,T;\Vmath ')}{\frac{q}{p}}   ,   \quad t \in [0,T].
  \end{split}   
\end{equation*}
Using the Gronwall lemma and passing to the limit as $R \to \infty $, we infer that
\begin{equation} \label{eq:app_H_est}
\sup_{n \in \nat } \sup_{t\in [0,T]} \mathbb{E}\, \bigl[ \nnorm{\Xn (t)}{}{q} \bigr]
\; \le \;  {\tilde{C}}_{1}(p,q) 
\end{equation}
for some constant ${\tilde{C}}_{1}(p,q)={\tilde{C}}_{1}(p,q,T,\eta ,\lambda ,\varrho , \nnorm{{\X }_0}{}{},  \norm{f}{{L}^{p}(0,T;{{\Vmath } }^{\prime })}{} ) >0$. Hence 
by the Fubini theorem
\[
\sup_{n \in \nat }  \mathbb{E} \biggl[ \int_{0}^{T}  \nnorm{\Xn (s)}{}{q}  \, ds \biggr]
\; \le \;  T  {\tilde{C}}_{1}(p,q) .
\]
Using this bound in \eqref{eq:apriori'} we also obtain
\begin{equation} \label{eq:app_HV_est}
\sup_{n \in \nat }  \mathbb{E}\, \biggl[ \int_{0}^{T} \nnorm{\Xn (s)}{}{q-2} \norm{\Xn (s)}{}{2} \, ds  \biggr]
\;  \le \;  {C}_{2}(p,q)
\end{equation}
for a new constant ${C}_{2}(p,q) $ dependent also on 
$T,\eta ,\lambda ,\varrho ,\nnorm{{\X }_0 }{}{}$ and $\norm{f}{{L}^{p}(0,T;{{\Vmath } }^{\prime })}{}$.
This completes the proof of estimates \eqref{eq:HV_estimate_truncated}.
Putting $q:=2$, by \eqref{eq:Hall-MHD_V-norm} we infer that \eqref{eq:V_estimate_truncated} holds.

\medskip  
\noindent
Let us move to the proof of estimate \eqref{eq:H_estimate_truncated_p}.
By the Burkholder-Davis-Gundy inequality, see eg. \cite{DaPrato+Zabczyk_Erg} or \cite{Revuz+Yor'99}, and the Schwarz inequality, there exists a constant ${c}_{q}$ such that 
\begin{equation}
\begin{split}
&  \mathbb{E}\,\biggl[ \sup_{0 \le s \le T\wedge \taunR }
 \biggl| \int_{0}^{s} q \, \nnorm{\Xn (\sigma )}{}{q-2} \dual{\Xn (\sigma )}{\Gn  (\sigma , \Xn (\sigma ) ) \, d W(\sigma ) }{} \biggr| \biggr]  \\
\; & \le \; {c}_{q} \cdot \mathbb{E}\,\biggl[
 {\biggl( \int_{0}^{T\wedge \taunR  } \, \nnorm{\Xn (\sigma )}{}{2q-2} \cdot
 \norm{ \Gn  (\sigma , \Xn (\sigma ) ) }{\lhs (\Kmath ,\Hmath )}{2} \, d\sigma
  \biggr) }^{\frac{1}{2}} \biggr]  \\
\;  & \le \;  {c}_{q} \cdot \mathbb{E}\,\biggl[ 
\sup_{0\le \sigma \le T\wedge \taunR  } \nnorm{\Xn (\sigma )}{}{\frac{q}{2}}
{\biggl( \int_{0}^{T \wedge \taunR } \,  \nnorm{\Xn (\sigma )}{}{q-2} \cdot
 \norm{  G (\sigma , \Xn (\sigma ) ) }{\lhs (\Kmath ,\Hmath )}{2}  \, d\sigma
  \biggr) }^{\frac{1}{2}} \biggr]   \\
\; &  \le \; \mathbb{E}\,\Bigl[ \frac{1}{2} \sup_{0\le s \le T\wedge \taunR  } \nnorm{\Xn (s )}{}{q} 
   + \frac{1}{2}{c}_{q}^{2} \, \int_{0}^{T\wedge \taunR } \,  \nnorm{\Xn (\sigma )}{}{q-2} \cdot
 \norm{  G (\sigma , \Xn (\sigma ) ) }{\lhs ({\Kmath },{\Hmath })}{2}
   \, d\sigma \Bigr] .
\end{split} \label{eq:BDG_ineq_truncated}
\end{equation}
Using \eqref{eq:BDG_ineq_truncated} and  \eqref{eq:G} in \eqref{eq:apriori}, we obtain
\begin{equation*}
\begin{split}
& \mathbb{E}\,\Bigl[ \sup_{0 \le s \le T\wedge \taunR  } \nnorm{\Xn (s)}{}{q} \Bigr]
   \; \le  \;    \nnorm{{\X }_0}{}{q}  
+  \biggl( K_q(\lambda ,\varrho )
+ \frac{1}{2} {c}_{q}^{2} \Bigl\{ \lambda  + \varrho  \Bigl( 1-\frac{2}{q}\Bigr) \Bigr\} \biggr)
\mathbb{E}\biggl[  \int_{0}^{T\wedge \taunR } \nnorm{\Xn (s)}{}{q} \, ds  \biggr] \\
&  +  \Bigl\{ (q-1)\varrho  +{c}_{q}^{2} \frac{\varrho }{q} \Bigr\} \, t
  + {(2{\eps }_{0})}^{-\frac{q}{2}} \, 
 {T}^{1-\frac{q}{p}} \cdot \norm{f}{{L}^{p}(0,T;\Vmath ')}{\frac{q}{p}}   
 \\
& +   \frac{1}{2} \mathbb{E}\,\Bigl[ \sup_{0\le s \le T \wedge \taunR  } \nnorm{\Xn (s )}{}{q} \Bigr]
+ \frac{1}{2}{c}_{q}^{2} (2-\eta ) \mathbb{E}\,\biggl[ \int_{0}^{T \wedge \taunR } \nnorm{\Xn (\sigma )}{}{q} 
 \norm{\Xn (\sigma )}{}{2}\, d \sigma \biggr]  .
\end{split}
\end{equation*}
Hence by inequalities \eqref{eq:app_H_est} and \eqref{eq:app_HV_est}  we infer that
there exists a constant ${C}_{1}(p,q)  = {C}_{1}(p,q,T,\eta ,\lambda ,\varrho ,\nnorm{{\X }_{0}}{}{}, \norm{f}{{L}^{p}(0,T;V')}{})$ such that for  every $n \in \nat $
\begin{equation*}
\mathbb{E}\,\Bigl[ \sup_{0 \le s \le T \wedge \taunR } \nnorm{\Xn (s)}{}{q} \Bigr] \; \le \;  {C}_{1}(p,q)  ,
\end{equation*}
Passing to the limit as $R\to \infty $, we infer that
\begin{equation*}
\sup_{n \in \nat }\mathbb{E}\,\Bigl[ \sup_{0 \le s \le T \wedge \taunR } \nnorm{\Xn (s)}{}{q} \Bigr] \; \le \;  {C}_{1}(p,q) ,
\end{equation*}
This completes the proof of estimate \eqref{eq:H_estimate_truncated_p} and of the lemma.
\end{proof}

\medskip
\subsection{Auxiliary remarks}

\medskip  \noindent
Let us consider the sequence ${(\Xn )}_{n\in \nat }$ of  solutions of equations \eqref{eq:Hall-MHD_truncated}. These solutions satisfy uniform estimates stated in Lemma \ref{lem:Hall-MHD_truncated_estimates}. Since we consider the Hall-MHD equations on ${\rzecz }^{3}$, the continuous embedding
\[
\Vmath \; \hookrightarrow \, \Hmath 
\] 
is not compact. However, using Lemma 2.5 from \cite{Holly+Wiciak'1995} (see \cite[Lemma C.1]{Brze+EM'13}) we can find a separable Hilbert space  $\Umath $ such that
$$ 
\Umath  \; \subset \;  \Vast \;  \subset \; \Vtest  \; \subset \; \Vmath \; \subset \; \Hmath ,
$$ 
where $\Vast = {\Vmath }_{m}$ for fixed $m>\frac{5}{2}$,
and  the embedding
\[
\Umath \; \hookrightarrow \; \Vast 
\]
is compact, see Appendix \ref{sec:aux_funct.anal}.
Using this structure, where the space $\Umath $ is of crucial importance, we can prove appropriate tightness criterion, see Appendix \ref{sec:comp-tight}, which we apply to prove the tightness of the sequence of laws of ${(\Xn )}_{n\in \nat }$ in Section \ref{sec:tightness}.  
The subsequent reasoning is a preparation for Section \ref{sec:tightness}.

\medskip  \noindent
Considering also the dual spaces, and identifying $\Hmath $ with its dual $\Hmath '$, we have the following system 
\begin{equation*}
\Umath \; \subset \;  \Vast \; \subset \Vtest  \; \subset \; \Vmath \; \subset \; \Hmath \; \cong \; \Hmath '
\; \to \; \Vprime \; \to \; \Vastprime \; \to \; \Uprime .  
\end{equation*}
Let us recall that the map 
\[ 
  \Pn : \Hmath  \to \Hn
\]
defined by \eqref{eq:P_n} is  $\ilsk{\cdot }{\cdot }{\Hmath }$-orthogonal projection on $(\Hn ,\ilsk{\cdot }{\cdot }{\Hmath })$. Further properties of the map $\Pn $ are stated in Corollary \ref{cor:P_n-pointwise_conv}.

\medskip  
\begin{remark} \label{rem:Pn_dual} \rm \
From Corollary \ref{cor:P_n-pointwise_conv}, it follows that we may consider the adjoint operators $\Pn '$ in appropriate spaces.
\begin{description}
\item[(1) ] Since  $\Pn \in \lcal ( \Vmath , \Vmath )$,  the adjoint operator $\Pn ' \in \lcal (\Vprime ,\Vprime )$ by definition  satisfies
\begin{equation*} 
\ddual{\Vprime}{\xi }{\Pn \varphi}{\Vmath }
 \,\, = \,\, \ddual{\Vprime}{{\Pn '}\xi }{ \varphi}{\Vmath } , \qquad
\mbox{ for all } \quad \xi \in \Vprime , \quad \varphi \in \Vmath   . 
\end{equation*} 
\item[(2) ] Since  $\Pn \in \lcal (\Vast , \Vmath) $, the adjoint operator $\Pn ' \in \lcal (\Vprime , \Vastprime )$
satisfies
\begin{equation*} 
\ddual{\Vprime}{\xi }{\Pn \varphi}{\Vmath }
 \,\, = \,\, \ddual{\Vastprime}{{\Pn '}\xi }{ \varphi}{\Vast } , \qquad
\mbox{ for all } \quad \xi \in \Vprime , \quad \varphi \in \Vast   .    
\end{equation*}
\item[(3) ] Since  $\Pn \in \lcal (\Vast ,\Vast )$,  the adjoint operator $\Pn ' \in \lcal (\Vastprime , \Vastprime )$ satisfies
\begin{equation*} 
\ddual{\Vastprime}{\xi }{\Pn \varphi}{\Vast }
 \,\, = \,\, \ddual{\Vastprime}{{\Pn '}\xi }{ \varphi}{\Vast } , \qquad
\mbox{ for all } \quad \xi \in \Vastprime , \quad \varphi \in \Vast   .    
\end{equation*}
\end{description}
\end{remark}

\medskip  \noindent
Using Definition \ref{def:trunctated_weak} and Remark \ref{rem:Pn_dual},
we will show that the approximating equation can be rewritten as an equation in the space $\Vastprime $, and by the injection $\Vastprime \to \Uprime $ - also as an equation in $\Uprime $.

\medskip  
\begin{remark} \label{rem:truncated_U'} \rm 
\begin{description}
\item[(i) ] If the $\Hn$-valued process $\Xn $ satisfies identity \eqref{eq:Hall-MHD_truncated_weak}, then
in particular, for all $t \in [0,T]$ and $\varphi \in \Vast   $ we have $\Pn \varphi \in \Hn$ and  
\begin{equation}
\begin{split}
&\ilsk{\Xn (t)}{\Pn \varphi}{\Hmath }
 + \int_{0}^{t} \ddual{\Vprime }{ \acal  \Xn (s)}{\Pn  \varphi }{\Vmath } \, ds
+ \int_{0}^{t} \ddual{\Vastprime}{ \MHD  (\Xn (s))}{\Pn \varphi}{\Vast}  ds \\
& \qquad \quad + \int_{0}^{t} \ddual{\Vastprime }{ \tHall   (\Xn (s))}{\Pn  \varphi }{\Vast } \, ds
  \\
& \quad  \; = \; \ilsk{{\X }_{0}}{\Pn \varphi}{\Hmath }
+ \int_{0}^{t} \ddual{\Vprime }{f(s)}{\Pn \varphi}{\Vmath }\, ds
 + \Ilsk{\int_{0}^{t}G(s,\Xn (s))\,  dW(s)}{\Pn \varphi}{\Hmath  }.
\end{split} \label{eq:truncated_weak_Pn}
\end{equation}
We used also the properties of the maps $\MHD $ and $\tHall $ stated in Lemmas \ref{lem:MHD-term_properties}  and \ref{lem:tHall-term_properties}, respectively.
Since $\Pn : \Hmath \to \Hn $ is an $\ilsk{\cdot }{\cdot }{\Hmath }$-orthogonal projection, 
\[
\ilsk{\Xn (t)}{ \Pn \varphi}{\Hmath } = \ilsk{\Pn \Xn (t)}{ \varphi}{\Hmath } = \ilsk{\Xn (t)}{ \varphi}{\Hmath }.
\]

\noindent
\item[(ii)] Using the operators $\Pn '$ introduced in Remark \ref{rem:Pn_dual}, \eqref{eq:truncated_weak_Pn} can be written in the form
\begin{equation}
\begin{split}
&\ilsk{\Xn (t)}{ \varphi}{\Hmath } 
 + \int_{0}^{t} \ddual{\Vprime }{ \Pn ' \acal    \Xn (s)}{  \varphi }{\Vmath  } \, ds
+ \int_{0}^{t} \ddual{\Vastprime }{ \Pn ' \MHD  (\Xn (s))}{ \varphi}{\Vast }  ds  \\
& \qquad  +  \int_{0}^{t} \ddual{\Vastprime }{ \Pn ' \tHall (\Xn (s))}{  \varphi }{\Vast  } \, ds
  \\
&  \; = \; \ilsk{\Pn {\X }_{0}}{ \varphi}{\Hmath }
+ \int_{0}^{t} \ddual{\Vprime }{\Pn ' f(s)}{ \varphi}{\Vmath }\, ds
+ \Ilsk{ \int_{0}^{t} \Pn \underline{\circ }G(s,\Xn (s)) \, dW(s)}{\varphi}{\Hmath }  , \quad t \in [0,T] .
\end{split}   \label{eq:truncated_weak_Pn'}
\end{equation}
where
\begin{equation}
\Pn \underline{\circ } G \; := \; \Gn ,
\label{eq:Pn_G}
\end{equation}
and $\Gn $ is the map defined by \eqref{eq:G_n}.
\item[(iii) ]
Let for $\varphi \in \Vast $ the map ${\varphi }^{\ast \ast } $ be defined by
\[
{\varphi }^{\ast \ast } (\xi ) \; := \; \xi (\varphi ) , \qquad \xi \in \Vastprime .
\] 
Note that since 
\[
|\ilsk{\Xn (t)}{\varphi }{\Hmath }| \; \le \; \nnorm{\Xn (t)}{\Hmath }{} \cdot \nnorm{\varphi }{\Hmath }{} 
\; \le \; \nnorm{\Xn (t)}{\Hmath }{} \cdot \norm{\varphi }{\Vast }{} ,
\]
we infer that
\[
\Bigl[ \Vast \ni \varphi \mapsto \ilsk{\Xn (t)}{\varphi }{\Hmath } \in \rzecz  \Bigr] \in \Vastprime . 
\]
In particular,
\[
\ilsk{\Xn (t)}{\varphi }{\Hmath } \; = \; {\varphi }^{\ast \ast } \bigl( \ilsk{\Xn (t)}{\cdot }{\Hmath } \bigr) 
\]
Thus \eqref{eq:truncated_weak_Pn'} can be rewritten in the form 
\[
\begin{split}
&{\varphi }^{\ast \ast } \bigl( \ilsk{\Xn (t)}{\cdot }{\Hmath } \bigr) 
+ {\varphi }^{\ast \ast } \biggl( \int_{0}^{t}\bigl[ \Pn ' \acal \un (s)
 + \Pn ' \MHD   \bigl( \Xn (s) \bigr) + \Pn ' \tHall  (\Xn (s)) \bigr] \, ds  \biggr)\\
\; &= \;  {\varphi }^{\ast \ast } \bigl( \ilsk{\Pn {\X }_{0} }{\cdot }{\Hmath } \bigr) 
+  {\varphi }^{\ast \ast } \biggl( \int_{0}^{t}\Pn ' f (s)\, ds  \biggr)
+  {\varphi }^{\ast \ast }\biggl(  \int_{0}^{t} \Pn \underline{\circ}G(s,\Xn (s)) \, dW(s) \biggr) .
\end{split}
\]
\item[(iv) ] Using the identification $\Hmath \cong \Hmath '$ and the fact that 
$\Hmath ' \hookrightarrow \Vmath ' \hookrightarrow \Vastprime $,
we identify $\Xn (t)$ with the functional induced by $\Xn (t)$ on the space $\Vast $. Since the family $\{ {\varphi }^{\ast \ast }; \varphi \in \Vast  \} $ separates elements of $\Vastprime $,  we infer that $\Xn (t)$ satisfies the following equation
\begin{equation*} 
\begin{split}
& \Xn (t) + \int_{0}^{t}\bigl[ \Pn ' \acal \un (s)
 + \Pn ' \MHD   \bigl( \Xn (s) \bigr) + \Pn ' \tHall  (\Xn (s)) \bigr] \, ds  
  \\ 
&  = \,\,  \Pn {\X }_{0} + \int_{0}^{t} \Pn ' f (s)\, ds 
+ \int_{0}^{t}\Pn \underline{\circ}G (s,\Xn (s)) \, dW(s) ,   
\qquad  t \in [0,T]  , 
\end{split} 
\end{equation*}  
where $\Pn \underline{\circ}G $  is defined by \eqref{eq:Pn_G}.
\end{description}
\end{remark}

\medskip
\section{Existence of a martingale solution. Proof of Theorem \ref{th:mart-sol_existence}} \label{sec:existence}

\medskip  \noindent
Having constructed the sequence $(\Xn )$ of  solutions of equations \eqref{eq:Hall-MHD_truncated} the idea of further steps of the proof is similar to \cite{Brze+EM'13}.
First we will prove that the sequence of laws of $\Xn $, $n \in \nat $, form a tight sequence of probability measures on appropriate functional space. Using the Jakubowski version of the Skorokhod theorem we construct  new stochastic bases an new processes. The last step is passing to the limit.

\medskip
\subsection{Tightness} \label{sec:tightness}

\medskip  \noindent
Let us consider the sequence ${(\Xn )}_{n\in \nat }$ of  solutions of equations \eqref{eq:Hall-MHD_truncated}. Using the tightness criterion stated in  Corollary \ref{cor:tigthness_criterion} in Appendix \ref{sec:comp-tight},
 we will prove that the sequence of laws of $\Xn $ is tight in the space $\zcal $ defined by \eqref{eq:Z_T}, i.e.
\begin{equation*}
\mathcal{Z} \; := \; {L}_{w}^{2}(0,T;\Vmath )  
\cap {L}^{2}(0,T;{\Hmath }_{loc})  \cap \ccal ([0,T];{\Hmath }_{w}) \cap \ccal ([0,T]; {\Umath }^{\prime }), 
\end{equation*}
equipped with the Borel $\sigma $-field  $\sigma(\tcal )$, see Definition \ref{def:space_Z_T}.

\medskip
\begin{lemma} \label{lem:X_n-tightness}
The set of probability measures $\{ \Law (\Xn ), n \in \nat   \} $ is tight on the space 

\noindent
$(\zcal ,\sigma(\tcal ))$. 
\end{lemma}

\medskip
\begin{proof}
We apply Corollary \ref{cor:tigthness_criterion}.
Let us note  that due to estimates  \eqref{eq:H_estimate_truncated_p} and \eqref{eq:V_estimate_truncated}, conditions \eqref{cond-a} and \eqref{cond-b}  of Corollary \ref{cor:tigthness_criterion} are satisfied. Thus, it is sufficient to  prove  that the sequence $(\Xn {)}_{n \in \nat }$ satisfies the Aldous condition \textbf{[A]}. 
By Lemma \ref{lem:Aldous_criterion} it is sufficient to proof the condition \textbf{[A']}.
Let ${(\taun )}_{n \in \nat} $ be a sequence of stopping times taking values in $[0,T]$.
By Remark \ref{rem:truncated_U'} (iv),    we have
\begin{equation*}
\begin{split}
\Xn (t) 
\; &= \; \Pn {\X }_{0}  - \int_{0}^{t} \Pn ' \acal  \Xn (s) \, ds 
    - \int_{0}^{t} \Pn '\MHD \bigl( \Xn (s) \bigr) \, ds
 - \int_{0}^{t} \Pn ' \tHall   (\Xn (s)) \, ds  
\\  
& \qquad   + \int_{0}^{t} \Pn ' {f}_{n}(s) \, ds 
+  \int_{0}^{t} \Pn \underline{\circ }G (s,\Xn (s)) \, dW(s)    \\
\; & =: \;   \Jn{1} + \Jn{2}(t) + \Jn{3}(t) + \Jn{4}(t) + \Jn{5}(t) + \Jn{6}(t), \qquad t \in [0,T].
\end{split}
\end{equation*}
Let us choose and $\theta  >0 $.  It is sufficient to show that each sequence $\Jn{i}$ of processes, $i=1,\cdots ,6$, satisfies the sufficient condition \textbf{[A']} from Lemma \ref{lem:Aldous_criterion}.
Since the term $\Jn{1}$  is  constant in time, it satisfies this condition.
 In fact, we will check that the terms $\Jn{2},  \Jn{5}, \Jn{6}$ satisfy condition
 \textbf{[A']} from Lemma \ref{lem:Aldous_criterion} in the space $E={{\Vmath } }^{\prime }$ and the terms $\Jn{3}, \Jn{4}$ satisfy this condition in  $E=\Vastprime $. Since the embeddings 
 ${{\Vmath } }^{\prime } \subset {\Umath }^{\prime } $ and
$ \Vastprime \subset {\Umath }^{\prime }$   are continuous, we infer that
\textbf{[A']}  holds in the space $E={\Umath }^{\prime }$, as well.

\medskip \noindent
\textbf{ Ad } ${\Jn{2}}$. \rm  By Remark \ref{rem:Acal-term_properties},
 the linear operator $\acal :{\Vmath }  \to {{\Vmath } }^{\prime }$ is bounded, 
and by Corollary \ref{cor:P_n-pointwise_conv},  $\sup_{n\in \nat } \nnorm{\Pn }{\lcal (\Vmath ,\Vmath )}{} < \infty  $. Using the H\"older inequality and \eqref{eq:V_estimate_truncated}, we obtain
\begin{equation*}
\begin{split}
& \mathbb{E}\,\bigl[ \nnorm{ \Jn{2} (\taun + \theta ) - \Jn{2}(\taun )  }{{{\Vmath } }^{\prime }}{}  \bigr]
\; \le \; \nnorm{\Pn '}{\lcal (\Vprime ,\Vprime )}{} \, \mathbb{E}\,\Bigl[  \int_{\taun }^{\taun + \theta }
\nnorm{\acal  \Xn (s) }{{\Vmath }^{\prime } }{}  \, ds \biggr]
 \\
& \; \le \;  \nnorm{\Pn }{\lcal (\Vmath ,\Vmath )}{} \, {\theta }^{\frac{1}{2}} \Bigl( \mathbb{E}\,\Bigl[  \int_{0 }^{T }
 \norm{  \Xn (s) }{}{2}  \, ds \Bigr] {\Bigr) }^{\frac{1}{2}}
\; \le \;   c_2 \cdot {\theta }^{\frac{1}{2}} ,  
\end{split} 
\end{equation*}
where $c_2 = {C}_{2}^{\frac{1}{2}}(p) \cdot \sup_{n\in \nat } \nnorm{\Pn }{\lcal (\Vmath ,\Vmath )}{} <\infty  $.


\medskip 
\noindent
\textbf{Ad } ${\Jn{3}}$. \rm 
By \eqref{eq:MHD-map_est-H-H} in Lemma \ref{lem:MHD-term_properties},
 $\MHD : {\Hmath } \times {\Hmath } \to {\Vmath }_{\ast }^{\prime } $ is bilinear and continuous (and hence bounded so that   the norm $\| \MHD \| $ of $\MHD : {\Hmath } \times {\Hmath } \to {{\Vmath }}_{\ast  }^{\prime }$ is finite), and  by Corollary \ref{cor:P_n-pointwise_conv},  $\sup_{n\in \nat } \nnorm{\Pn }{\lcal (\Vast ,\Vast )}{} < \infty  $.  Then by \eqref{eq:H_estimate_truncated_p} we have the following estimates
\begin{equation*}
\begin{split}
&\mathbb{E}\,\bigl[ \nnorm{ \Jn{3}(\taun +\theta ) - \Jn{3}(\taun)}{{{\Vmath }}_{\ast }^{\prime }}{}  \bigr]
\; = \; \nnorm{\Pn '}{\lcal (\Vastprime ,\Vastprime )}{} \, \mathbb{E}\,\Bigl[  \Nnorm{ \int_{\taun }^{\taun + \theta }
\MHD \bigl(\Xn (r) \bigr) \, dr }{{{\Vmath }}_{\ast }^{\prime }}{} \Bigr]   \\
\; & \le \;  \nnorm{\Pn }{\lcal (\Vast ,\Vast )}{} \, \mathbb{E}\,\Bigl[  \int_{\taun }^{\taun + \theta }
\nnorm{ \MHD ( \Xn (r)  )  }{{{\Vmath }}_{\ast }^{\prime }}{} \, dr \Bigr]
\; \le \;  \nnorm{\Pn }{\lcal (\Vast ,\Vast )}{} \, \| \MHD \| \, \mathbb{E}\,\biggl[  \int_{\taun }^{\taun + \theta }  \nnorm{\Xn (r)}{{\Hmath }}{2}   \, dr \biggr]    \\
\; & \le \; \nnorm{\Pn }{\lcal (\Vast ,\Vast )}{} \, \| \MHD \|  \cdot  \mathbb{E}\,\bigl[ \sup_{r \in [0,T]} \nnorm{\Xn (r)}{{\Hmath }}{2}\bigr] \cdot \theta
\; \le \;    c_3 \, \theta ,   
\end{split} 
\end{equation*}
where $c_3=  \sup_{n\in \nat }\nnorm{\Pn }{\lcal (\Vast ,\Vast )}{} \,  \| \MHD \| \,  {C}_{1}(p,2) <\infty  $.

\medskip  \noindent
\textbf{Ad } ${\Jn{4}}$. \rm 
By \eqref{eq:tHall-map_est-H-V} in Lemma \ref{lem:tHall-term_properties} $\tHall : {\Hmath } \times {\Vmath } \to \Vastprime $ is bilinear and continuous (and hence bounded so that   the norm $\| \tHall \| $ of $\tHall : {\Hmath } \times {\Vmath } \to \Vastprime $ is finite) and by Corollary \ref{cor:P_n-pointwise_conv},  $\sup_{n\in \nat } \nnorm{\Pn }{\lcal (\Vast ,\Vast )}{} < \infty  $.  Then by \eqref{eq:H_estimate_truncated_p} and \eqref{eq:V_estimate_truncated} we have the following estimates
\begin{equation*}
\begin{split}
&  \mathbb{E}\,\bigl[ {| \Jn{4} (\taun + \theta ) - \Jn{3}(\taun) |}_{\Vastprime }  \bigr]
\; \le \;  \nnorm{\Pn '}{\lcal (\Vastprime ,\Vastprime )}{} \, \mathbb{E}\,\Bigl[  \int_{\taun }^{\taun + \theta } \nnorm{\tHall \bigl( \Xn (r)  \bigr) }{\Vastprime }{} \, dr \Bigr]   \\
\; & \le \;  \nnorm{\Pn }{\lcal (\Vast ,\Vast )}{} \norm{\tHall }{}{} \, \mathbb{E}\Bigl[ \int_{\taun }^{\taun + \theta }   \nnorm{ \Xn (r)}{\Hmath }{} \, \norm{  \Xn (r)}{}{}  \, dr \Bigr]  \\
\; &\le \;  \nnorm{\Pn }{\lcal (\Vast ,\Vast )}{} \norm{\tHall }{}{} \,  \biggl( \mathbb{E} \Bigl[   \sup_{r \in [\taun,\taun + \theta]}  \nnorm{\Xn (r)}{\Hmath }{2} \Bigr] {\biggr) }^{\frac{1}{2}}
\biggl( \mathbb{E} \Bigl[ \int_{\taun }^{\taun + \theta }
 \norm{\Xn (r)}{}{2} \, dr \Bigr] {\biggr) }^{\frac{1}{2}}  {\theta }^{\frac{1}{2}}   \\
\; & \le \;  \nnorm{\Pn }{\lcal (\Vast ,\Vast )}{} \norm{\tHall }{}{} \,   \biggl( \mathbb{E} \Bigl[   \sup_{r \in [0,T]}  \nnorm{\Xn (r)}{{\Hmath } }{2} \Bigr]  {\biggr) }^{\frac{1}{2}}
\biggl( \mathbb{E} \Bigl[ \int_{0}^{T}
       \norm{  \Xn (r)}{}{2} \, dr \Bigr] {\biggr) }^{\frac{1}{2}} {\theta }^{\frac{1}{2}}
\;  \le \;   c_4  {\theta }^{\frac{1}{2}}  , 
\end{split}
\end{equation*}
where $c_4 =  \sup_{n\in \nat }\nnorm{\Pn }{\lcal (\Vast ,\Vast )}{} \norm{\tHall }{}{} \,   [{C}_{1}(p,2)]^{\frac{1}{2}} [{C}_{2}(p)]^{\frac{1}{2}} < \infty $.

\medskip \noindent
\textbf{Ad } ${\Jn{5}}$. \rm 
Since  $f \in {L}^{p}(0,T;{{\Vmath } }^{\prime })$ and by Corollary \ref{cor:P_n-pointwise_conv},  $\sup_{n\in \nat } \nnorm{\Pn }{\lcal (\Vmath ,\Vmath )}{} < \infty $,  using the H\"{o}lder inequality,  we have
\begin{equation*}
\begin{split}
& \mathbb{E} \,\bigl[ \nnorm{\Jn{5} (\taun + \theta ) - \Jn{5}(\taun)}{{{\Vmath } }^{\prime }}{}  \bigr]
\; \le \; \nnorm{\Pn '}{\lcal (\Vprime ,\Vprime )}{}  \, \mathbb{E}\,\Bigl[ \Nnorm{ \int_{\taun }^{\taun +\theta } f(s) \, ds }{{{\Vmath } }^{\prime }}{} \Bigr]
   \\
\; & \le \; \nnorm{\Pn }{\lcal (\Vmath ,\Vmath )}{} \, {\theta }^{\frac{p-1}{p}}   
\Bigl( \mathbb{E} \,\Bigl[  \int_{0 }^{T }
\nnorm{f(s)}{{{\Vmath } }^{\prime }}{p}  \, ds \Bigr] {\Bigr) }^{\frac{1}{p}}
\; \le \; {c}_{5} \cdot {\theta }^{\frac{p-1}{p}}, 
\end{split}  
\end{equation*}
where ${c}_{5}:= \sup_{n\in \nat } \nnorm{\Pn }{\lcal (\Vmath ,\Vmath )}{}   \norm{f}{{L}^{p}(0,T;{{\Vmath } }^{\prime })}{} < \infty $.

\medskip \noindent
\textbf{Ad } ${\Jn{6}}$. \rm Since $\Pn \underline{\circ }G = \Gn $ (see \eqref{eq:G_n} and \eqref{eq:Pn_G}), by  \eqref{eq:G*} and inequality \eqref{eq:H_estimate_truncated_p}, we obtain the following inequalities
\begin{equation*}
\begin{split}
& \mathbb{E} \,\bigl[ \nnorm{\Jn{6} (\taun + \theta )- \Jn{6}(\taun ) }{{{\Vmath } }^{\prime }}{2}  \bigr] 
\; = \;  \mathbb{E} \,\Bigl[ \Nnorm{ \int_{\taun }^{\taun +\theta } \Gn (s,\Xn (s)) \, dW(s) }{{{\Vmath } }^{\prime }}{2}  \Bigr]  \\
\; &= \; \e \Bigl[ \Bigl| \sup_{\psi \in \Vmath , \norm{\psi }{\Vmath }{} \le 1 } 
\Ddual{\Vmath '}{\int_{\taun }^{\taun +\theta } \Gn (s,\Xn (s)) \, dW(s)}{\psi }{\Vmath } {\Bigr| }^{2} \Bigr]  \\
\; &= \; \mathbb{E} \Bigl[ \sup_{\psi \in \Vmath , \norm{\psi }{\Vmath }{} \le 1 }  
\int_{\taun }^{\taun + \theta } 
\norm{ \ddual{\Vmath '}{\Gn (s,\Xn (s))(\cdot )}{\psi }{\Vmath } }{\lhs (\Kmath  ,\rzecz )}{2} \, ds \Bigr] 
\\
\; &= \; \mathbb{E} \Bigl[ \sup_{\psi \in \Vmath , \norm{\psi }{\Vmath }{} \le 1 } 
 \int_{\taun }^{\taun + \theta } \sup_{y\in \Kmath , \norm{y}{\Kmath }{} \le 1 }
 \bigl| \ddual{\Vmath '}{\Gn (s,\Xn (s))(y )}{\psi }{\Vmath } {\bigr| }^{2} \, ds \Bigr] 
 \\
\; &= \; \mathbb{E} \Bigl[ \sup_{\psi \in \Vmath , \norm{\psi }{\Vmath }{} \le 1 } 
 \int_{\taun }^{\taun + \theta } \sup_{y\in \Kmath , \norm{y}{\Kmath }{} \le 1 }
 \bigl| \ilsk{\Pn (G (s,\Xn (s))(y ))}{ \psi }{\Hmath } {\bigr| }^{2} \, ds \Bigr] 
\\
\; &= \; \mathbb{E} \Bigl[ \sup_{\psi \in \Vmath , \norm{\psi }{\Vmath }{} \le 1 } 
 \int_{\taun }^{\taun + \theta } \sup_{y\in \Kmath , \norm{y}{\Kmath }{} \le 1 }
 \bigl| \ilsk{G (s,\Xn (s))(y )}{\Pn \psi }{\Hmath } {\bigr| }^{2} \, ds \Bigr] 
\\
\; &\le \; \mathbb{E} \Bigl[ \int_{\taun }^{\taun + \theta } 
\sup_{\Psi \in \Vmath , \norm{\psi }{\Vmath }{} \le 1 } 
  \sup_{y\in \Kmath , \norm{y}{\Kmath }{} \le 1 }
 \bigl| \ddual{\Vmath '}{g(s,\Xn (s))(y )}{\Pn \psi }{\Vmath } {\bigr| }^{2} \, ds \Bigr]  
  \\
\; & \le \;  C \cdot \sup_{\psi \in \Vmath , \norm{\psi }{\Vmath }{} \le 1 }  \norm{\Pn \psi }{\Vmath }{}  \cdot \mathbb{E} \Bigl[ \int_{\taun }^{\taun + \theta } 
( 1 + \nnorm{\Xn (s) }{\Hmath }{2} )\, ds \Bigr]    \\
\;  &\le \;   C  \nnorm{\Pn }{\lcal (\Vmath ,\Vmath )}{}  \, \Bigl\{ 1 +  \mathbb{E} \,\Bigl[ 
  \sup_{s \in [0,T]} \nnorm{\Xn (s)}{{\Hmath }}{2}\Bigr] \Bigr\} \theta   
\;  \le \;   {c}_{6} \cdot \theta .
\end{split} 
\end{equation*} 
Here ${c}_{6} := C \sup_{n\in \nat } \nnorm{\Pn }{\lcal (\Vmath ,\Vmath )}{}  \{1+ {C}_{1}(2) \} < \infty $.
Thus the proof of Lemma \ref{lem:X_n-tightness} is complete.
\end{proof}

\subsection{Application of the Skorokhod theorem} \label{sec:main_th-proof}

\medskip \noindent
We will use the  Jakubowski's generalization of the Skorokhod theorem, see Theorem \ref{th:2_Jakubowski}
in Appendix \ref{sec:comp-tight}. 
Let ${(\Xn )}_{n \in \nat }$ be a sequence of the solutions of the approximate equations \eqref{eq:Hall-MHD_truncated}. By Lemma  \ref{lem:X_n-tightness}, the set of  laws $\{ \Law (\Xn , n \in \nat \} $ is tight on the space $(\mathcal{Z},\sigma (\mathcal{T}))$, where $\sigma (\mathcal{T})$ denotes the topological $\sigma $-field.
By  Theorem \ref{th:2_Jakubowski} there exists a subsequence ${({n}_{k})}$, a probability space $(\tOmega , \tfcal , \tp )$ and, on this space $\mathcal{Z}$-valued random variables $\tX $, $\tXnk $, $k \in \nat $
 such that  
 \begin{equation}
 \begin{split}
&\mbox{the variables $\Xnk $ and $\tXnk $ have the same laws on the Borel $\sigma$-algebra
$\mathcal{B}\big(\zcal \big)$}
\end{split}
\label{eqn:equal_laws_truncated}
\end{equation}   
and  
\begin{equation}
\mbox{$\tXnk $ converges to $\tX $ in $\zcal, \; \; \tp $-a.s.,   }
\label{eqn:conv-as_truncated}
\end{equation}   
where $\zcal $ defined by \eqref{eq:Z_T}. 
Hence, in particular,
\begin{equation}
\tX \; \in \; {L}^{2}(0,T;\Vmath ) \cap  \ccal ([0,T];{\Hmath }_{w}) \cap  \ccal ([0,T];{\Umath }^{\prime }) .
\label{eq:tX_in_Z}
\end{equation}
We will denote the subsequence $(\tXnk )$ again by $(\tXn )$.
Define a corresponding  sequence of  filtrations by
$
{\tilde{\mathbb{F}}}_{n} =({\tfcal }_{n,t})_{t\geq 0}, \mbox{ where }
{\tfcal }_{n,t} = \sigma \{ (\tXn (s)), \, \, s \le t \},\;\;  t \in [0,T].
$
Using Lemma \ref{lem:Hall-MHD_truncated_estimates},
we infer that the processes $\tXn $, $n \in \nat $,  satisfy the following inequalities:  
for every $q\in [2,p]$
\begin{equation} 
\sup_{n \in \nat } \tilde{\mathbb{E}} \Bigl[ \sup_{s\in [0,T] } \nnorm{\tXn (s)}{\Hmath }{q} \Bigr]
\; \le \;  {C}_{1}(p,q)
\label{eq:H_estimate_truncated'_q}
\end{equation}
and
\begin{equation} 
\sup_{n \in \nat } \tilde{\mathbb{E}} \Bigl[ \int_{0}^{T} \norm{\tXn (s)}{\Vmath }{2} \, ds \Bigr] 
\; \le \;  {C}_{2}(p).
\label{eq:V_estimate_truncated'}
\end{equation}  
Let us emphasize that the constants  ${C}_{1}(p,q)$ and ${C}_{2}(p)$ are the same as in Lemma \ref{lem:Hall-MHD_truncated_estimates}.
In particular, by \eqref{eq:H_estimate_truncated'_q} with $q:=p$
\begin{equation} 
\te  \Bigl[ \sup_{s\in [0,T]} \nnorm{\tXn (s)}{\Hmath }{p} \Bigr] \; \le \;  {C}_{1}(p),
\label{eq:H_estimate_truncated'_p}
\end{equation}
where ${C}_{1}(p):=C_1(p,p)$.
Using inequality  \eqref{eq:H_estimate_truncated'_p} we choose a subsequence, still denoted by $(\tXn )$, convergent weak star in the space $ {L}^{p}(\tOmega ;{L}^{\infty }(0,T;{\Hmath }))$  and infer that
and that the limit process $\tX $ satisfies  \eqref{eq:H_estimate_truncated'_p}, as well. That is,
\begin{equation} 
\te  \Bigl[ \sup_{s\in [0,T]} \nnorm{\tX (s)}{\Hmath }{p} \Bigr] \; \le \;  {C}_{1}(p).
\label{eq:H_estimate_Hall-MHD'_p}
\end{equation}
This means that the process $\tX $ satisfies inequality \eqref{eq:H_estimate_Hall-MHD_q} for $q:=p$.  
Now, let us fix $q \in [1,p)$. 
Since $ \sup_{t\in [0,T]}  \nnorm{\tX (t)}{\Hmath }{q}
    \le {\Bigl( \sup_{t\in [0,T]} \nnorm{\tX (t)}{\Hmath }{p} \Bigr) }^{q/p}$,
by the H\"{o}lder inequality we obtain   
\begin{equation}
\begin{split}
& \tilde{\mathbb{E}} \Bigl[ \sup_{t\in [0,T]}  \nnorm{\tX (t)}{\Hmath }{q}  \Bigr]
 \; \le \;  \tilde{\mathbb{E}} \Bigl[ {\Bigl( \sup_{t\in [0,T]} \nnorm{\tX (t)}{\Hmath }{p} \Bigr) }^{q/p}  \Bigr] \\
& \qquad  \le \;  {\biggl( \tilde{\mathbb{E}} \Bigl[  \sup_{t\in [0,T]} \nnorm{\tX (t)}{\Hmath }{p}   \Bigr]  \biggr) }^{q/p}
\;  \le \;   {\bigl( {C}_{1}(p) \bigr) }^{q/p},
 \end{split}
\label{eq:H_estimate_Hall-MHD'_q} 
\end{equation}    
which means that process $\tX $ satisfies inequality \eqref{eq:H_estimate_Hall-MHD_q}
 with the constant
${C}_{1}(p,q):={\bigl( {C}_{1}(p) \bigr) }^{q/p}$.

\medskip  \noindent
By inequality \eqref{eq:V_estimate_truncated'} we infer that the sequence $(\tXn )$ contains further subsequence, denoted again by $(\tXn )$, convergent weakly in the space ${L}^{2}([0,T]\times \bar{\Omega };{\Vmath } )$ to $\tX $. Moreover, it is clear that
\begin{equation} 
 \tilde{\mathbb{E}} \Bigl[ \int_{0}^{T} \norm{\tX  (s)}{\Vmath }{2} \, ds \Bigr]
\; \le \; {C}_{2}(p),
\label{eq:V_estimate_Hall-MHD'}
\end{equation}
which means that the process $\tX $ satisfies \eqref{eq:V_estimate_Hall-MHD}.

\subsection{Continuation of the proof of Theorem \ref{th:mart-sol_existence}. Passing to the limit}

\medskip  \noindent
We use the argumentation similar to \cite{Brze+EM'13}, which is closely related to \cite{Flandoli+Gatarek'1995} and \cite[Section 8]{DaPrato+Zabczyk'2014}.
For each $n\ge 1 $, let us consider a process $\tMn $ with trajectories in $\ccal ([0,T];\Hn )$ (in particular in $\ccal ([0,T];\Hmath )$) defined by 
\begin{equation}  \label{eq:tMn}
\tMn (t) \; = \;  \tXn (t) \,  -  \tXn (0) 
+ \int_{0}^{t}\bigl[ \ARiesz( \tXn (s)) + \BRiesz  (\tXn (s)) +\RRiesz ( \tXn (s)) -\fRiesz (s)\bigr] \, ds 
 ,  \quad
t \in [0,T],
\end{equation}
where $\ARiesz, \BRiesz , \RRiesz $ and $\fRiesz $  are defined by \eqref{eq:Acal_Hn_Riesz}, \eqref{eq:Bn_Hn_Riesz}, \eqref{eq:R_Hn_Riesz} and \eqref{eq:f_Hn_Riesz}, respectively.

\medskip
\begin{lemma} \label{lem:tMn-martingale}
$\tMn $ is a square integrable martingale with respect to the filtration
${\tilde{\mathbb{F}}}_{n} =({\tilde{\fcal }}_{n,t})$, where
${\tilde{\fcal }}_{n,t} = \sigma \{ \tXn (s), \, \, s \le t \} $, with the  quadratic variation
\begin{equation} \label{eq:tMn_qvar}
{\qvar{\tMn }}_{t} \; = \; \int_{0}^{t} \norm{ \Gn (\tXn (s))}{\lhs (\Kmath ,\Hmath )}{2}   \, ds ,
  \qquad t \in [0,T].
\end{equation}
\end{lemma}

\medskip
\begin{proof}
Indeed, since the process
\[
\Mn (t) \; = \;  \Xn (t) \,  -  \Xn (0) 
+ \int_{0}^{t}\bigl[ \ARiesz( \Xn (s)) + \BRiesz  (\Xn (s)) +\RRiesz ( \Xn (s)) -\fRiesz (s)\bigr] \, ds 
 ,  \quad t \in [0,T],
\]
is a zero-mean  $\Hn $-valued square integrable martingale with the quadratic variation
\[
{\qvar{\Mn }}_{t} \; = \; \int_{0}^{t} \norm{ \Gn (\Xn (s))}{\lhs (\Kmath ,\Hmath )}{2}   \, ds ,
  \qquad t \in [0,T],
\]
and the processes $\tXn $ and $\Xn $ have the same laws, 
thus for all $s, t \in [0,T]$, $s \le t $, all functions $h$ bounded continuous on $\ccal ([0,s]; {\Hmath }_{n} )$, and all $\psi , \zeta  \in \Hn   $, we have
\begin{equation*} 
\te \bigl[ \dual{ \tMn (t)-\tMn (s)}{\psi }{\Hmath } \, h \bigl( \tXn {}_{|[0,s]} \bigr) \bigr]
\; = \; 0
\end{equation*}
and
\begin{equation*}
\begin{split}
&  \te \Bigl[ \Bigl( \dual{\tMn (t)}{\psi }{\Hmath } \dual{\tMn (t)}{\zeta }{\Hmath }
 - \dual{\tMn (s)}{\psi }{\Hmath } \dual{\tMn (s)}{\zeta }{\Hmath }  \\
&  \quad - \int_{s}^{t} \ilsk{\Gn  (\tXn (\sigma )) \psi }{\Gn  (\tXn (\sigma )) \zeta }{\Hmath }
\, d\sigma  \Bigl)  \cdot h \bigl( \tXn {}_{|[0,s]} \bigr) \Bigr] \; = \; 0 .   
\end{split}
\end{equation*}
This concludes the proof of the lemma.
\end{proof}

\medskip  \noindent
Let $\tM $ be an ${\Umath }^{\prime }$-valued process defined by
\begin{equation}  \label{eq:tM}
\tM (t) \; = \;  \tX (t) \,  -  \tX (0) 
+ \int_{0}^{t} \bigl[  \acal \tX (s)  + \MHD  (\tX (s)) +\tHall (\tX (s)) - f(s) \bigr] \, ds 
 , \quad
t \in [0,T].
 \end{equation}

\medskip
\begin{lemma} \label{lem:tM-continuity}
The process ${(\tM (t))}_{t\in [0,T]}$ has $\tp$-a.s.  trajectories in $\ccal ([0,T];\Umath ')$. 
\end{lemma}

\medskip  
\begin{proof}
Since by \eqref{eq:tX_in_Z}, $\tX \in \ccal ([0,T];\Umath ')$, it is sufficient to show that the remaining terms on the r.h.s of \eqref{eq:tM} are $\tp $-a.s. in $\ccal ([0,T];\Umath ')$. In fact, the remaining terms are more regular; they belong to $\ccal ([0,T];\Vmath ')$ or $\ccal ([0,T];\Vtest ')$, which follows from the following inequalities. 
By Remark \ref{rem:Acal-term_properties}, the H\"{o}lder (Cauchy-Schwarz) inequality  and \eqref{eq:V_estimate_Hall-MHD'} we have
\begin{equation*}
\te \biggl[ \nnorm{\acal (\tX (s))}{\Vmath '}{} \, ds  \biggr]
\; \le \; {T}^{\frac{1}{2}} \, \biggl(  \te \biggl[ \int_{0}^{T} \norm{\tX (s)}{}{2} \, ds \biggr] {\biggr) }^{\frac{1}{2}} \; < \infty . 
\end{equation*}
By Lemma \ref{lem:MHD-term_properties} (i),  \eqref{eq:Hall-MHD_V-norm}, \eqref{eq:H_estimate_Hall-MHD'_q} and \eqref{eq:V_estimate_Hall-MHD'} we have
\begin{equation*}
\te \biggl[ \int_{0}^{T} \nnorm{\MHD (\tX (s))}{\Vmath '}{} \, ds  \biggr]
\; \le \;  {c}_{\MHD } \, \te \biggl[ \int_{0}^{T} \norm{\tX (s)}{\Vmath }{2} \, ds \biggr]
\; <  \; \infty . 
\end{equation*}
By Lemma \ref{lem:tHall-term_properties}  (i),  \eqref{eq:Hall-MHD_V-norm}, \eqref{eq:H_estimate_Hall-MHD'_q} and \eqref{eq:V_estimate_Hall-MHD'} we have
\begin{equation*}
\te \biggl[ \int_{0}^{T} \nnorm{\tHall (\tX (s))}{\Vast '}{} \, ds  \biggr]
\; \le \;  {c}_{\tHall } \, \te \biggl[ \int_{0}^{T} \norm{\tX (s)}{\Vmath }{2} \, ds \biggr]
\; <  \; \infty . 
\end{equation*}
Finally, by condition \bf (H.2) \rm in Assumption \ref{assumption-data},  
$
\int_{0}^{T} \nnorm{f(s)}{\Vmath '}{} \, ds   <   \infty . 
$
The proof of the lemma is thus complete. 
\end{proof}

\medskip
\begin{lemma} \label{lem:pointwise_conv}
For all $s,t \in [0,T]$ such that $s \le t $ and all $\varphi \in \Vast $
\begin{description}
\item[(a)] $\lim_{n\to \infty }  \ilsk{\tXn (t)}{\Pn \varphi}{\Hmath } = \ilsk{\tX (t)}{\varphi}{\Hmath  }  $, $\tp$-a.s.,
\item[(b)] $\lim_{n\to \infty }  \int_{s}^{t} \dual{\acal \tXn (\sigma )}{\Pn \varphi}{} \, d \sigma
= \int_{s}^{t} \dual{\acal \tX (\sigma )}{\varphi}{} ) \, d\sigma   $,
$\tp$-a.s.,
\item[(c)] $\lim_{n\to \infty }  \int_{s}^{t} \dual{ \MHD (\tXn (\sigma ))}{\Pn \varphi}{} \, d \sigma 
= \int_{s}^{t} \dual{ \MHD (\tX (\sigma ))}{\varphi}{} \,d\sigma   $, $\tp$-a.s.,
\item[(d)] $\lim_{n\to \infty }  \int_{s}^{t} \dual{ \tHall (\tXn (\sigma ))}{\Pn \varphi}{}
= \int_{s}^{t} \dual{ \tHall (\tX (\sigma ))}{\varphi}{} ) \,d\sigma  $, $\tp$-a.s.
\item[(e)] $\lim_{n\to \infty }  \int_{s}^{t} \dual{  f(\sigma )}{\Pn \varphi}{} \, d\sigma 
= \int_{s}^{t}  \dual{  f(\sigma )}{\varphi}{} ) \,d\sigma  $, $\tp$-a.s.,
\end{description} 
where $\dual{\cdot}{\cdot }{}$ denotes appropriate duality pairing.
\end{lemma}

\medskip
\begin{proof}
Let us fix $s,t \in [0,T]$, $s \le t $ and $\varphi \in \Vast $.

\medskip  \noindent
\bf Ad. (a). \rm Since 
$$\ilsk{\tXn (t)}{\Pn \varphi}{\Hmath } = \ilsk{\Pn\tXn (t)}{ \varphi}{\Hmath }  =\ilsk{\tXn (t)}{ \varphi}{\Hmath } ,$$
and by \eqref{eqn:conv-as_truncated}
$ \tXn \to \tX $ in $\ccal ([0,T];{\Hmath }_{w})$, $\tp$-a.s., 
we infer that 
$\ilsk{\tXn (\cdot )}{\varphi}{\Hmath } \to \ilsk{\tX (\cdot )}{\varphi}{\Hmath } $
in $\ccal ([0,T];\rzecz )$, $\tp$-a.s.. Hence, in particular,
\[
\lim_{n\to \infty }  \ilsk{\tXn (t)}{ \varphi}{\Hmath } = \ilsk{\tX (t)}{\varphi}{\Hmath  }  ,
\qquad \tp-a.s.,
\] 
which completes the proof of (a).

\medskip  \noindent
\bf Ad. (b). \rm 
By \eqref{eq:A_acal_rel}, we have
\[
\int_{s}^{t} \dual{ \acal \tXn (\sigma )}{\Pn \varphi}{} \, d\sigma  
\; = \;  \int_{s}^{t} \dirilsk{\tXn (\sigma )}{\Pn \varphi}{} \, d\sigma  .
\]
By \eqref{eqn:conv-as_truncated}, $\tXn  \to \tX $ in ${L}_{w}^{2}(0,T;\Vmath )$, 
$\tp$-a.s.
Moreover, since $\varphi \in \Vast  $, we infer that $\Pn \varphi \to \varphi $ in $\Vmath $, 
(see  Corollary \ref{cor:P_n-pointwise_conv}).
Thus 
\begin{equation*} 
\begin{split}
& \lim_{n \to \infty } \int_{s}^{t} \dirilsk{\tXn (\sigma )}{\Pn \varphi}{} \, d\sigma   
\; = \; \int_{s}^{t} \dirilsk{ \tX (\sigma )}{\varphi}{} \, d\sigma  
\; = \; \int_{s}^{t} \dual{ \acal \tX (\sigma )}{\varphi}{} \, d\sigma , \qquad \mbox{$\tp$-a.s.} 
\end{split} 
\end{equation*}
This completes the proof of (b).

\medskip  \noindent
\bf Ad. (c). \rm 
Let us move to the $\MHD $-term. Let us fix $\varphi \in \Vast $.
For every $\sigma \in [0,T]$ we have
\[
\begin{split}
& \dual{ \MHD (\tXn (\sigma ))}{\Pn \varphi}{} - \dual{ \MHD (\tX (\sigma ))}{\varphi}{} \\
\; &= \; \dual{ \MHD (\tXn (\sigma ))}{\Pn \varphi - \varphi }{} + \dual{ \MHD (\tXn (\sigma )) - \MHD(\tX (\sigma ))}{ \varphi}{}.
\end{split}
\]
Since by \eqref{eqn:conv-as_truncated}, the sequence $(\tXn )$ is $\tp$-a.s., weakly convergent  to $\tX $ in ${L}^{2}(0,T;\Vmath )$,  $(\tXn )$ is bounded in ${L}^{2}(0,T;\Vmath )$,  and in particular, by \eqref{eq:Hall-MHD_V-norm},  $(\tXn )$ is bounded in ${L}^{2}(0,T;\Hmath )$, as well.
Moreover, by \eqref{eqn:conv-as_truncated}, $\tXn  \to \tX $ in 
${L}^{2}(0,T;{\Hmath }_{loc}) $, $\tp$-a.s.
By Corollary \ref{cor:MHD-map_conv-aux}, we infer that $\tp$-a.s.
\[
\lim_{n\to \infty } \int_{s}^{t} \dual{ \MHD (\tXn (\sigma )) - \MHD (\tX (\sigma ))}{\varphi}{} \, d\sigma \; = \; 0 .
\]
Using \eqref{eq:MHD-map_est-H-H}
we have
\[
\begin{split}
&\Bigl| \int_{s}^{t} \ddual{\Vastprime }{ \MHD (\tXn (\sigma ))}{\Pn \varphi - \varphi }{\Vast } \, d\sigma  \Bigr|
\; \le \;  \int_{s}^{t} |\ddual{\Vastprime }{ \MHD (\tXn (\sigma ))}{\Pn \varphi - \varphi }{\Vast } | \, d\sigma  \\
\; &\le \; \norm{\MHD }{}{} \cdot \int_{s}^{t} \nnorm{\tXn (\sigma )}{\Hmath }{2}  \, d\sigma  
\cdot \norm{\Pn \varphi - \varphi }{\Vast }{} 
\; \le \; \norm{\MHD }{}{} \cdot \norm{\tXn }{{L}^{2}(0,T;\Hmath )}{2} \cdot \norm{\Pn \varphi - \varphi }{\Vast }{}  , \quad \tp-a.s.
\end{split}
\]
Since $\varphi \in \Vast $ then, by Corollary \ref{cor:P_n-pointwise_conv},   $\Pn \varphi \to \varphi $ in $\Vast$, and by the boundedness of the sequence $(\tXn )$ in ${L}^{2}(0,T;\Hmath )$,
we infer that $\tp$-a.s.
\[
\lim_{n\to \infty } \int_{s}^{t} \ddual{\Vastprime }{ \MHD (\tXn (\sigma ))}{\Pn \varphi - \varphi }{\Vast } \, d\sigma  
\; = \; 0.
\]
Hence 
\begin{equation*} 
\begin{split}
& \lim_{n\to \infty } \int_{s}^{t} (\dual{ \MHD( \tXn (\sigma ))}{\Pn \varphi}{} - \dual{\MHD (\tX (\sigma ))}{\varphi }{} ) \, d\sigma   \; = \; 0 ,
\qquad \mbox{$\tp$-a.s.} 
\end{split}
\end{equation*}
The proof of (c) is thus complete.

\medskip  
\noindent
\bf Ad. (d). \rm 
Let us move to the $\tHall $-term. Let us fix $\varphi \in \Vast $.
For every $\sigma \in [0,T]$ we have
\[
\begin{split}
& \dual{ \tHall (\tXn (\sigma ))}{\Pn \varphi}{} - \dual{ \tHall (\tX (\sigma ))}{\varphi}{} \\
\; &= \; \dual{ \tHall (\tXn (\sigma ))}{\Pn \varphi - \varphi }{} + \dual{ \tHall (\tXn (\sigma )) - \tHall (\tX (\sigma ))}{ \varphi}{}.
\end{split}
\]
By \eqref{eqn:conv-as_truncated}, the sequence $(\tXn )$ is $\tp$-a.s. convergent to $\tX $ in ${L}_{w}^{2}(0,T;\Vmath ) \cap {L}^{2}(0,T;{\Hmath }_{loc}) $.
In particular,  $(\tXn )$ is bounded in ${L}^{2}(0,T;\Hmath )$.
By Corollary \ref{cor:tHall-term_conv_general}  we infer that $\tp$-a.s.
\[
\lim_{n\to \infty } \int_{s}^{t} \dual{\tHall (\tXn (\sigma )) - \tHall (\tX (\sigma ))}{\varphi}{}\, d\sigma  \; = \; 0 .
\]
Using \eqref{eq:tHall-map_est-H-V},
we have
\[
\begin{split}
&\Bigl| \int_{s}^{t} \ddual{\Vastprime }{ \tHall (\tXn (\sigma ))}{\Pn \varphi - \varphi }{\Vast } \, d\sigma  \Bigr|
\; \le \;  \int_{s}^{t} |\ddual{\Vastprime }{ \tHall (\tXn (\sigma ))}{\Pn \varphi - \varphi }{\Vast } | \, d\sigma  \\
\; &\le \; \norm{\tHall }{}{} \cdot \int_{s}^{t} \nnorm{\tXn (\sigma )}{\Hmath }{} \, \nnorm{\tXn (\sigma )}{\Vmath }{}  \, d\sigma  
\cdot \norm{\Pn \varphi - \varphi }{\Vast }{} \\
\; &\le \; \norm{\MHD }{}{} \cdot \norm{\tXn }{{L}^{2}(0,T;\Hmath )}{} \, \norm{\tXn }{{L}^{2}(0,T;\Vmath )}{}
\cdot \norm{\Pn \varphi - \varphi }{\Vast }{}  , \quad \tp-a.s.
\end{split}
\]
Since $\varphi \in \Vast $ then, by Corollary \ref{cor:P_n-pointwise_conv},  $\Pn \varphi \to \varphi $ in $\Vast$. By the boundedness of the sequence $(\tXn )$ in ${L}^{2}(0,T;\Vmath )$ (and hence in ${L}^{2}(0,T;\Hmath )$, as well),
we infer that $\tp$-a.s.
\[
\lim_{n\to \infty } \int_{s}^{t} \ddual{\Vastprime }{ \tHall (\tXn (\sigma ))}{\Pn \varphi - \varphi }{\Vast } \, d\sigma   \; = \; 0.
\]
Hence 
\begin{equation*} 
\begin{split}
& \lim_{n\to \infty } \int_{s}^{t} (\dual{ \tHall (\tXn (\sigma ))}{\Pn \varphi}{} - \dual{\tHall (\tX (\sigma ))}{\varphi }{} ) \, d\sigma   \; = \; 0 ,
\qquad \mbox{$\tp$-a.s.} 
\end{split}
\end{equation*}
The proof of (d) is thus complete.

\medskip  \noindent
\bf Ad. (e). \rm This assertion follows from the facts that $f \in {L}^{p}(0,T;\Vmath ' )$ and for $\varphi \in \Vast  $,  $\Pn \varphi \to \varphi $ in $\Vmath $.

\medskip  \noindent
The proof of the lemma is thus complete.
\end{proof}

\begin{lemma} \label{lem:conv_martingale}
For all $ s,t \in [0,T] $ such that $s \le t$, all $h \in {\ccal }_{b}(\ccal ([0,s];{\Umath }');\rzecz ) $ and all $ \psi \in \Vast $:
\[
\lim_{n\to \infty } \te \bigl[ \dual{\tMn (t)-\tMn (s) }{\psi }{}\, h \bigl( \tXn {}_{|[0,s]} \bigr) \bigr]
\; = \;  \te \bigl[ \dual{ \tM (t)-\tM (s) }{\psi }{} \, h \bigl( \tX {}_{|[0,s]} \bigr) \bigr] ,
\]
where $\dual{\cdot}{\cdot }{}$ denotes appropriate duality pairing.
\end{lemma}
\noindent
Here, ${\ccal }_{b}(\ccal ([0,s];{\Umath }');\rzecz ) $ is the space of  $\rzecz $-valued bounded and  continuous functions defined on $\ccal ([0,s];{\Umath }')$. 

\medskip
\begin{proof}
Let us fix $s,t \in [0,T]$, $s \le t$ and  $\psi \in \Vast $.
By Remark \ref{rem:truncated_funct_Riesz} with $\varphi := \Pn \psi $,  we have
\begin{equation*}
\begin{split}
& \dual{\tMn (t)-\tMn (s)}{\psi }{}
\; = \;  \ilsk{\tXn (t)}{\Pn \psi }{\Hmath } - \ilsk{\tXn (s)}{\Pn \psi }{\Hmath }
+ \int_{s}^{t} \dual{ \acal \tXn (\sigma) }{\Pn \psi }{} \, d\sigma
\\
& +\int_{s}^{t} \dual{ \MHD (\tXn (\sigma  )  }{\Pn \psi }{}\, d\sigma
 + \int_{s}^{t} \ddual{\Vastprime }{ \tHall (\tXn (\sigma ))}{\Pn \psi  }{\Vast } \, d\sigma  
      +\int_{s}^{t} \dual{ f(\sigma) }{\Pn \psi }{} \, d\sigma .
\end{split}      
\end{equation*}
By Lemma \ref{lem:pointwise_conv}, we infer that
\begin{equation} \label{eq:mart_pointwise_conv}
\lim_{n \to \infty }  \dual{\tMn (t)-\tMn (s)}{\psi }{} \; = \; \dual{\tM (t)-\tM (s)}{\psi }{},
   \quad  {\tp } \mbox{ - a.s.}
\end{equation}
Let us notice that    $\sup_{n \in \nat } \norm{h( \tXn {}_{|[0,s]})}{{L}^{\infty }}{} < \infty $. Moreover, since by \eqref{eqn:conv-as_truncated} $\tXn \to \tX $ in $\ccal ([0,T];\Umath ')$, we infer that 
$\lim_{n \to \infty }h(\tXn {}_{|[0,s]} ) =h( \tX {}_{|[0,s]})$,  ${\tp } $-a.s.

\medskip  \noindent
Let us denote
\begin{equation*}
{\xi }_{n}(\omega ) \; := \;  \bigl( \dual{\tMn (t, \omega )}{\psi }{} - \dual{\tMn (s, \omega )}{\psi }{} \bigr)
\, h \bigl( \tXn {}_{|[0,s]} \bigr) , \qquad \omega \in {\tOmega } .
\end{equation*}
We will prove that the functions $\{ {\xi }_{n} {\} }_{n \in \nat }$ are uniformly integrable.
We claim  that
\begin{equation} \label{eq:mart_uniform_int}
\sup_{n \ge 1}  \te \bigl[ {| {\xi }_{n} |}^{2} \bigr] \; < \; \infty.
\end{equation}
Indeed, by the continuity of the embedding $\Vast \hookrightarrow \Hmath $ and the Schwarz inequality, for each $n \in \nat $ we have
\begin{equation}  \label{eq:mart_uniform_int_1}
\te \bigl[ {| {\xi }_{n}|}^{2} \bigl] \; \le \;  2c\norm{h}{{L}^{\infty }}{2} \norm{\psi }{\Vast }{2}
\te \bigl[ \nnorm{\tMn (t)}{\Hmath }{2}) + \nnorm{\tMn (s)}{\Hmath }{2} \bigr] .
\end{equation}
Since $\tMn  $  is a continuous martingale with quadratic variation given  in \eqref{eq:tMn_qvar}, by the Burkhol\-der-Davis-Gundy inequality we obtain
\begin{equation} \label{eq:mart_BDG_est}
\te \Bigl[ \sup_{t \in [0,T]} \nnorm{\tMn (t)}{\Hmath }{2}\Bigr]
\; \le \;  c \te  \Bigl[ \Bigl( \int_{0}^{T}\norm{\Gn (\tXn (\sigma ))}{\lhs (\Kmath ,\Hmath )}{2}  \, d\sigma {\Bigr) }^{\frac{1}{2}} \Bigr] .
\end{equation}
Since  $\Pn :\Hmath \to \Hn  $ is an orthogonal projection,
by \eqref{eq:G}, we have
\begin{equation} \label{eq:mart_BDG_est_1}
\norm{ \Gn (\tXn (\sigma ))}{\lhs (\Kmath ,\Hmath )}{2} \le (2-\eta ) \norm{\tXn (\sigma )}{}{2}
 + {\lambda }_{} \nnorm{\tXn (\sigma )}{\Hmath }{2} + \varrho , \qquad \sigma \in [0,T].
\end{equation}
By \eqref{eq:mart_BDG_est}, \eqref{eq:mart_BDG_est_1}, \eqref{eq:V_estimate_Hall-MHD'} and
\eqref{eq:H_estimate_Hall-MHD'_q} (with $q:=2$), we infer that
\begin{equation} \label{eq:mart_uniform_int_2}
\sup_{n \in \nat } \te  \Bigl[ \sup_{t \in [0,T]} \nnorm{\tMn (t)}{\Hmath }{2}\Bigr] \; < \; \infty .
\end{equation}
Then by \eqref{eq:mart_uniform_int_1} and \eqref{eq:mart_uniform_int_2} we see that \eqref{eq:mart_uniform_int} holds.
Since the sequence $( {\xi }_{n} {)}_{n \in \nat }$ is uniformly integrable and by
\eqref{eq:mart_pointwise_conv} it is $\tp $-a.s. pointwise convergent, application of the Vitali theorem
completes the proof of the lemma.
\end{proof}

\medskip
\begin{lemma} \label{lem:qvar_conv_left}
For all $s,t \in [0,T]$ such that $s \le t$, all $h \in {\ccal }_{b}(\ccal ([0,s];{\Umath }');\rzecz ) $ and all $\psi ,\zeta \in \Vast $:
\begin{equation*}
\begin{split}
\lim_{n\to\infty }
\te \Bigl[ \bigl\{  \dual{\tMn (t)}{\psi }{} \dual{\tMn (t)}{\zeta }{}
 - \dual{\tMn (s)}{\psi }{} \dual{\tMn (s)}{\zeta }{} \bigr\} \, h \bigl( \tXn {}_{|[0,s]} \bigr)\Bigr] &  
\\
\; = \;  \te \Bigl[ \bigl\{  \dual{\tM (t)}{\psi }{} \dual{\tM (t)}{\zeta }{}
 - \dual{\tM (s)}{\psi }{} \dual{\tM (s)}{\zeta }{} \bigr\} \, h \bigl( \tX {}_{|[0,s]}\bigr) \Bigr] , & 
\end{split} 
\end{equation*}
where $\dual{\cdot}{\cdot }{}$ denotes appropriate duality pairing.
\end{lemma}

\medskip
\begin{proof}
Let us fix $s,t \in [0,T]$ such that $s \le t$ and  $\psi ,\zeta \in {\Vmath }_{\ast }$ and let us denote
\begin{equation*}
\begin{split}
{\xi }_{n}(\omega ) \; &:= \; \bigl\{  \dual{\tMn (t,\omega )}{\psi }{} \dual{\tMn (t, \omega )}{\zeta }{}
 - \dual{\tMn (s, \omega )}{\psi }{} \dual{\tMn (s, \omega )}{\zeta }{} \bigr\} \, h \bigl( \tXn {}_{|[0,s]}(\omega )\bigr) ,\\
 \xi (\omega ) \; &:= \;  \bigl\{  \dual{\tM (t,\omega )}{\psi }{} \dual{\tM (t, \omega )}{\zeta }{}
 - \dual{\tM (s, \omega )}{\psi }{} \dual{\tM (s, \omega )}{\zeta }{} \bigr\} \, h \bigl( \tX {}_{|[0,s]}(\omega )\bigr),
\qquad \omega \in {\tOmega }.
\end{split}
\end{equation*}
Let us notice that    $\sup_{n \in \nat } \norm{h( \tXn {}_{|[0,s]})}{{L}^{\infty }}{} < \infty $. Moreover, since $\tXn \to \tX $ in $\ccal ([0,T];\Umath ')$, we infer that 
$\lim_{n \to \infty }h(\tXn {}_{|[0,s]} ) =h( \tX {}_{|[0,s]})$,  ${\tp } $-a.s.
From this and Lemma \ref{lem:pointwise_conv}, we infer that
$\lim_{n \to \infty } {\xi }_{n} (\omega ) = \xi (\omega )$ for  $\tp$-almost all $\omega \in \tOmega $.

\medskip  \noindent
\bf Uniform integrability. \rm 
We will prove that the functions $\{ {\xi }_{n} {\} }_{n \in \nat }$ are uniformly integrable.
To this end, it is sufficient to show that for some $r>1$,
\begin{equation} \label{eq:uniform_int}
\sup_{n \ge 1}  \te \bigl[ {| {\xi }_{n} |}^{r} \bigr] \; < \; \infty.
\end{equation}
In fact, we will show that condition \eqref{eq:uniform_int} holds for any $1<r\le\frac{p}{2}$.

\medskip  \noindent
Indeed, using the H\"{o}lder inequality we have for each $n\ge 1 $
\begin{equation}
\begin{split}
&\te [{|{\xi }_{n}|}^{r}] \\
\; & \le
 \frac{1}{2} \, \norm{h}{{L}^{\infty }}{r} \Bigl\{ \te [{| \dual{\tMn (t)}{\psi }{} |}^{2r} ] 
+ \te [{| \dual{\tMn (t)}{\zeta }{} |}^{2r} ]
+ \te [{| \dual{\tMn (s)}{\psi }{} |}^{2r} ] 
+ \te [{| \dual{\tMn (s)}{\zeta }{} |}^{2r} ] \Bigr\} 
\end{split}
\label{eq:uniform_int_1}
\end{equation}
Note that
\[
\dual{\tMn (t)}{\psi }{} \; = \; \ilsk{\tMn (t)}{\psi }{\Hmath }.
\]
By Lemma \ref{lem:tMn-martingale} ${(\tMn )}_{t\in [0,T]}$ is an $\Hmath $-valued continuous square integrable martingale with the quadratic variation given by \eqref{eq:tMn_qvar}. Thus ${(\ilsk{\tMn (t)}{\psi }{\Hmath })}_{t\in [0,T]}$ is a $\rzecz $-valued square integrable martingale with the quadratic variation
\begin{equation*} 
\begin{split}
{\qvar{\ilsk{\tMn }{\psi }{\Hmath }}}_{t} \; &= \; \int_{0}^{t} \norm{ \ilsk{\Gn (s,\tXn (s)) }{\psi }{\Hmath }}{\lhs (\Kmath ,\rzecz  )}{2}   \, ds  \\
\; &= \; \int_{0}^{t} \norm{\ilsk{ G(s,\tXn (s)) }{\Pn \psi }{\Hmath }}{\lcal (\Kmath ,\rzecz  )}{2}   \, ds ,
  \qquad t \in [0,T].
\end{split}  
\end{equation*}
By Remark \ref{rem:G_properties}(iii) we have for every $y \in \Kmath $
\[
\ilsk{ G(s,\tXn (s)) (y) }{\Pn \psi }{\Hmath } \; = \; \ddual{\Vmath '}{ g(s,\tXn (s)) (y) }{\Pn \psi }{\Vmath } ,
\]
and hence
\[
\begin{split}
&\norm{\ilsk{ G(s,\tXn (s)) }{\Pn \psi }{\Hmath }}{\lcal (\Kmath ,\rzecz  )}{2} 
\; = \; \sup_{y\in \Kmath , \norm{y}{\Kmath }{} \le 1 } 
{|\ilsk{ G(s,\tXn (s)) (y) }{\Pn \psi }{\Hmath } |}^{2}  \\
\; &= \; \sup_{y\in \Kmath , \norm{y}{\Kmath }{} \le 1 }  {|\ddual{\Vmath '}{ g(s,\tXn (s)) (y) }{\Pn \psi }{\Vmath }|}^{2} 
\;  \le \; C (1+ \nnorm{\tXn (s)}{\Hmath }{2}) \cdot \norm{\Pn \psi }{\Vmath }{2} \\
\;  &\le \; C (1+ \nnorm{\tXn (s)}{\Hmath }{2}) \cdot \bigl( \sup_{n\in \nat }\nnorm{\Pn }{\lcal (\Vmath ,\Vmath )}{2} \bigr) \norm{ \psi }{\Vmath }{2}
\; \le \; \tilde{C} (\psi ) (1+ \nnorm{\tXn (s)}{\Hmath }{2}) 
\end{split}
\]
for some constant positive $\tilde{C} (\psi )$.

\medskip  \noindent
By the Burkholder-Davies-Gungy inequality and the H\"{o}lder inequality
\begin{equation}
\begin{split}
& \te \Bigl[ \sup_{t\in [0,T]}{|\dual{\tMn (t)}{\psi }{} |}^{2r} \Bigr] 
\; \le \; \biggl( \te  \biggl[ \int_{0}^{T} \norm{ \ilsk{\Gn (s,\tXn (s)) }{\psi }{\Hmath }}{\lhs (\Kmath ,\rzecz  )}{2} \, ds \biggr] {\biggr) }^{r} \\
\; &\le \; {\tilde{C}}^{r} (\psi ) \cdot \biggl( \te \biggl[ \int_{0}^{T}(1+ \nnorm{\tXn (s)}{\Hmath }{2}) \, ds \biggr] {\biggr) }^{r}  
\; \le  \; {2}^{r-1}  {\tilde{C}}^{r} (\psi ) \cdot {T}^{r} \,
 \Bigl( 1+ \te \Bigl[ \sup_{s\in [0,T]} \nnorm{\tXn (s)}{\Hmath }{2r} \Bigr] \Bigr) .
\end{split}
\label{eq:BDG_est}
\end{equation}
Using   \eqref{eq:BDG_est} and \eqref{eq:H_estimate_Hall-MHD'_q} in \eqref{eq:uniform_int_1}   we infer that for any $r\in (1,\frac{p}{2}]$
\begin{equation*}
\begin{split}
\te [{|{\xi }_{n}|}^{r}] 
\; &\le  \; {2}^{r}  {\tilde{C}}^{r} (\psi ) \cdot {T}^{r} \,  \norm{h}{{L}^{\infty }}{r}  \Bigl( 1+ {C}_{2}(2r,p) \Bigr) \; < \; \infty .
\end{split}
\end{equation*}
Thus condition \eqref{eq:uniform_int} holds.

\medskip  \noindent
By the Vitali theorem 
\[
\lim_{n\to \infty } \te \bigl[ {\xi }_{n} \bigr] \; = \;  \te [\xi ].
\]
The proof of the lemma is thus complete.
\end{proof}

\medskip
\noindent
For a fixed $\psi \in \Hmath $ let us define the following map
\begin{equation*}
{\psi }^{\ast \ast } : \lhs (\Kmath ,\Hmath )  \; \ni \, f \; \mapsto \;
[\Kmath \ni y \; \mapsto \; \ilsk{f(y)}{\psi }{\Hmath }] \; \in \; \lhs (\Kmath , \rzecz ).
\end{equation*}

\medskip
\begin{lemma}  \label{lem:conv_quadr_var}
\bf (Convergence of quadratic variations). \rm
For all $s,t \in [0,T]$ such that $s\le t$, all $h \in {\ccal }_{b}(\ccal ([0,s];{\Umath }');\rzecz )$ and $\psi ,\zeta \in \Vmath $, we have
\[
\begin{split}
& \lim_{n\to \infty } \te \biggl[ \biggl( \int_{s}^{t} 
\Dual{{\psi }^{\ast \ast }[G_n(\sigma , \tXn (\sigma ))]}{{\zeta  }^{\ast \ast }[G_n(\sigma , \tXn (\sigma ))]}{\lhs (\Kmath ,\rzecz )} \, d \sigma  \biggr) \cdot h \bigl( \tXn {}_{|[0,s]} \bigr) \biggr] \\
\; & \qquad = \; \te \biggl[ \biggl( \int_{s}^{t} 
\Dual{{\psi }^{\ast \ast }[G(\sigma , \tX (\sigma ))]}{{\zeta  }^{\ast \ast }[G(\sigma , \tX (\sigma ))]}{\lhs (\Kmath ,\rzecz )} \, d \sigma  \biggr) \cdot h \bigl( \tX {}_{|[0,s]} \bigr) \biggr] .
\end{split}
\]
\end{lemma}

\medskip
\begin{proof}
Let us fix $s,t \in [0,T] $, $s \le t $ and $\psi , \zeta \in \Vmath  $.
Note that  $\sup_{n \in \nat } \norm{h( \tXn {}_{|[0,s]})}{{L}^{\infty }}{} < \infty $. Moreover, since $\tXn \to \tX $ in $\ccal ([0,T];\Umath ')$, we infer that 
$\lim_{n \to \infty }h(\tXn {}_{|[0,s]} ) =h( \tX {}_{|[0,s]})$,  ${\tp } $-a.s.

\medskip  \noindent
Let us denote
\[
\begin{split}
{\xi }_{n} (\omega ) \; &:= \; 
\int_{s}^{t} 
\Dual{{\psi }^{\ast \ast }[\Gn (\sigma , \tXn (\sigma ))]}{{\zeta  }^{\ast \ast }[\Gn (\sigma , \tXn (\sigma ))]}{\lhs (\Kmath ,\rzecz )} \, d \sigma , 
\\
\xi  (\omega ) \; &:= \; 
\int_{s}^{t} 
\Dual{{\psi }^{\ast \ast }[G (\sigma , \tX (\sigma ))]}{{\zeta  }^{\ast \ast }[G (\sigma , \tX (\sigma ))]}{\lhs (\Kmath ,\rzecz )} \, d \sigma ,
\qquad \omega \in \tOmega  .
\end{split}
\]
It is sufficient to prove that
\[
\lim_{n\to \infty } {\xi }_{n} (\omega ) \; = \; \xi (\omega ) \quad 
\]
 for $\tp $-almost all $ \omega \in \tOmega $, 
and that the sequence $({\xi }_{n})$ is uniformly integrable.

\medskip  \noindent
By Remark \ref{rem:G_properties}(iii), for every $y \in \Kmath $ we have
\begin{equation*}
\begin{split}
({\psi }^{\ast \ast  } [G (\sigma , \tX (\sigma ))]) (y)
\; &= \; \ilsk{G  (\sigma , \tX (\sigma ))}{\psi }{\Hmath }
\; = \; \ddual{\Vmath '}{g (\sigma , \tX (\sigma ))}{ \psi }{\Vmath }
\; = \;  [( {\tilde{g}}_{\psi } (\tX ))(\sigma )] (y) ,
\end{split}
\end{equation*}
where ${\tilde{g}}_{\psi }$ is a map defined by \eqref{eq:G**},
and
\begin{equation*}
\begin{split}
({\psi }^{\ast \ast  } [\Gn (\sigma , \tXn (\sigma ))]) (y)
\; &= \; \ilsk{\Gn  (\sigma , \tXn (\sigma ))}{\psi }{\Hmath }
\; = \; \ilsk{\Pn G  (\sigma , \tXn (\sigma ))}{\psi }{\Hmath }  \\
\; &= \; \ilsk{G  (\sigma , \tXn (\sigma ))}{\Pn \psi }{\Hmath } 
\; = \; \ddual{\Vmath '}{g (\sigma , \tXn (\sigma ))}{\Pn \psi }{\Vmath } . 
\end{split}
\end{equation*}

\medskip
\noindent
\bf Pointwise convergence (for $\tp$-almost all $\omega \in \tOmega $). \rm 
Let us consider
\[
\begin{split}
& {\xi }_{n} - \xi  
\\
\; &= \; 
\int_{s}^{t} \Dual{{\psi }^{\ast \ast }[\Gn (\sigma , \tXn (\sigma )) -G(\sigma , \tX (\sigma ))]}{
{\zeta  }^{\ast \ast }[\Gn (\sigma , \tXn (\sigma )) - G(\sigma , \tX (\sigma ))]}{\lhs (\Kmath ,\rzecz )} \, d\sigma \\
& \qquad + \int_{s}^{t} \Dual{{\psi }^{\ast \ast }[\Gn (\sigma , \tXn (\sigma )) -G(\sigma , \tX (\sigma ))]}{
{\zeta  }^{\ast \ast }[ G(\sigma , \tX (\sigma ))]}{\lhs (\Kmath ,\rzecz )}  \; d\sigma \\
& \qquad +\int_{s}^{t} \Dual{{\psi }^{\ast \ast }[G(\sigma , \tX (\sigma ))]}{
{\zeta  }^{\ast \ast }[\Gn (\sigma , \tXn (\sigma )) - G(\sigma , \tX (\sigma ))]}{\lhs (\Kmath ,\rzecz )} \, d\sigma 
\\
\; &=: \; {I}^{n}_{1} + {I}^{n}_{2} + {I}^{n}_{3} .
\end{split}
\]
We will analyze separately the terms ${I}^{n}_{1},{I}^{n}_{2},{I}^{n}_{3}$.

\medskip  \noindent
Let us begin with  ${I}^{n}_{1}$. Note that by the Schwarz inequality
\[
\begin{split}
&|{I}^{n}_{1}(\sigma )| \; = \; \Bigl| \Dual{{\psi }^{\ast \ast }[\Gn (\sigma , \tXn (\sigma )) -G(\sigma , \tX (\sigma ))]}{
{\zeta  }^{\ast \ast }[\Gn (\sigma , \tXn (\sigma )) - G(\sigma , \tX (\sigma ))]}{\lhs (\Kmath ,\rzecz )} \Bigr|
\\
\; &\le \; \norm{{\psi }^{\ast \ast }[\Gn (\sigma , \tXn (\sigma )) -G(\sigma , \tX (\sigma ))]}{\lhs (\Kmath ,\rzecz )}{}
\cdot \norm{{\zeta }^{\ast \ast }[\Gn (\sigma , \tXn (\sigma )) -G(\sigma , \tX (\sigma ))]}{\lhs (\Kmath ,\rzecz )}{} \\
\; &\le \; \frac{1}{2} \bigl\{  
\norm{{\psi }^{\ast \ast }[\Gn (\sigma , \tXn (\sigma )) -G(\sigma , \tX (\sigma ))]}{\lhs (\Kmath ,\rzecz )}{2} \\
& \qquad  \quad + \norm{{\zeta }^{\ast \ast }[\Gn (\sigma , \tXn (\sigma )) -G(\sigma , \tX (\sigma ))]}{\lhs (\Kmath ,\rzecz )}{2}
\bigr\} 
\; =: \; \frac{1}{2} \{ {|{I}^{n}_{11}|}^{2}(\sigma )+{|{I}^{n}_{12}|}^{2}(\sigma ) \} .
\end{split} 
\]
By Remark \ref{rem:G_properties} (iii), we have
\[
\begin{split}
|{I}^{n}_{11}(\sigma )| \; &= \; \norm{{\psi }^{\ast \ast }[\Gn (\sigma , \tXn (\sigma )) -G(\sigma , \tX (\sigma ))]}{\lhs (\Kmath ,\rzecz )}{} \\
\; &= \; \sup_{y \in \Kmath , \norm{y}{\Kmath }{} \le 1 } \bigl|
\ddual{{\Vmath }^{\ast }}{g(\sigma , \tXn (\sigma ))(y)}{\Pn \psi }{\Vmath }
- \ddual{{\Vmath }^{\ast }}{g(\sigma , \tX (\sigma ))(y)}{ \psi }{\Vmath }
\bigr|
\\
\; &\le  \; \sup_{y \in \Kmath , \norm{y}{\Kmath }{} \le 1 } \bigl|
\ddual{{\Vmath }^{\ast }}{g(\sigma , \tXn (\sigma ))(y)}{\Pn \psi  -\psi }{\Vmath } \bigr| \\
& \qquad + \sup_{y \in \Kmath , \norm{y}{\Kmath }{} \le 1 } \bigl|
\ddual{{\Vmath }^{\ast }}{g(\sigma , \tXn (\sigma ))(y)- g(\sigma , \tX (\sigma ))(y)}{ \psi }{\Vmath }
\bigr|
\\
\; &\le  \; {\bigl[ C (1+ \nnorm{\tXn (\sigma )}{\Hmath }{2}) \bigr] }^{\frac{1}{2}} \cdot \norm{\Pn \psi - \psi }{\Vmath }{}
+ \norm{  ({\tilde{g}}_{\psi } (\tXn )) (\sigma )  - ({\tilde{g}}_{\psi } (\tX )) (\sigma )  }{\lhs (\Kmath, \rzecz )}{}  .
\end{split}
\]
Thus
\begin{equation}
\begin{split}
{|{I}^{n}_{11}(\sigma )|}^{2} \; &\le  \; \frac{1}{2} \bigl\{ 
 C (1+ \nnorm{\tXn (\sigma )}{\Hmath }{2}) \cdot \norm{\Pn \psi - \psi }{\Vmath }{2}
+ \norm{  ({\tilde{g}}_{\psi } (\tXn )) (\sigma )  - ({\tilde{g}}_{\psi } (\tX )) (\sigma )  }{\lhs (\Kmath, \rzecz )}{2}  
\bigr\} .
\end{split}
\label{eq:QV_I_11_est}
\end{equation}
We have
\[
|{I}^{n}_{1}| \; \le \; \int_{s}^{t} |{I}^{n}_{1}(\sigma )| \; d\sigma 
\; \le \; \frac{1}{2} \int_{s}^{t} (|{I}^{n}_{11}(\sigma )|^{2}+|{I}{n}_{12}(\sigma )|^{2}) \, d\sigma  .
\]
By \eqref{eq:QV_I_11_est},  
\[
\begin{split}
&\int_{s}^{t} |{I}^{n}_{11}(\sigma )|^{2} \, d \sigma 
\; \le \; \frac{1}{2} \biggl\{ 
 C \Bigl( \int_{0}^{T}(1+ \nnorm{\tXn (\sigma )}{\Hmath }{2}) \, d\sigma \Bigr) \cdot \norm{\Pn \psi - \psi }{\Vmath }{2} \\
& \qquad \qquad \qquad \qquad \qquad + \int_{0}^{T} \Norm{  ({\tilde{g}}_{\psi } (\tXn )) (\sigma )  - ({\tilde{g}}_{\psi } (\tX )) (\sigma )  }{\lhs (\Kmath, \rzecz )}{2}  \, d \sigma 
\biggr\} \\
\; &= \; \frac{1}{2} \Bigl\{ 
 C T \Bigl( 1+ \sup_{\sigma \in [0,T]} \nnorm{\tXn (\sigma )}{\Hmath }{2}  \Bigr) \cdot 
\norm{\Pn \psi - \psi }{\Vmath }{2} 
+ \norm{{\tilde{g}}_{\psi } (\tXn ) -{\tilde{g}}_{\psi } (\tX )}{{L}^{2}(0,T;\lhs (\Kmath, \rzecz ))}{2} 
\Bigr\}  .
\end{split}
\]
Since $(\tXn ) \subset {L}^{2}(0,T;\Hmath )$ and $\tXn \to \tX $ in ${L}^{2}(0,T;{\Hmath }_{loc })$, by Remark \ref{rem:G_properties} (iii) we infer that for every $\psi \in \Vmath $
\[
\lim_{n\to \infty }  {\tilde{g}}_{\psi } (\tXn ) \; = \; {\tilde{g}}_{\psi } (\tX )
\quad \mbox{ in } \quad {L}^{2}(0,T;\lhs (\Kmath ,\rzecz )) ,
\]
Using additionally Corollary \ref{cor:P_n-pointwise_conv}, we obtain
\[
 \lim_{n\to \infty } \int_{s}^{t} |{I}^{n}_{11}(\sigma )|^{2} \, d \sigma \; = \; 0.
\]
Analogously,
\[
\lim_{n\to \infty } \int_{s}^{t} |{I}^{n}_{12}(\sigma )|^{2} \, d \sigma \; = \; 0.
\]
Thus
\begin{equation*}
\lim_{n\to \infty }{I}^{n}_{1}  \;= \;  \lim_{n\to \infty } \int_{s}^{t} {I}^{n}_{1}(\sigma ) \, d \sigma \; = \; 0.
\end{equation*}

\medskip  \noindent
Let us move to ${I}^{n}_{2}$.
Proceeding analogously as in the analysis of the term ${I}^{n}_{1}$, we obtain the following inequalities
\[
\begin{split}
|{I}^{n}_{2}|
\; &\le \;
\int_{s}^{t}  \Bigl\{ 
{\bigl[ C (1+ \nnorm{\tXn (\sigma )}{\Hmath }{2}) \bigr] }^{\frac{1}{2}}  
 \cdot \norm{\Pn \psi - \psi }{\Vmath }{} 
+   \norm{ ({\tilde{g}}_{\psi } (\tXn )) (\sigma )  - ({\tilde{g}}_{\psi } (\tX )) (\sigma )  }{\lhs (\Kmath, \rzecz )}{}  \Bigr\} 
\\
& \qquad 
\cdot {[C(1+ \nnorm{\tX (\sigma )}{\Hmath }{2})]}^{\frac{1}{2}} \cdot \norm{\zeta }{\Vmath }{} \,d\sigma   
\\
\; &\le \; C T \,  \norm{\zeta }{\Vmath }{} \, \Bigl( 1+ \sup_{\sigma \in [0,T ]}\nnorm{\tXn (\sigma )}{\Hmath }{2} {\Bigr) }^{\frac{1}{2}} 
\cdot \Bigl( 1+ \sup_{\sigma \in [0,T ]}\nnorm{\tX (\sigma )}{\Hmath }{2} {\Bigr) }^{\frac{1}{2}} 
 \cdot  \norm{\Pn \psi - \psi }{\Vmath }{} 
 \\
& \qquad +  \norm{\zeta }{\Vmath }{} \, \Bigl( 1+ \sup_{\sigma \in [0,T ]}\nnorm{\tX (\sigma )}{\Hmath }{2} {\Bigr) }^{\frac{1}{2}} \cdot 
 \int_{0}^{T} \norm{ ({\tilde{g}}_{\psi } (\tXn )) (\sigma )  - ({\tilde{g}}_{\psi } (\tX )) (\sigma )  }{\lhs (\Kmath, \rzecz )}{} \, d\sigma  .
\end{split}
\]
Using  Corollary \ref{cor:P_n-pointwise_conv} and Remark \ref{rem:G_properties} (iii), we infer that
\[
\lim_{n\to \infty } {I}^{n}_{2}  \;= \; \lim_{n\to \infty } \int_{s}^{t} {I}^{n}_{2}(\sigma ) \, d\sigma \; = \; 0.
\]
Analogously,
\[
\lim_{n\to \infty }{I}^{n}_{3}  \;= \; 
\lim_{n\to \infty } \int_{s}^{t} {I}^{n}_{3}(\sigma ) \, d\sigma \; = \; 0.
\]
This concludes  the proof of the pointwise convergence.

\medskip
\noindent
\bf Uniform integrability. \rm 
We will prove that the sequence $({\xi }_{n})$ is uniformly integrable. To this end it is sufficient to show that for some $r>0$
\begin{equation*}
\sup_{n\in \nat } \te [{|{\xi }_{n}|}^{r}] \; < \; \infty .
\end{equation*}
In fact we will show that the above condition holds for every $r\in (1,\frac{p}{2}]$.

\medskip  \noindent
We have $\tp $-a.s
\[
\begin{split}
|{\xi }_{n}| 
\; &\le \; \int_{0}^{T} |
\dual{{\psi }^{\ast \ast }[\Gn (\sigma , \tXn (\sigma ,\omega ))]}{{\zeta  }^{\ast \ast }[\Gn (\sigma , \tXn (\sigma ,\omega ))]}{\lhs (\Kmath ,\rzecz )} | \, d \sigma \\
\; &\le \; \int_{0}^{T} 
\norm{{\psi }^{\ast \ast }[\Gn (\sigma , \tXn (\sigma ,\omega ))]}{\lhs (\Kmath ,\rzecz )}{}
\norm{{\zeta }^{\ast \ast }[\Gn (\sigma , \tXn (\sigma ,\omega ))]}{\lhs (\Kmath ,\rzecz )}{}
\, d \sigma 
\end{split}
\]
By Remark \ref{rem:G_properties}(iii), we have for every $\sigma \in [0,T]$
\[
\begin{split}
&\norm{{\psi }^{\ast \ast }[\Gn (\sigma , \tXn (\sigma ))]}{\lhs (\Kmath ,\rzecz )}{}
\; = \; \sup_{y\in \Kmath , \norm{y}{\Kmath }{} \le 1 } 
|({\psi }^{\ast \ast }[\Gn (\sigma , \tXn (\sigma ))])(y)|  \\
\; &= \; \sup_{y\in \Kmath , \norm{y}{\Kmath }{} \le 1 } 
|\ddual{\Vmath '}{g (\sigma , \tXn (\sigma ))(y)}{\Pn \psi }{\Vmath }| 
\; \le \; {[C(1+ \nnorm{\tXn (\sigma )}{\Hmath }{2})]}^{\frac{1}{2}} \cdot \nnorm{\Pn }{\lcal (\Vmath ,\Vmath )}{} \, \norm{\psi }{\Vmath }{}.
\end{split}
\]
Thus
\[
\begin{split}
|{\xi }_{n}| 
\; &\le \; C \nnorm{\Pn }{\lcal (\Vmath ,\Vmath )}{2} \, \norm{\psi }{\Vmath }{} \, \norm{\zeta }{\Vmath }{}
\int_{0}^{T}  (1+ \nnorm{\tXn (\sigma )}{\Hmath }{2}) \, d \sigma 
\\
\; &\le \; C T \Bigl( \sup_{n\in \nat }\nnorm{\Pn }{\lcal (\Vmath ,\Vmath )}{2} \Bigr) \, \norm{\psi }{\Vmath }{} \, \norm{\zeta }{\Vmath }{}
 \, \Bigl( 1+ \sup_{\sigma \in [0,T]}\nnorm{\tXn (\sigma )}{\Hmath }{2} \Bigr)  ,
\end{split}
\]
and for $r \in (1,\frac{p}{2}]$
\[
\begin{split}
\te \bigl[ {|{\xi }_{n}|}^{r}  \bigr]
\; &\le \; c \Bigl( 1+\te\Bigl[ \sup_{\sigma \in [0,T]}\nnorm{\tXn (\sigma )}{\Hmath }{2r} \Bigr] \Bigr) 
\; \le \; c (1+ {C}_{2}(2r,p)) \; < \; \infty .
\end{split}
\]
(Here $c= {2}^{r-1} \bigl[ C T \bigl( \sup_{n\in \nat }\nnorm{\Pn }{\lcal (\Vmath ,\Vmath )}{2} \bigr) {\bigr] }^{r} \, \norm{\psi }{\Vmath }{r} \, \norm{\zeta }{\Vmath }{r} <\infty $, by Corollary \ref{cor:P_n-pointwise_conv}.) This completes the proof of the uniform integrability.

\medskip  \noindent
By the Vitali theorem
\[
\lim_{n \to \infty } \te [{\xi }_{n} ] \; = \; \te [\xi ] . 
\]
The proof of the lemma is thus complete.
\end{proof}

\medskip
\noindent
\bf Conclusion of the proof of Theorem \ref{th:mart-sol_existence}. \rm 
We use the idea analogous to the reasoning used by Da Prato and Zabczyk in \cite[Section 8.3]{DaPrato+Zabczyk'2014}.
Let us consider the operator $L$ defined by \eqref{eq:op_L} in Appendix \ref{sec:aux_funct.anal},
\[
L : \Umath \; \supset \; D(L) \; \to \; \Hmath 
\]
which is  self-adjoint and bijection. Let ${L}^{-1}$ be its inverse, i.e.
\[
{L}^{-1} : \Hmath \; \to \Umath .
\] 
By Lemmas \ref{lem:tMn-martingale}, \ref{lem:tM-continuity} , \ref{lem:qvar_conv_left} and \ref{lem:conv_quadr_var} (with $\psi := {L}^{-1}\varphi $ and $\zeta := {L}^{-1}\eta $, where $\varphi ,\eta \in \Hmath $), we infer that
the process ${(\tM )}_{t\in [0,T]}$ is a ${\Umath }^{\prime }$-valued continuous square integrable martingale with respect to the filtration $\tilde{\mathbb{F}} = \bigl( {\tfcal }_{t} \bigr) $,
 where  ${\tfcal }_{t}= \sigma \{ \tX (s), \, \, s \le t \} $.
Let us consider the dual operator of ${L}^{-1}$:
\[
({L}^{-1})' : {\Umath }^{'} \; \to \; {\Hmath }^{\prime }  \; \; \equiv \Hmath ,
\]
where the space $\Hmath '$ is identified with $\Hmath $.
Then the process ${(({L}^{-1})'\tM )}_{t\in [0,T]}$ is a ${\Hmath  }^{\prime } \cong \Hmath $-valued continuous square integrable martingale with respect to the filtration $\tilde{\mathbb{F}} = \bigl( {\tfcal }_{t} \bigr) $,
 where  ${\tfcal }_{t}= \sigma \{ \tX (s), \, \, s \le t \} $
with the quadratic variation
\[
\qvar{({L}^{-1})'\tM  }_{t} \; = \;  \int_{0}^{t}\norm{({L}^{-1})'\underline{\circ }\tilde{G}(s,\tX (s))}{\lhs (\Kmath ,\Hmath )}{2} \; ds ,
\]
where using (using the identification $\Hmath \equiv \Hmath '$) the injection $j:\Hmath ' \ni \xi \to {\xi }_{|\Umath } \in  \Umath '$
\[
\tilde{G}:  [0,T]\times \Hmath \; \ni (t,u) \; \mapsto \bigl\{ \Kmath \ni y \mapsto {j(G(t,u)(y))}_{} \bigr\} \;  \in \; \lcal  (\Kmath ,\Umath ') .
\]
and 
\[
({L}^{-1})'\underline{\circ }\tilde{G}:  [0,T]\times \Hmath \; \ni (t,u) \; \mapsto \bigl\{ \Kmath \ni y \mapsto ({L}^{-1})' [\tilde{G}(t,u)(y)] \bigr\} \;  \in \; \lhs (\Kmath ,\Hmath )
\]
(The fact that $({L}^{-1})'\underline{\circ }\tilde{G} \in \lhs (\Kmath ,\Hmath )$ follows from the fact that $G(t,u) \in \lhs (\Kmath ,\Hmath )$.)
By the martingale representation theorem, see \cite{DaPrato+Zabczyk'2014}, there exist
\begin{itemize}
\item a stochastic basis
$\bigl( \ttOmega , \ttfcal , {( \ttfcal  {}_{t} ) }_{t \ge 0} , \ttp  \bigr) $,
\item a cylindrical Wiener process $\ttW  (t)$ defined on this basis,
\item and a progressively measurable process $\ttX  (t)$
such that
\[
\begin{split}
&({L}^{-1})'\ttX (t) \,  -  ({L}^{-1})' \ttX (0) 
+ \int_{0}^{t}({L}^{-1})'  \bigl[  \acal \ttX (s)  + \MHD  (\ttX (s)) +\tHall (\ttX (s)) - f(s) \bigr] \, ds 
\\
\; &= \; \int_{0}^{t}({L}^{-1})' \underline{\circ }\tilde{G} (s,\ttX (s) \, d \ttW (s)
 , \qquad t \in [0,T].
\end{split}
\]
\end{itemize}
Thus for all $t \in [0,T]$ and all $v \in \Umath $
\[
\begin{split}
&\ilsk{\ttX (t)}{v}{\Hmath } \,  -  \ilsk{\ttX (0)}{v}{\Hmath } 
+ \int_{0}^{t} \dual{  \acal \ttX (s)  + \MHD  (\ttX (s)) +\tHall (\ttX (s)) - f(s) }{v}{} \, ds 
\\
\; &= \; \int_{0}^{t} \dual{ G (s,\ttX (s) \, d \ttW (s) }{v}{}
 , \qquad t \in [0,T].
\end{split}
\]
Since the space $\Umath $ is dense in $\Vtest $, we infer that the above equation holds for every $v \in \Vtest $. In conclusion, the system $(\bar{\mathfrak{A}},\bar{W}, \X)$, where 
$\bar{\mathfrak{A}} := (\ttOmega , \ttfcal , {( \ttfcal  {}_{t} )}_{t \ge 0} , \ttp ) $, $\bar{W}:= \ttW $ 
and $\X := \ttX $,
 is a martingale solution of the problem \eqref{eq:Hall-MHD_functional} in the sense of definition \ref{def:mart-sol}. The proof of theorem \ref{th:mart-sol_existence} is thus complete.

\appendix 

\medskip
\section{The spaces $\Ldn $ and the cut-off operators $\Sn $ }
\label{sec:Fourier_truncation}

\medskip  \noindent
We recall some results concerning the Friedrichs method which is based on the Fourier analysis, see \cite[Section 4, p.174]{Bahouri+Chemin+Danchin'11}. This presentation is also closely related to \cite{Brze+Dha'20} and \cite{Feff+McCorm+Rob+Rod'2014}.
  
\medskip
\subsection{Preliminaries}

\medskip  \noindent
Let us recall that the Fourier transform of a rapidly decreasing function $\psi \in \scal (\rd )$ is defined by (see \cite{Rudin}, \cite{Taylor_I'11})
\[
\widehat{\psi } (\xi ) \; := \; {(2\pi )}^{-\frac{d}{2}} \int_{\rd } {e}^{-i\xi \cdot x } \psi (x) \, dx , \qquad \xi \in \rd ,
\]
and the Fourier transform  of a tempered distribution is defined via duality, i.e., if $f \in \scal '(\rd )$ then
\[
\dual{\widehat{f}}{\psi }{} \; := \; \dual{f}{\widehat{\psi }}{}, \qquad \psi \in \scal (\rd ).
\]

\medskip  \noindent
Let us recall that for $s\ge 0 $ the Sobolev space is defined by
\[
 {H}^{s}(\rd ) \; := \; \{ u \in {L}^{2}(\rd ): \; \; {(1+{|\xi |}^{2})}^{\frac{s}{2}} \widehat{u} \in {L}^{2}(\rd ) \}  
\]
and
\[
\norm{u}{{H}^{s}}{} \; := \; \nnorm{{(1+{|\xi |}^{2})}^{\frac{s}{2}} \widehat{u}}{{L}^{2}(\rd )}{}
\; = \; \biggl( \int_{\rd } {(1+{|\xi |}^{2})}^{s} {|\widehat{u}(\xi )|}^{2} \, d \xi {\biggr) }^{\frac{1}{2}}.
\]
(See \cite{Taylor_I'11}, \cite{Stein'1970}.) The spaces ${H}^{s}(\rd )$ are also called Lebesgue spaces and denoted by ${L}_{s}^{2}$.

\medskip
\subsection{Subspaces $\Ldn $ and the cut-off operators $\Sn $.} \label{sec:cut-off_operators}

\medskip  \noindent
Let
\[
{\bar{B}}_{n} \; := \; \{ \xi \in \rd : \; \; |\xi | \le n \}  \; \subset \; {\rzecz }^{d}, \qquad n \in \nat , 
\]
and let
\begin{equation}
\Ldn  \; := \; \{ u \in {L}^{2}(\rd ) : \; \; \supp \widehat{u} \subset {\bar{B}}_{n} \}  .
\label{eq:Ld_n}
\end{equation}
On the subspace $\Ldn $ we consider the norm inherited from ${L}^{2}(\rd )$.

\medskip  \noindent
The \bf cut-off operator \rm  $\Sn $ is defined by
\begin{equation}
\Sn (u)  \; := \; {\fcal }^{-1} (\ind{{\bar{B}}_{n}} \widehat{u}) , \qquad u \in {L}^{2}(\rd ),
\label{eq:S_n}
\end{equation}
where ${\fcal }^{-1}$ denotes the inverse Fourier transform.
(See  \cite[Section 4, p.174]{Bahouri+Chemin+Danchin'11}.)

\medskip
\begin{remark} \label{rem:S_n-projection}
\rm (See \cite[Section 4, p.174]{Bahouri+Chemin+Danchin'11} and \cite{Brze+Dha'20}.)  \it 
The map
\[
\Sn  : {L}^{2}(\rd ) \; \to \; \Ldn  
\] 
is an  $\ilsk{\cdot }{\cdot }{{L}^{2}}$-orthogonal projection.
\end{remark}

\medskip
\subsection{Properties of the  operators $\Sn $.}

\medskip  \noindent
The following lemma is closely related to \cite[p. 1042]{Feff+McCorm+Rob+Rod'2014} and \cite[Lemma 4.1]{Brze+Dha'20}.

\medskip
\begin{lemma} \label{lem:S_n-pointwise_conv}
Let  $s\ge 0 $ be fixed. 
Then for all $n \in \nat $:
\[
\Sn  : {H}^{s}(\rd ) \; \to \; {H}^{s}(\rd )
\]
is well defined linear and bounded. Moreover, for every $ u \in {H}^{s}(\rd ) $
\begin{equation}
\norm{\Sn u}{{H}^{s}}{} \; \le \ \norm{u}{{H}^{s}}{} 
\label{eq:S_n:H^s-H^s}
\end{equation}
and
\begin{align}
\lim_{n \to \infty } \norm{\Sn u-u}{{H}^{s}}{}  \; = \; 0 . 
\label{eq:S_n_pointwise_conv}
\end{align} 
\end{lemma}

\medskip
\begin{proof}
For every $u \in {H}^{s}(\rd )$ we have
\begin{align*}
\norm{\Sn u}{{H}^{s}}{2} 
\; &= \;  \int_{\rd } {(1+{|\xi |}^{2})}^{s} {|\widehat{\Sn u}(\xi )|}^{2} \, d \xi
\; = \;  \int_{\rd } {(1+{|\xi |}^{2})}^{s} {|\ind{{\bar{B}}_{n}}(\xi ) \, \widehat{u}(\xi )|}^{2} \, d \xi \\
\; &= \;  \int_{\rd } {(1+{|\xi |}^{2})}^{s} \ind{{\bar{B}}_{n}}(\xi ) \, {| \widehat{u}(\xi )|}^{2} \, d \xi
\; = \;  \int_{{\bar{B}}_{n}} {(1+{|\xi |}^{2})}^{s} \, {| \widehat{u}(\xi )|}^{2} \, d \xi  \\
\; & \le \;  \int_{\rd } {(1+{|\xi |}^{2})}^{s} \, {| \widehat{u}(\xi )|}^{2} \, d \xi
\; = \; \norm{u}{{H}^{s}}{2} .
\end{align*}
which completes the proof of \eqref{eq:S_n:H^s-H^s}.

\medskip  \noindent
To prove \eqref{eq:S_n_pointwise_conv} we proceed as follows
\begin{align*}
\norm{\Sn u-u}{{H}^{s}}{2} 
\; & = \; \int_{\rd } {(1+{|\xi |}^{2})}^{s} {|\widehat{\Sn u}(\xi ) - \widehat{u}(\xi )|}^{2} \, d \xi
\; = \;  \int_{\rd } {(1+{|\xi |}^{2})}^{s} {|\ind{{\bar{B}}_{n}}(\xi ) \, \widehat{u}(\xi ) - \widehat{u}(\xi )|}^{2} \, d \xi \\
\ &= \  \int_{\rd } {(1+{|\xi |}^{2})}^{s} \bigl[ \ind{{\bball }_{n}}(\xi ) -1\bigr] \, {| \widehat{u}(\xi )|}^{2} \, d \xi
\; = \;  \int_{{\bar{B}}^{c}_{n}} {(1+{|\xi |}^{2})}^{s} \, {| \widehat{u}(\xi )|}^{2} \, d \xi .
\end{align*}
Since $u \in {H}^{s}(\rd )$, we infer that 
\[
\lim_{n\to \infty } \int_{{\bar{B}}^{c}_{n}} {(1+{|\xi |}^{2})}^{s} \, {| \widehat{u}(\xi )|}^{2} \, d \xi 
\; = \; 0,
\]
which completes the proof of \eqref{eq:S_n_pointwise_conv} and of the lemma.
\end{proof}

\medskip  \noindent
The following lemma is based on \cite[p. 1042]{Feff+McCorm+Rob+Rod'2014}.

\medskip
\begin{lemma} \label{lem:S_n-norm_conv}
If $s\ge 0 $ and $k>0 $, then 
\begin{equation*}
\Sn :{H}^{s+k}(\rd ) \to {H}^{s}(\rd )
\end{equation*} 
is well defined  and bounded and 
$
\nnorm{\Sn }{\lcal ({H}^{s+k},{H}^{s})}{} \; \le  \; 1 .
$
Moreover, for every $u \in {H}^{s+k}(\rd )$: 
\begin{equation}
\norm{\Sn u-u}{{H}^{s}}{2} \; \le \; \frac{1}{{(1+{n}^{2})}^{k}} \, \norm{u}{{H}^{s+k}}{2}.
\label{eq:S_n:H^s+k-H^s_est}
\end{equation}
Thus
\begin{equation}
\lim_{n\to \infty } \nnorm{\Sn -I}{\lcal ({H}^{s+k},{H}^{s})}{} \; = \; 0, 
\label{eq:S_n:H^s+k-H^s_norn-conv}
\end{equation}
where $I$ stands for the identity operator.  
\end{lemma}

\medskip
\begin{proof}
Let us fix $s \ge 0 $ and  $k>0$. Let $u \in {H}^{s+k}(\rd )$. 
We have
\begin{align*}
&\norm{\Sn u}{{H}^{s}}{2} 
\;  = \; \int_{\rd } {(1+{|\xi |}^{2})}^{s} {|\widehat{\Sn u}(\xi ) |}^{2} \, d \xi
\; = \;  \int_{\rd } {(1+{|\xi |}^{2})}^{s} {|\ind{{\bar{B}}_{n}}(\xi ) \, \widehat{u}(\xi ) |}^{2} \, d \xi \\
\; &= \;  \int_{{\bar{B}}_{n}} {(1+{|\xi |}^{2})}^{s} \, {| \widehat{u}(\xi )|}^{2} \, d \xi 
\; = \;  \int_{{\bar{B}}_{n}} \frac{1}{{(1+{|\xi |}^{2})}^{k} } \cdot {(1+{|\xi |}^{2})}^{s+k} \, {| \widehat{u}(\xi )|}^{2} \, d \xi \\
 \; &\le \;  \int_{{\bar{B}}_{n}} {(1+{|\xi |}^{2})}^{s+k} \, {| \widehat{u}(\xi )|}^{2} \, d \xi  
\; \le \;  \int_{\rd } {(1+{|\xi |}^{2})}^{s+k} \, {| \widehat{u}(\xi )|}^{2} \, d \xi
\; = \;  \norm{u}{{H}^{s+k}}{2}.
\end{align*}
Thus   $\Sn \in \lcal ({H}^{s+k},{H}^{s}) $ and
$
\nnorm{\Sn }{\lcal ({H}^{s+k},{H}^{s})}{} \; \le  \; 1 .
$

\medskip  \noindent
Moreover,
\begin{align*}
&\norm{\Sn u-u}{{H}^{s}}{2} 
\;  = \; \int_{\rd } {(1+{|\xi |}^{2})}^{s} {|\widehat{\Sn u}(\xi ) - \widehat{u}(\xi )|}^{2} \, d \xi
\; = \;  \int_{\rd } {(1+{|\xi |}^{2})}^{s} {|\ind{{\bar{B}}_{n}}(\xi ) \, \widehat{u}(\xi ) - \widehat{u}(\xi )|}^{2} \, d \xi \\
\; &= \;  \int_{\rd } {(1+{|\xi |}^{2})}^{s} \bigl[ \ind{{\bar{B}}_{n}}(\xi ) -1\bigr] \, {| \widehat{u}(\xi )|}^{2} \, d \xi
\; = \;  \int_{{\bar{B}}^{c}_{n}} {(1+{|\xi |}^{2})}^{s} \, {| \widehat{u}(\xi )|}^{2} \, d \xi \\
\; &= \;  \int_{{\bar{B}}^{c}_{n}} \frac{1}{{(1+{|\xi |}^{2})}^{k} } \cdot {(1+{|\xi |}^{2})}^{s+k} \, {| \widehat{u}(\xi )|}^{2} \, d \xi 
 \; \le \; \frac{1}{{(1+{n}^{2})}^{k} } \int_{{\bar{B}}^{c}_{n}} {(1+{|\xi |}^{2})}^{s+k} \, {| \widehat{u}(\xi )|}^{2} \, d \xi  \\
\; &\le \; \frac{1}{{(1+{n}^{2})}^{k} } \cdot \int_{\rd } {(1+{|\xi |}^{2})}^{s+k} \, {| \widehat{u}(\xi )|}^{2} \, d \xi
\; = \; \frac{1}{{(1+{n}^{2})}^{k} } \cdot \norm{u}{{H}^{s+k}}{2}, 
\end{align*}
which completes \eqref{eq:S_n:H^s+k-H^s_est}.
Thus
\[
\nnorm{\Sn -I}{\lcal ({H}^{s+k},{H}^{s})}{2} \; \le \; \frac{1}{{(1+{n}^{2})}^{k}} .
\]
which implies that \eqref{eq:S_n:H^s+k-H^s_norn-conv} holds.
The proof is thus complete.
\end{proof}

\medskip
\subsection{Relation between the spaces $\Ldn $ and  ${H}^{s}(\rd )$ for $s\ge 0$.}

\medskip  \noindent
Let us recall that on the spaces $\Ldn $, by definition, we consider the norms inherited from the space ${L}^{2}(\rd )$, see \eqref{eq:Ld_n}.

\medskip
\begin{lemma}  \label{lem:Ld_n-H^s-relation}
For each $n \in \nat $
\begin{equation*}
\Ldn \; {\hookrightarrow } \; {H}^{s}(\rd )  \quad \mbox{ for all } \, s \ge 0  
\end{equation*}
and  for every $ s \ge 0 $ and  $ u \in \Ldn :$
\begin{equation}
\norm{u}{{H}^{s}}{2} \; \le \;  {(1+{n}^{2})}^{s} \, \nnorm{u}{\Ldn }{2} .  
\label{eq:Ld_n-H^s-ineq} 
\end{equation}
\end{lemma}

\medskip  \noindent
Note that  the norm of the embedding $\Ldn \subset {H}^{s}(\rd )$ depends on $n$ and $s$.

\medskip
\begin{proof}
Let $u \in \Ldn $. Then $\Sn u=u$. Let us fix $s >0 $. We will show that $u \in {H}^{s}(\rd )$.
Indeed,
\begin{align*}
\int_{\rd } {(1+{|\xi |}^{2})}^{s} \, {|\widehat{u}(\xi )|}^{2} \, d\xi 
\; &= \; \int_{\rd } {(1+{|\xi |}^{2})}^{s} \, {|\widehat{\Sn u}(\xi )|}^{2} \, d\xi 
\; = \;  \int_{\rd } {(1+{|\xi |}^{2})}^{s} \, \ind{{\bar{B}}_{n}}(\xi ) \,   {|\widehat{u}(\xi )|}^{2} \, d\xi  \\
\; & = \;  \int_{{\bar{B}}_{n} } {(1+{|\xi |}^{2})}^{s} \,    {|\widehat{u}(\xi )|}^{2} \, d\xi 
\; \le  \; {(1+{n}^{2})}^{s} \, \int_{{\bar{B}}_{n} }     {|\widehat{u}(\xi )|}^{2} \, d\xi \\
\; &\le  \; {(1+{n}^{2})}^{s} \, \int_{\rd }     {|\widehat{u}(\xi )|}^{2} \, d\xi
\; = \; {(1+{n}^{2})}^{s} \, \nnorm{u}{{L}^{2}}{2} \; = \; {(1+{n}^{2})}^{s} \, \nnorm{u}{\Ldn }{2}.
\end{align*}
which complete the proof of \eqref{eq:Ld_n-H^s-ineq}  and of the lemma.
\end{proof}

\medskip
\begin{cor}  \label{cor:Ld_n-H^s-norm_equiv}
On the subspace $\Ldn $ the norms $\nnorm{\cdot}{\Ldn }{}$ and $\norm{\cdot }{H^s}{}$, for $s>0$, are equivalent (with appropriate constants depending on $s$ and $n$).   
\end{cor}

\medskip
\begin{proof}
The assertion is a consequence of the following inequalities: for all $u\in \Ldn $
\[
\nnorm{u}{\Ldn }{2} \; \le \; \norm{u}{{H}^{s}}{2} \; \le \;  {(1+{n}^{2})}^{s} \, \nnorm{u}{\Ldn }{2} .
\]
\end{proof}

\medskip
\section{Auxiliary results from functional analysis: the space $\Umath $ and the operator $L$} 
\label{sec:aux_funct.anal}

\medskip  \noindent
We have the following spaces, defined in Section \ref{sec:Hall-MHD_funct-setting}, which appear in the functional setting of problem \eqref{eq:Hall-MHD_u}-\eqref{eq:Hall-MHD_ini-cond}
\[
\Vtest \; \subset \; \Vmath \; \subset \; \Hmath  . 
\]
Recall that $\Vtest $, see \eqref{eq:Vmath_m1,m2}, is the space of test functions used in Definition \ref{def:mart-sol}.
For fixed $m >  \frac{5}{2}$, 
let us consider the space
\begin{equation}
\Vast \; := \; {\Vmath }_{m}, 
\label{eq:Vast}
\end{equation}
where ${\Vmath }_{m} = {V}_{m} \times {V}_{m} $ is defined by \eqref{eq:Vmath_m}.
The choice of the space $\Vast $ corresponds to the properties on nonlinear maps $\MHD $ and $\tHall $, 
see Lemmas \ref{lem:MHD-term_properties} and \ref{lem:tHall-term_properties} and Corollaries  \ref{cor:MHD-map_conv-aux} and \ref{cor:tHall-term_conv_general} in   Section \ref{sec:Hall-MHD_funct-setting}.
In fact, instead of \eqref{eq:Vast} the space $\Vast $ can be defined as $ {\Vmath }_{m_1,m_2}$ for fixed $m_1,m_2 >  \frac{5}{2}$, defined by \eqref{eq:Vmath_m1,m2}.
However, for us it will be sufficient use the space $\Vast $ defined by \eqref{eq:Vast}.

\medskip  \noindent
\bf Space $\Umath $. \rm 
Since  the embeddings of  Sobolev space are not compact in  the case of an unbounded domain,  
we introduce some auxiliary space $\Umath $ which will be of crucial importance in the compactness in tightness results.

\medskip  \noindent
Since $\Vast $ is dense in $\Hmath $ and the embedding $\Vast \hookrightarrow \Hmath $ is continuous, by Lemma 2.5 from \cite{Holly+Wiciak'1995} (see \cite[Lemma C.1]{Brze+EM'13})
there exists a separable Hilbert space $\Umath $  such that 
$\Umath \subset \Vast $, $\Umath $ is dense in $\Vast $ and  
\begin{equation}
\mbox{the embedding  ${\iota }_{} : \Umath \hookrightarrow \Vast   $ is compact.}
\label{eq:Umath}
\end{equation}
Then we have
\begin{equation}
\Umath \; \hookrightarrow  \; \Vast \; \; \hookrightarrow  \; \Vtest \;
\hookrightarrow \; \Vmath \; \hookrightarrow \; \Hmath .
\label{eq:spaces}
\end{equation}

\medskip  \noindent
\bf Operator L. \rm 
We define some auxiliary operator which will be used in the proof of the existence of martingale solutions.
By \eqref{eq:Umath} and \eqref{eq:spaces} we infer that, 
in particular, $\Umath $ is compactly embedded into the space $\Hmath $.
Let us  denote 
\[
\iota  : \; \Umath \; \hookrightarrow \; \Hmath 
\]
and let 
\[
{\iota }^{*}  :  \Hmath \; \to \;  \Umath  .
\]
be its adjoint operator.
Since the range of $\iota $ is dense in $\Hmath $, the map ${\iota }^{*} : \Hmath  \to \Umath  $ is one-to-one. Let us put
\begin{equation}
\begin{split}
D(L) \; &:= \; {\iota }^{*}(\Hmath ) \subset \Umath ,  \\
Lu \; &:= \;  \bigl( {\iota }^{*} {\bigr) }^{-1} u , \qquad u \in D(L) . 
\end{split} \label{eq:op_L}
\end{equation}
It is clear that $L:D(L) \to \Hmath  $ is onto. Let us also notice that
\begin{equation} \label{eq:op_L_ilsk}
\ilsk{Lu}{w}{\Hmath } \; = \; \ilsk{u}{w}{\Umath }, \qquad u \in D(L), \quad w \in \Umath .
\end{equation}
Indeed, by \eqref{eq:op_L} we have for all $u\in D(L)$ and $w\in \Umath $
\[
\ilsk{Lu}{w}{\Hmath } \; = \; \ilsk{{({\iota }^{\ast })}^{-1}u}{\iota w}{\Hmath }
\; = \; \ilsk{ {\iota }^{\ast} {({\iota }^{\ast })}^{-1}u}{ w}{\Umath }
\; = \; \ilsk{u}{w}{\Umath },
\]
which proves \eqref{eq:op_L_ilsk}.
By equality \eqref{eq:op_L_ilsk} and the density of $\Umath $ in $\Hmath $, we infer that $D(L)$ is dense in $\Hmath $.

\medskip
\section{Appendix: Compactness and tightness results. The Skorokhod theorem} \label{sec:comp-tight}

\medskip  \noindent
In this section we present some compactness and tightness results being a straightforward adaptations to our framework  of the results proved in \cite{Brze+EM'13} and \cite{EM'14}. We use the spaces $\Hmath $ and $\Vmath $ which appear in the statement of problem \eqref{eq:Hall-MHD_functional} (see Definition \ref{def:mart-sol}) as well as the auxiliary space $\Umath $ constructed in Appendix \ref{sec:aux_funct.anal}.
By \eqref{eq:Umath} and  \eqref{eq:spaces}, in particular,
we have
\begin{equation*}
\Umath  \; \hookrightarrow \; \Vmath \; \hookrightarrow \; \Hmath 
\; \cong \; {\Hmath }^{\prime } \; \hookrightarrow \; {\Umath }^{\prime },
\end{equation*}
the embedding  $ \Umath \hookrightarrow \Vmath   $ being compact.

\medskip  \noindent
Let us consider the following functional spaces being the counterparts in our framework of the spaces used in \cite{Brze+EM'13} and \cite{EM'14}, see also  \cite{Metivier'88}:
\begin{itemize}
\item $\ccal ([0,T],{\Umath }^{\prime }) $ := the space of  continuous functions 
 $ \phi :[0,T] \to {\Umath }^{\prime } $  with the topology  $ {\tcal }_{1}$ induced by norm 
${|\phi |}_{\ccal ([0,T];{\Umath }^{\prime })} := \sup_{t\in [0,T]} \nnorm{\phi (t)}{{\Umath }^{\prime }}{}$,
\item ${L}_{w}^{2}(0,T;\Vmath ) $ := the space ${L}^{2} (0,T;\Vmath )$ with the weak topology 
                     $ {\tcal}_{2} $,      
\item ${L}^{2}(0,T;{\Hmath }_{loc})$ := the space of measurable functions 
 $ \phi :[0,T] \to \Hmath   $ such that for all $ R \in \nat $
\begin{equation*}                 
{p}_{T,R}(\phi):= \Bigl(  \int_{0}^{T} \int_{{\ocal }_{R}} \bigl[ {|{\phi }_{1} (t,x)|}^{2} +  {|{\phi }_{2} (t,x)|}^{2} \bigr] \, dxdt {\Bigr) }^{\frac{1}{2}}
<\infty , 
\end{equation*}   
where $\phi  = ({\phi }_{1},{\phi }_{2})$,   
with the topology  $ {\tcal }_{3}$  generated by the seminorms 
$({p}_{T,R}{)}_{R\in \nat } .$                     
\end{itemize}

\medskip  \noindent
Let ${\Hmath }_{w}$ denote the Hilbert space $\Hmath $ endowed with the weak topology. 
Let us consider the fourth space, see \cite{Brze+EM'13},
\begin{itemize} 
\item $\ccal ([0,T];{\Hmath }_{w}) $ : = the space of weakly continuous functions  
$ \phi : [0,T] \to \Hmath $  with the weakest topology ${\tcal }_{4}$ such that for all 
                          $h \in \Hmath  $   the  mappings 
\begin{equation*}
\ccal ([0,T];{\Hmath }_{w}) \ni \phi   \mapsto \ilsk{\phi (\cdot )}{h}{\Hmath } \in \ccal  ([0,T];\rzecz )  
\end{equation*} 
are continuous.                     
In particular,  
${\phi }_{n} \to \phi  $ in $\ccal ([0,T];{\Hmath }_{w}) $ iff  for all $ h \in \Hmath  $:
$
\ilsk{{\phi }_{n} (\cdot )}{h}{\Hmath }  \to \ilsk{\phi (\cdot )}{h}{\Hmath }   \mbox{ in the space }  
\ccal  ([0,T];\rzecz ).
$ 
\end{itemize}

\medskip
\subsection{Compactness and tightness criteria} 

\medskip
\begin{definition}\label{def:space_Z_T}
\rm Let us put
\begin{equation}
\mathcal{Z} \; := \; {L}_{w}^{2}(0,T;\Vmath )  
\cap {L}^{2}(0,T;{\Hmath }_{loc})  \cap \ccal ([0,T];{\Hmath }_{w}) \cap \ccal ([0,T]; {\Umath }^{\prime }) 
\label{eq:Z_T}
\end{equation}
and let  $\mathcal{T} $ be   the supremum of the corresponding four topologies, i.e. the smallest topology on $\mathcal{Z}$ such that the four natural embeddings from $\mathcal{Z}$ are continuous.
The space  $\mathcal{Z}$  will  also considered with the Borel $\sigma $-algebra, denoted by $\sigma (\mathcal{Z})$, i.e. the smallest $\sigma $-algebra containing the family $\mathcal{T} $.
\end{definition}

\medskip  \noindent
The following compactness criterion is a simple modification of Lemma 4.1 from  \cite{EM'14}. See also \cite[Section 3.1]{Brze+EM'13} for the case of the Navier-Stokes equations.

\medskip  
\noindent
\begin{lemma}  \label{lem:Dubinsky_unbound_cont}  
\rm (see \cite[Lemma 4.1]{EM'14}) \it
Let
\[
\kcal \subset {L}^{\infty }(0,T;\Hmath ) \cap {L}^{2}(0,T;\Vmath ) \cap \ccal ([0,T];{\Umath }^{\prime })
\]
satisfy the following three conditions 
\begin{description}
\item[(a) ]  
$\, \, \sup_{\phi \in \kcal } \norm{\phi }{{L}^{\infty }(0,T;\Hmath )}{} < \infty  $,
i.e. $\kcal $ is bounded in ${L}^{\infty }(0,T;\Hmath )$, 
\item[(b) ] $ \sup_{\phi \in \kcal } \int_{0}^{T} \norm{\phi (s)}{\Vmath }{2} \, ds < \infty  $,
  i.e. $\kcal $ is bounded in ${L}^{2}(0,T;\Vmath )$,
\item[(c) ] $\lim{}_{\delta \to 0 } \sup_{\phi \in \kcal } \sup_{|t-s|\le \delta } \norm{\phi (t)-\phi (s)}{{\Umath }^{\prime }}{} =0 $.
\end{description}
Then $\kcal \subset \zcal $  and $\kcal $ is $\tcal $-relatively compact in $\zcal $ defined by \eqref{eq:Z_T}.
\end{lemma}

\medskip  \noindent 
One of the main tools in the construction of a martingale solution is the tightness criterion in the space $\mathcal{Z}$ defined in identity  \eqref{eq:Z_T}.
We  use a slight  modification of  the criterion stated in \cite[Corollary 3.9]{Brze+EM'13},  in the framework of the Navier-Stokes equation and  in \cite[Corollary 4.2]{EM'14}, where more general setting is considered.

\medskip
\begin{cor} \label{cor:tigthness_criterion} 
\it Let $(\Xn {)}_{n \in \nat }$ be a sequence of continuous $\mathbb{F}$-adapted 
${\Umath }^{\prime }$-valued processes such that
\begin{description}
\item[(a)] there exists a positive constant ${C}_{1}$ such that
\begin{equation}
\sup_{n\in \nat} \e \bigl[ \sup_{s \in [0,T]} \nnorm{\Xn (s) }{\Hmath }{}  \bigr] \; \le \; {C}_{1} ,
\label{cond-a}
\end{equation}
\item[(b)] there exists a positive constant ${C}_{2}$ such that
\begin{equation}
\sup_{n\in \nat} \e \Bigl[  \int_{0}^{T} \norm{\Xn (s)}{\Vmath }{2} \, ds    \Bigr] \; \le \; {C}_{2} ,
\label{cond-b}
\end{equation}
\item[(c)]  $(\Xn {)}_{n \in \nat }$ satisfies the Aldous condition  in ${\Umath }^{\prime }$.
\end{description}
Let ${\tilde{\p }}_{n}$ be the law of $\Xn $ on ${\zcal }_{}$.
Then for every $\eps >0 $ there exists a compact subset ${K}_{\eps }$ of ${\zcal }_{}$ such that
\[
{\tilde{\p }}_{n} ({K}_{\eps }) \; \ge \; 1-\eps .
\]
\end{cor}

\medskip  \noindent
The proof of Corollary \ref{cor:tigthness_criterion} is essentially the same as the proof of
\cite[Corollary 3.9]{Brze+EM'13}  or \cite[Corollary 4.2]{EM'14}.

\medskip  \noindent 
Let us recall the Aldous condition in the form given by  M\'{e}tivier  \cite{Metivier'88}.

\medskip  
\begin{definition} (M. M\'{e}tivier) \label{def:Aldous}
\rm A sequence $({X}_{n}{)}_{n\in \nat }$  satisfies the \bf  Aldous condition \rm
in the space ${\Umath }^{\prime }$
iff
\begin{description}
\item[\mbox{\bf [A]\rm }] for every  $\eps >0 $ and $ \eta >0 $ there exists  $ \delta >0 $ such that for every sequence $({{\tau}_{n} } {)}_{n \in \nat }$ of $\mathbb{F}$-stopping times with
${\tau }_{n}\le T$ one has
\[
\sup_{n \in \nat} \, \sup_{0 \le \theta \le \delta }  \p \bigl\{
{| {X}_{n} ({\tau }_{n} +\theta )-{X}_{n} ( {\tau }_{n}  ) |}_{{\Umath }^{\prime }} \ge \eta \bigr\}  
\; \le \;  \eps .
\]
\end{description}
\end{definition}

\medskip  \noindent
Below we recall a sufficient condition for the Aldous condition.

\medskip  
\begin{lemma} \label{lem:Aldous_criterion} 
\rm (See \cite[Lemma 9]{EM'13}) \it
Let $(E,\norm{\cdot }{E}{})$ be a separable Banach space and let $(\Xn {)}_{n \in \nat }$ be a sequence of $E$-valued random variables such that 
\begin{itemize} 
\item[\mbox{\bf [A']\rm } ] there exist $\alpha ,\beta >0 $ and $C>0$ such that
for every sequence $({{\tau}_{n} } {)}_{n \in \nat }$ of $\mathbb{F}$-stopping times with
${\tau }_{n}\le T$ and for every $n \in \nat $ and $\theta \ge 0 $ the following condition holds
\begin{equation} \label{eq:Aldous_est}
\e \bigl[ \bigl( \norm{ \Xn ({\tau }_{n} +\theta )- \Xn  ( {\tau }_{n}  ) }{E}{\alpha } \bigr] 
\; \le \;  C {\theta }^{\beta } .
\end{equation}
\end{itemize}
Then the sequence $(\Xn {)}_{n\in \nat }$  satisfies  condition \rm \bf [A] \rm in the space $E$.  \rm
\end{lemma}

\medskip
\subsection{Jakubowski's generalization of the Skorokhod theorem} 

\medskip  \noindent
In the proof of the theorem on the existence of a martingale solution we use a version of the Skorokhod theorem for nonmetric spaces. For convenience of the reader let us recall  the following Jakubowski's \cite{Jakubowski'98} generalization of the Skorokhod theorem.

\medskip
\begin{theorem} \label{th:2_Jakubowski} \rm (Theorem 2 in \cite{Jakubowski'98}). \it
Let $(\mathcal{X} , \tau )$ be a topological space such that there exists a sequence $({f}_{m}) $ of continuous functions ${f}_{m}:\mathcal{X}  \to \rzecz $ that separates points of $\mathcal{X} $.
Let $({X}_{n} )$ be a sequence of  $\mathcal{X} $-valued Borel random variables. Suppose that for every $\eps >0$
there exists a compact subset ${K}_{\eps} \subset \xcal $ such that
\[
\sup_{n \in \nat } \p (\{ {X}_{n} \in {K}_{\eps } \} ) \; > \;  1-\eps .
\]
Then there exists a subsequence $({X}_{{n}_{k}}{)}_{k\in \nat }$, a sequence $({Y}_{k}{)}_{k\in \nat }$ of $\mathcal{X} $-valued Borel random variables and an $\mathcal{X} $-valued Borel random variable $Y$ defined on some probability space  $(\Omega , \fcal ,\p )$ such that
\[
\lcal ({X}_{{n}_{k}}) \; = \; \lcal ({Y}_{k}), \qquad k=1,2,...
\]
and for all $ \omega \in \Omega $:
\[
{Y}_{k}(\omega ) \; \stackrel{\tau }{\longrightarrow } \; Y(\omega )  \quad \mbox{ as } \; k \to \infty .
\]
\end{theorem}

\medskip  \noindent
Proceeding analogously to \cite[Corollary 3.12]{Brze+EM'13} or \cite[Remark C.2]{EM'14} it is easy to prove the following result for the space $\mathcal{Z}$ from Definition \ref{def:space_Z_T}.

\medskip
\begin{lemma}\label{lem-Z_T} 
The topological space $\mathcal{Z}$  satisfies the assumptions of Theorem \ref{th:2_Jakubowski}.
\end{lemma}

\medskip
\section{Proof of Lemma \ref{lem:Hall-term_conv_general}} \label{app:Hall-term_conv_general_proof}

\medskip
\begin{proof}
\bf Step ${1}^{0}$. \rm Assume first that $\psi \in \vcal $. There exists $R>0 $ such that $\supp \psi $ is a compact subset of the ball $K_R:= \{ x \in {\rzecz }^{3} : \; |x| < R \} $.
We have
\[
\begin{split}
&\int_{0}^{t} [\dual{\Hall (\un (s),\wn (s))}{\psi }{}  - \dual{\Hall (u(s),w(s))}{\psi }{} ] \, ds  \\
\; &= \; \int_{0}^{t}\dual{\Hall (\un (s) -u (s) ,\wn (s))}{\psi }{} \, ds  
+ \int_{0}^{t}\dual{\Hall (u(s), \wn (s)- w(s))}{\psi }{} \, ds \; =: \; I_1(n) + I_2(n).
\end{split}
\]
We will analyze separately the terms $I_1(n)$ and $I_2(n)$. 

\medskip  \noindent
Let us consider the term $I_1(n)$. 
By \eqref{eq:hall-form}, \eqref{eq:Hall-map_V-V},
 the H\"{o}lder inequality and the Sobolev embedding theorem we obtain
\[
\begin{split}
&|\dual{\Hall (u,w)}{\psi }{}| \; = \; 
|\hall (u,w,\psi )| \; = \; \Bigl| \int_{K_R} [u \times (\curl w)] \cdot \curl \psi \, dx   \Bigr|
\; \le \; \nnorm{u}{L^2(K_R)}{} \cdot \nnorm{\curl w}{L^2}{} \cdot \nnorm{\curl \psi }{L^{\infty }}{} \\
\; &\le \tilde{c} \, \nnorm{u}{L^2(K_R)}{} \cdot \norm{w}{H^1}{} \cdot \norm{\curl \psi }{H^{m -1}}{} 
\; \le c \, \nnorm{u}{L^2(K_R)}{} \cdot \norm{w}{H^1}{} \cdot \norm{\psi }{H^{m }}{}.
\end{split}
\]  
for every $m> \frac{5}{2}$.
Thus by the Schwarz inequality for all $t \in 0,T]$
\[
\begin{split}
|I_1(n)| \; &\le \; c \, \int_{0}^{t} \nnorm{\un (s) - u(s)}{L^2(K_R)}{} \cdot \norm{\wn }{H^1}{} \, ds  \cdot \norm{\psi }{H^{m }}{} \\
\; & \le \; c\, \norm{\un -u}{{L}^{2}(0,T;L^2(K_R))}{} \cdot \norm{\wn }{{L}^{2}(0,T;V)}{}
 \cdot \norm{\psi }{H^{m }}{} \\
\; &= \; c\, {p}_{T,R} (\un -u ) \cdot \norm{\wn }{{L}^{2}(0,T;V)}{} \cdot \norm{\psi }{H^{m }}{} .
\end{split}
\]
Recall that ${p}_{T,R}$ denotes  the seminorm defined by \eqref{eq:seminorms-L^2(0,T;H_loc)}. 
Since $\un \to u $ in ${L}^{2}(0,T; {H}_{loc})$ and the sequence $(\wn )$ is bounded in ${L}^{2}(0,T;V)$, we infer that
\[
\lim_{n\to \infty } I_1(n) \; = \; 0.
\] 
\noindent
Let us move to the term $I_2(n)$.  Note that 
\[
\dual{\Hall (u,\wn -w)}{\psi }{}
\; = \; \hall (u, \wn - w , \psi ) \; = \; \int_{{\rzecz }^{3}} [u \times \curl (\wn -w )] \cdot \curl \psi \, dx.
\]
On the other hand, the map 
\[
\xi : L^2(0,T;L^2({\rzecz }^{3},{\rzecz }^{3})) \ni z \; \mapsto \, \int_{{\rzecz }^{3}} [u \times z] \cdot \curl \psi \, dx  \in \rzecz 
\]
is continuous linear functional on the space  $L^2(0,T;L^2({\rzecz }^{3},{\rzecz }^{3}))$.
 Since $\wn \to w $ weakly in ${L}^{2}(0,T;V)$, we infer that $\frac{\partial \wn }{\partial x_i} \to \frac{\partial w}{\partial x_i}$ weakly in ${L}^{2}(0,T;{L}^{2}({\rzecz }^{3},{\rzecz }^{3}))$  for each $i=1,2,3$. Thus  $\curl (\wn -w ) \to 0 $  weakly in ${L}^{2}(0,T;{L}^{2}({\rzecz }^{3},{\rzecz }^{3}))$.
 In conclusion,
\[
\lim_{n\to \infty } I_2(n) \; = \; \lim_{n\to \infty } \xi [\curl (\wn -w )] \; = \; 0 .
\]
\bf Step ${2}^{0}$. \rm If $\psi \in \Hsol{m} $, then for every $\eps >0 $ there exists ${\psi }_{\eps } \in \vcal $ such that $\norm{\psi - {\psi }_{\eps }}{{H}^{m }}{} < \eps $.
We have
\[
\begin{split}
&\dual{\Hall (\un ,\wn )-\Hall (u,w)}{\psi }{}   \\
\; &= \; \dual{\Hall (\un ,\wn )-\Hall (u,w)}{\psi -{\psi }_{\eps }}{} 
+ \dual{\Hall (\un ,\wn )-\Hall (u,w)}{{\psi }_{\eps }}{} .
\end{split}
\]
Thus, by \eqref{eq:Hall-map_est-H-V}
\[
\begin{split}
& \Bigl| \int_{0}^{t} \dual{\Hall (\un (s),\wn (s))-\Hall (u(s),w(s))}{\psi }{} \, ds \Bigr| \\
\; &\le \; \int_{0}^{t} |\dual{\Hall (\un (s),\wn (s))-\Hall (u(s),w(s))}{\psi -{\psi }_{\eps }}{} | \, ds 
\\
& \qquad + \Bigl| \int_{0}^{t} \dual{\Hall (\un (s),\wn (s))-\Hall (u(s),w(s))}{{\psi }_{\eps }}{} \, ds \Bigr|
\\
\; &\le \; \int_{0}^{t}\Bigl[ 
\nnorm{\Hall (\un (s),\wn (s))}{\Hsolprime{m}}{} 
+ \nnorm{\Hall (u(s),w(s))}{\Hsolprime{m}}{}  \Bigr] \cdot \norm{\psi - {\psi }_{\eps }}{{H}^{m }}{} \Bigr] \, ds \\
& \qquad + \Bigl| \int_{0}^{t} \dual{\Hall (\un (s),\wn (s))-\Hall (u(s),w(s))}{{\psi }_{\eps }}{} \, ds \Bigr|  \\
\; & \le \;  \eps \, c \, \int_{0}^{t}\Bigl[ \nnorm{\un (s)}{L^2}{} \norm{\wn (s)}{H^1}{} +  \nnorm{u(s)}{L^2}{} \norm{w(s)}{H^1}{}  \Bigr] \, ds  \\
& \qquad + \Bigl| \int_{0}^{t} \dual{\Hall (\un (s),\wn (s))-\Hall (u(s),w(s))}{{\psi }_{\eps }}{} \, ds \Bigr| 
\\
\; & \le \;  \frac{\eps \, c }{2}\, \int_{0}^{t}\Bigl[ \nnorm{\un (s)}{L^2}{2} + \norm{\wn (s)}{H^1}{2} +  \nnorm{u(s)}{L^2}{2}+  \norm{w(s)}{H^1}{2}  \Bigr] \, ds  \\
& \qquad + \Bigl| \int_{0}^{t} \dual{\Hall (\un (s),\wn (s))-\Hall (u(s),w(s))}{{\psi }_{\eps }}{} \, ds \Bigr|  
\\
\; & \le \;  \frac{\eps \, c }{2}\, \Bigl[ \sup_{n\in \nat } \Bigl( \norm{\un }{L^{2}(0,T;H)}{2} + \norm{\wn }{L^2(0,T;V)}{2} \Bigr)  +   \norm{u}{L^{2}(0,T;H)}{2}+ \norm{w}{{L}^{2}(0,T;V)}{2}  \Bigr]  \\
& \qquad + \Bigl| \int_{0}^{t} \dual{\Hall (\un (s),\wn (s))-\Hall (u(s),w(s))}{{\psi }_{\eps }}{} \, ds \Bigr| .
\end{split}
\]  
Passing to the upper limit as $n \to \infty $ and using step ${1}^{0}$ we obtain
\[
\limsup_{n\to \infty } \Bigl| \int_{0}^{t} \dual{\Hall (\un (s),\wn (s))-\Hall (u(s),w(s))}{\psi }{} \, ds \Bigr| \; \le \; M \eps ,
\]
where $M=\frac{c}{2}\, \bigl[ \sup_{n\in \nat } \bigl( \norm{\un }{L^{2}(0,T;H)}{2} + \norm{\wn }{L^2(0,T;V)}{2} \bigr)  +   \norm{u}{L^{2}(0,T;H)}{2}+ \norm{w}{{L}^{2}(0,T;V)}{2}  \bigr] < \infty $.
Since $\eps >0 $ is arbitrary, we infer that the assertion holds for all 
$\psi \in \Hsol{m} $. The proof of Lemma \ref{lem:Hall-term_conv_general} is thus complete.
\end{proof}

\end{document}